\documentclass[fleqn,reqno,11pt,a4paper,final]{amsart}

\usepackage[a4paper,left=20mm,right=20mm,top=20mm,bottom=20mm,marginpar=20mm]{geometry}
\usepackage{amsmath}
\usepackage{amssymb}
\usepackage{amsthm}
\usepackage{amscd}
\usepackage[utf8]{inputenc}
\usepackage{bbm}
\usepackage{color}
\usepackage[english=american]{csquotes}
\usepackage[final]{graphicx}
\usepackage{subcaption}
\usepackage{soul}
\usepackage{hyperref}
\usepackage{calc}

\usepackage{mathptmx}

\usepackage{mathtools}
\usepackage{tikz}
\usepackage{graphicx}
\usepackage{float}
\usepackage{enumitem}
\usepackage{multicol}
\usepackage{esint}
\usepackage{pgfplots}
\pgfplotsset{compat=1.18}

\usetikzlibrary{angles}

\DeclareSymbolFont{cmcal}{OMS}{cmsy}{m}{n}
\DeclareSymbolFontAlphabet{\mathcal}{cmcal}

\usepackage[backend=biber,bibencoding=utf8,giveninits,maxnames=20,style=alphabetic,isbn=false, doi=false, eprint= true, url= false,date=year]{biblatex}
\definecolor{luh-dark-blue}{rgb}{0.0, 0.313, 0.608}
\definecolor{luh-light-blue}{rgb}{0.6, 0.725, 0.847}
\definecolor{luh-green}{rgb}{0.784, 0.827, 0.09}

\graphicspath{{../Pictures/}}

\numberwithin{equation}{section}

\newtheoremstyle{thmlemcorr}{10pt}{10pt}{\itshape}{}{\bfseries}{.}{10pt}{{\thmname{#1}\thmnumber{ #2}\thmnote{ (#3)}}}
\newtheoremstyle{thmlemcorr*}{10pt}{10pt}{\itshape}{}{\bfseries}{.}\newline{{\thmname{#1}\thmnumber{ #2}\thmnote{ (#3)}}}
\newtheoremstyle{remexample}{10pt}{10pt}{}{}{\bfseries}{.}{10pt}{{\thmname{#1}\thmnumber{ #2}\thmnote{ (#3)}}}
\newtheoremstyle{ass}{10pt}{10pt}{}{}{\bfseries}{.}{10pt}{{\thmname{#1}\thmnumber{ A#2}\thmnote{ (#3)}}}

\theoremstyle{thmlemcorr}
\newtheorem{theorem}{Theorem}
\numberwithin{theorem}{section}
\newtheorem{lemma}[theorem]{Lemma}
\newtheorem{corollary}[theorem]{Corollary}
\newtheorem{proposition}[theorem]{Proposition}

\theoremstyle{thmlemcorr*}
\newtheorem*{theorem*}{Theorem}
\newtheorem{lemma*}[theorem]{Lemma}
\newtheorem{corollary*}[theorem]{Corollary}
\newtheorem{proposition*}[theorem]{Proposition}
\newtheorem{problem*}[theorem]{Problem}
\newtheorem{conjecture*}[theorem]{Conjecture}
\newtheorem{definition*}[theorem]{Definition}
\newtheorem{assumption*}[theorem]{Assumption}

\theoremstyle{remexample}
\newtheorem{remark}[theorem]{Remark}

\theoremstyle{ass}

\newcommand{\Acal}{\mathcal{A}}

\newcommand{\Fcal}{\mathcal{F}}

\newcommand{\Hcal}{\mathcal{H}}

\newcommand{\Lcal}{\mathcal{L}}
\newcommand{\Mcal}{\mathcal{M}}

\newcommand{\Ocal}{\mathcal{O}}

\newcommand{\Cbb}{\mathbb{C}}

\newcommand{\Nbb}{\mathbb{N}}

\newcommand{\Rbb}{\mathbb{R}}

\newcommand{\Tbb}{\mathbb{T}}
\newcommand{\T}{\mathbb{T}}

\newcommand{\Zbb}{\mathbb{Z}}

\DeclareMathOperator{\id}{Id}

\DeclareMathOperator{\dist}{dist}

\renewcommand{\Re}{\operatorname{Re}}
\newcommand{\We}{\operatorname{We}}
\newcommand{\Ca}{\operatorname{Ca}}
\newcommand{\Bo}{\operatorname{Bo}}

\newcommand{\ee}{\mathrm{e}}
\newcommand{\ii}{\mathrm{i}}
\newcommand{\set}[2]{\left\{\, #1 \ \textup{:}\ #2 \,\right\}}

\newcommand{\setb}[2]{\bigl\{\, #1 \ \textup{:}\ #2 \,\bigr\}}
\newcommand{\setB}[2]{\Bigl\{\, #1 \ \textup{:}\ #2 \,\Bigr\}}

\newcommand{\norm}[1]{\|#1\|}

\newcommand{\normb}[1]{\bigl\|#1\bigr\|}
\newcommand{\normB}[1]{\Bigl\|#1\Bigr\|}

\newcommand{\abs}[1]{|#1|}
\newcommand{\abslr}[1]{\left|#1\right|}

\newcommand{\absb}[1]{\bigl|#1\bigr|}

\newcommand{\dpr}[1]{\langle #1 \rangle}	
\newcommand{\dprlr}[1]{\left\langle #1 \right\rangle}

\newcommand{\dd}{\;\mathrm{d}}

\newcommand{\N}{\mathbb{N}}
\newcommand{\R}{\mathbb{R}}
\newcommand{\C}{\mathbb{C}}

\newcommand{\Z}{\mathbb{Z}}

\newcommand{\eps}{\varepsilon}

\def\div{\mathrm{div\,}}

\def\XXint#1#2#3{{\setbox0=\hbox{$#1{#2#3}{\int}$}
\vcenter{\hbox{$#2#3$}}\kern-.5\wd0}}

\renewcommand{\eps}{\varepsilon}
\renewcommand{\phi}{\varphi}

\newcommand{\Hev}[1]{H^{#1}_{\mathrm{ev}}}
\newcommand{\Hevd}[1]{\dot H^{#1}_{\mathrm{ev}}}
\newcommand{\Lev}[1]{L_{#1,\mathrm{ev}}}
\newcommand{\Levd}[1]{\dot L_{#1,\mathrm{ev}}}

\begin{document}


\title[]{Analysis of a three-dimensional fluid flow in rotating cylinders}

\author{Juri Joussen}
\address{\textit{Juri Joussen:} Institute of Analysis, Dynamics and Modeling, University of Stuttgart, Pfaffenwaldring~57, 70569 Stuttgart, Germany}
\email{juri.joussen@iadm.uni-stuttgart.de}

\author{Janne Laudien}
\address{\textit{Janne Laudien:} University of Stuttgart, Pfaffenwaldring~57, 70569 Stuttgart, Germany}
\email{st179362@stud.uni-stuttgart.de}

\author{Christina Lienstromberg}
\address{\textit{Christina Lienstromberg:}  Institute of Analysis, Dynamics and Modeling, University of Stuttgart, Pfaffenwaldring~57, 70569 Stuttgart, Germany}
\email{christina.lienstromberg@iadm.uni-stuttgart.de}

\author{Juan J.L. Velázquez }
\address{\textit{Juan J.L. Velázquez:} Institute of Applied Mathematics, University of Bonn, Endenicher Allee~60, 53115 Bonn, Germany}
\email{velazquez@iam.uni-bonn.de}

\begin{abstract}
Subject of consideration is the modelling and analysis of a capillary-driven three-dimensional rimming-flow problem. We present the derivation of a fourth-order quasilinear degenerate-parabolic partial differential equation for the height $h > 0$ of a fluid film coating the inner wall of a cylinder that rotates around a horizontal axis. The equation arises from a rescaled Navier--Stokes system for thin fluid films by means of a lubrication approximation and accounts for the physical effects of rotation, surface tension and gravity. The effect of the latter is measured by a non-dimensional parameter $0 \leq \delta \ll 1$.\\
\noindent 
We characterise the structure of the steady states depending on the ratio $\ell$ of the cylinder length to its radius. 
In the absence of gravity ($\delta=0$), in the case $\frac{\ell}{\pi} \notin \Z$, steady states are unique.
For $0 < \delta \ll 1$, steady states are shown to be locally unique for any $\ell$. These steady states are stable for $\ell < \pi$, while they are unstable for $\ell > \pi$.
\\
Furthermore, in the absence of gravity, for all $\ell > 0$, we show that there exists a manifold of time-periodic solutions. In the critical case $\ell = \pi$, we study the dynamics of the solutions close to the manifold of periodic orbits in the critical case $\ell = \pi$ on the large time scale $\tau = \delta^2 t$. It turns out that in the time scale $\tau$ this dynamics can be approximated by a system of ordinary differential equations.
\end{abstract}
\vspace{4pt}

\maketitle






\section{Introduction}

The present paper investigates a capillary-driven three-dimensional thin-film rimming flow. A \emph{rimming flow} is the flow of a fluid partially filling a rotating horizontal cylinder with gravity acting in the perpendicular direction. We consider a cylinder of radius $R>0$ and length $L>0$ rotating around its axis at constant angular velocity $\omega>0$. We assume that the cylinder is partially filled by a Newtonian incompressible fluid which forms for certain values of $\omega$ a thin, coating fluid layer of mean thickness $d\ll R$ at the inner wall of the impermeable solid cylinder. We denote by $\tilde g$ the gravitational acceleration and by $\tilde\Omega(\tilde t)$ the region occupied by the fluid at time $\tilde t$. See Figure \ref{fig:rimming-flow} for the described setting.

\begin{center}
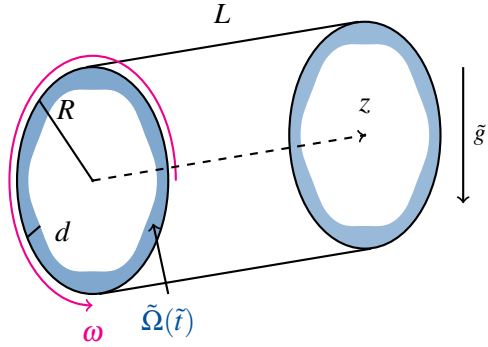
\begin{figure}[ht]
	\begin{tikzpicture}[xscale=0.2, yscale=0.3]
		\begin{scope}[xshift=-12,yshift=5cm]
		\draw [thick] (0,0) -- (18cm,2cm) node[midway,label=\(L\)] {};
		\end{scope}
		\begin{scope}[xshift=12,yshift=-5cm]
		\draw [thick] (0,0) -- (18cm,2cm);
		\end{scope}
		\begin{scope}
		\draw [thick, fill=luh-dark-blue!50] circle (5);
		\fill [fill=white] (0:4.25) 
		[out=90,in=300] to (30:4) 
		[out=120,in=330]   to (60:4.25) 
		[out=150,in=0]  to (90:4) 
		[out=180,in=30] to (120:4.25) 
		[out=210,in=60]  to (150:4) 
		[out=240,in=90] to (180:4.25) 
		[out=270,in=120] to (210:4) 
		[out=300,in=150] to (240:4.25) 
		[out=330,in=180] to (270:4) 
		[out=0,in=210] to (300:4.25) 
		[out=30,in=240] to (330:4) 
		[out=60,in=270] to (0:4.25);
		\filldraw circle (0.05);
		\draw [thick,->,color=magenta] (5.5,0) arc (0:270:5.5) node[label=below:\(\omega\)] {}; 
		\draw [thick] (0,0) -- (135:5) node[midway,label=\(R\)] {}; 
		\draw [thick] (210:5) -- (210:4) node[midway,label=right:\(d\)] {}; 
		\draw[thick,->] (5,-5) node[below] {\textcolor{luh-dark-blue}{$\tilde \Omega(\tilde t)$}} -- (4,-1.9) {};
		\end{scope}
		\begin{scope}[xshift=18cm,yshift=2cm]
		\draw [thick, fill=luh-light-blue] circle (5);
		\fill [fill=white] (0:4.25) 
		[out=90,in=300] to (30:4) 
		[out=120,in=330]   to (60:4.25) 
		[out=150,in=0]  to (90:4) 
		[out=180,in=30] to (120:4.25) 
		[out=210,in=60]  to (150:4) 
		[out=240,in=90] to (180:4.25) 
		[out=270,in=120] to (210:4) 
		[out=300,in=150] to (240:4.25) 
		[out=330,in=180] to (270:4) 
		[out=0,in=210] to (300:4.25) 
		[out=30,in=240] to (330:4) 
		[out=60,in=270] to (0:4.25);
		\filldraw circle (0.05);
		\draw[thick,->] (6.5,3) -- (6.5,0) node[right] {\footnotesize{$\tilde g$}} -- (6.5,-3) {};
		\end{scope}
		\draw [thick,dashed,->] (0,0) -- (18cm,2cm) node[label=\(z\)] {}; 
	\end{tikzpicture}
\caption{Thin, coating fluid film on impermeable solid cylinder wall.}
\label{fig:rimming-flow}
\end{figure} 
\end{center}

We assume that the free surface of the fluid film is the graph of a function $\tilde h(\tilde t,\theta,\tilde z)>0$ which denotes the film height above the cylinder wall at time $\tilde t$, angle $\theta$ and length $\tilde z$ in the axis direction. Thence, we start our analysis by modelling the described fluid flow by the Navier--Stokes equations with a free surface described by $\tilde h$. Since we assume our fluid film to be very thin ($d\ll R$), we may reduce our model by means of a formal asymptotic limit $\eps\coloneq\tfrac{d}{R}\to 0$. This procedure is well-known under the term lubrication approximation, see e.g. \cite{Ockendon_1995}. Together with a number of careful scaling assumptions on the physical parameters of the problem, most prominently that the flow is capillary-driven, this procedure leads to a single closed evolution equation for the (rescaled) film height $h(t,\theta,z)$:
\begin{equation}\label{eq:rimming-flow_intro}
\begin{cases}
	h_t + h_\theta + \gamma\,\div\bigl(h^3 \nabla(\Delta h + h)\bigr) - \delta \bigl(h^3\cos(\theta)\bigr)_\theta = 0,
	& 
	t > 0,\, \theta \in \Tbb, 0<z<\ell 
	\\
	h(0,\cdot,\cdot) = h_0,
	&
	\theta \in \Tbb, 0<z<\ell,
\end{cases}
\end{equation}
where $\ell\coloneq\tfrac{L}{R}$ is of order one. This equation has been previously studied in \cite{pukhnachov2005capillary}.
Moreover, equation \eqref{eq:rimming-flow_intro} is in many respects similar to the well-known classical thin-film equation (cf. for instance \cite{BF1990} for the one-dimensional case or \cite{bertsch_passo_garcke_gruen_1998} for the higher-dimensional analogue) but involves additional terms due to the circular geometry and effects of other involved physical quantities like the rotation of the cylinder and gravity. Indeed, the rotation leads, via a no-slip boundary condition for viscous fluids in the Navier--Stokes system, to the linear transport term $h_\theta$. Therefore, it matches exactly the rotation speed of the cylinder which has been rescaled to one in dimensionless coordinates. Moreover, the (dominating) surface tension gives rise to the highest-order term $\gamma\,\div\big(h^3\nabla(\Delta h + h)\big)$ with the lower-order term $\gamma\,\div\big(h^3\nabla h\big)$ arising from the circular geometry. The rescaled surface tension parameter $\gamma$ is assumed to be of order one. Finally, the gravitational influence results in the term $\delta\big(h^3\cos(\theta)\big)_\theta$ which we assume to be small, in particular $\delta\ll\gamma$.

For results on a rigorous mathematical justification of the lubrication approximation for the derivation of the classical one-dimensional thin-film equation we refer to \cite{Giacomelli_Otto_2003, Guenther_Prokert_2008}. We mention explicitly the previous paper \cite{jlv2023} for a two-dimensional flow only in the cross-section of the cylinder and the earlier work of Pukhnachov in \cite{pukhnachev1978motion,pukhnachov2005capillary}, where the same scaling limit is investigated and equation \eqref{eq:rimming-flow_intro} is studied.
The present paper revisits and extends the results of those two works in many respects.

We point out that solutions to equation \eqref{eq:rimming-flow_intro} are periodic in the angle $\theta\in\T$ where we identify the circle $\T = \R\ (\mathrm{mod}\ 2\pi)$ with the interval $[0,2\pi)$. However, this problem can only be well-posed if we impose additional boundary conditions in the $z$-direction. For simplicity we impose the Neumann-type boundary conditions $h_z = h_{zzz} = 0$ on $\T\times\{0,\ell\}$ which correspond to a zero-contact-angle condition (i.e. the free surface is perpendicular to the cylinder covers) and a no-flux condition preventing a mass flux through the covers.
While these Neumann-type boundary conditions are commonly used throughout the literature, see again e.g. \cite{BF1990}, we point out that they are, in fact, somewhat artificial. There is no general physical principle requiring the fluid surface to be perpendicular to the covers and thus these boundary conditions likely only hold true in very specific circumstances.
However, a main assumption of a lubrication approximation requires the flow velocity to be nearly parallel to the tangent plane of the supporting solid body
and this condition cannot be expected to hold true close to the cylinder covers. On the other hand, the Neumann-type boundary conditions allow us to make the problem periodic also in $z$ by means of a reflection principle. 
From there the equation is again in line with the work of Pukhnachov \cite{pukhnachov2005capillary} who imposed periodic boundary conditions from the start.

Problem \eqref{eq:rimming-flow_intro}, together with the mentioned boundary conditions becomes a quasilinear degenerate-parabolic problem of fourth-order. It is degenerate as the parabolicity is lost in the event of a film rupture $\min_{\theta,z} h(t,\theta,z) = 0$, when the mobility coefficient $h^3$ is not uniformly positive any more. As there is no general maximum principle for fourth-order problems, this can happen in finite-time even for uniformly positive initial data. Such a loss of parabolicity typically comes with non-uniqueness of solutions, cf. \cite{BF1990}. We refer to the book \cite{DiB1993} for background material on (second-order) degenerate-parabolic equations.

\subsection{Further literature on rimming-flows}

In addition to the already mentioned works \cite{pukhnachev1978motion,pukhnachov2005capillary,jlv2023}, there is a large number of previous mathematical and physical literature on rimming flows although most mathematical results seem to be restricted to a cross-section of the cylinder. We mention just a selection of the existing literature and do not claim that it is complete in any sense.

The comprehensive review article \cite{Seiden_Thomas_2011} presents a concise overview about the influence of a large number of physical quantities (e.g. rotation speed, inertia, viscosity, surface tension) relevant to the dynamics of rimming flows in existing numerical and experimental studies, for instance \cite{OS1978,ESR2004}.

For a more theoretical derivation and analysis of various rimming-flow equations we refer the reader to \cite{Moffatt_1977,OBrien2002,BOB2005,B2006,BLOB2012} and the references therein. For the investigation of weak solutions we refer for instance to \cite{BC2011,CP2014}.

Various mathematical contributions have been dedicated to the case of one-dimensional rimming-flow equations corresponding to a two-dimensional flow in the cross-section of the cylinder. These works generally assume that the fluid flow is homogeneous in the axis direction of the cylinder. For the following literature review we mostly consider the rimming-flow problem equivalent to the \emph{coating-flow} problem where the fluid is on the \emph{outside} of a rotating cylinder because the asymptotic equations are usually very similar.

Various models for capillary rimming flows in the one-dimensional cross-section case have been investigated for instance in \cite{karabash_multi-parameter_2024,benilov2005does,burchard2012convergence,chugunova2010nonnegative,karabut2007two}. Notably, the work \cite{B2006} treats the case of a full cylinder (no cross-section) but for the most part imposes the simplifying assumption $z\ll\theta$. To the best of our knowledge, equation \eqref{eq:rimming-flow_intro}, first proposed in \cite{pukhnachov2005capillary} is the only model of these   
which is, at least at a formal level, an asymptotic limit of the corresponding Navier--Stokes system.

A number of other works treats rimming flows in the case of small surface tension (or no surface tension at all) and/or large gravity. The generally understanding is that a (possibly regularised) shock-like solution evolves on the lower uprising wall of the cylinder \cite{ashmore_hosoi_stone_2003,badali2011regularized,lopes2018multiple,benilov2008steady}.

Finally, we mention a selection of papers on industrial applications of rimming flows. These arise for instance in the coating of photographic films or aluminium foils in the roller-coating industry, the industrial rotational moulding/casting process \cite{Throne_Gianchandani_1980}, or the
papermaking industry, in particular the Fourdrinier machine and paper machine dryers \cite{malkin1937behaviour,white1956residual,white1958effect,yih1960instability}. Capillary rimming flows are of particular interest in the pharmaceutical industry where the coating fluid layer controls the drug transport. 

\subsection{Main results of the paper}
We close the introduction by briefly outlining the organisation and the main results of this paper.

In Section \ref{sec:modelling} we present the derivation of a closed partial differential equation modelling the height of a thin-film rimming flow. The underlying physical model for the rimming-flow problem is a free-boundary incompressible Navier--Stokes system that accounts for the effect of gravity. Since we are interested in thin fluid films, we apply lubrication approximation to reduce the system to a single evolution equation for the height of the fluid film. We mention again that the resulting equation has already been investigated in \cite{pukhnachov2005capillary}. However, we present its derivation for a clear understanding of the involved physical quantities.

In Section \ref{sec:notation_basics}, we introduce the notation and fix the functional setting we work in. In particular, we introduce reflected versions $h$ of the solutions $\tilde h$ that allow us to consider the original boundary-value problem on the torus $\Tbb^2$ instead of $\Tbb\times (0,\ell)$ and thus to work with functions that are even in the direction of the cylinder axis.

Section \ref{sec:well-posedness_and_stat_sol} is concerned with the local well-posedness of the initial-boundary value problem for thin-film rimming-flows with positive initial values and Neumann-type boundary conditions. We show in Theorem \ref{thm:local_existence} that, for a fixed surface tension parameter $\gamma > 0$ of order one and all $\delta \geq 0$, the rimming-flow equation is quasilinear, degenerate-parabolic and of fourth order such that we can use analytic semigroup theory to establish the local existence and uniqueness of a positive maximal solution
\begin{equation*}
    H \in C\bigl([0,T);H^\alpha(\Tbb^2)\bigr) 
    \cap
    C^1\bigl((0,T);L_2(\Tbb^2)\bigr) \cap C^{\alpha/4}\bigl([0,T);L_2(\Tbb^2)\bigr)
\end{equation*}
for some $\alpha>3$.

In Section \ref{sec:steady_states}, we investigate positive stationary solutions of the rimming-flow problem. Subsection \ref{sec:stability_delta=0} treats the case $\delta=0$ when gravity is neglected. More precisely, in Lemma \ref{lem:steady-states_delta=0} we characterise positive steady states depending on the length of the cylinder. 
It turns out that in the case $\frac{\ell}{\pi} \notin \Z$ positive steady states are given by positive constants $H \equiv m$, corresponding to a cylindrical solution profile around the horizontal axis.
On the other hand, in the case $\frac{\ell}{\pi}\in\Z$, the positive steady states of fixed mass $m>0$ form a one-dimensional manifold of the form $H = H(\zeta) = m + a \cos\bigl(\frac{\ell}{\pi} \zeta\bigr)$ with $a\in\R$ and $|a| < m$. In general, these steady states have the shape of deformed cylinders in the sense that each cross section remains a circle centred on the cylinder axis but the radius may vary with $\zeta$.\\
In addition, we investigate the stability and instability properties of the constant steady states and those close to constants. In the case $\ell > \pi$ of a rather long cylinder, we find in Proposition \ref{prop:const-unstable-large-ell} that constant steady states are unstable. On the other hand, in the case $\ell \leq \pi$ of a rather short cylinder, we find in Theorem \ref{thm:stability} that the constant positive steady states are orbitally exponentially stable. More precisely, starting out from a positive initial value $H_0 > 0$ being $H^4$-close to $m > 0$, the solution converges exponentially fast to a travelling wave $v_\ast$, i.e. we have
\begin{equation*}
    \|H(t,\theta,\zeta) - m - v_\ast(\theta-t,\zeta)\|_{H^4} \leq
    C \|H_0 - m\|_{H^4} \ee^{-\omega t}
\end{equation*}
for constants $C, \omega > 0$ and all $t > 0$. The cross sections of the travelling wave are circles around a centre that is slightly shifted away from the axis of rotation. This stability result does also apply for the non-constant steady states $m + a \cos\bigl(\frac{\ell}{\pi} \zeta\bigr)$ if the coefficient $a$ is small enough, i.e. if the steady state is not too far from the constant $m > 0$.\\
The main results of this paper are discussed in Subsection \ref{sec:stability_delta_small} and in Section \ref{sec:slow_manifold}. Indeed, Subsection \ref{sec:stability_delta_small} is concerned with positive steady states in the case of small gravitational influence. In Theorem \ref{thm:perturbed-steady-states}, we state the existence of locally unique positive steady states $H_\delta$ for arbitrary cylinder lengths $\ell$ and provide an expansion in $\delta$, necessary for the subsequent stability analysis. 
In the case $\frac{\ell}{\pi}\notin \Z$, where steady states $H\equiv m$ are unique for $\delta=0$, this result is a consequence of the implicit function theorem. \\
There is an argument in \cite[Prop. 1]{pukhnachov2005capillary} suggesting that, if $0 < \delta \ll 1$, in the case $\tfrac{\ell}{\pi}\in\Z$ there is a curve of solutions for any $\delta > 0$. This means that the non-uniqueness of steady states for $\delta=0$ would persist for small positive $\delta$.
However, in the case $\frac{\ell}{\pi}\in \Z$, using a Lyapunov--Schmidt reduction as well as a Taylor expansion to third order, we show in this paper that there exists a unique solution for small but positive $\delta$. In other words, the gravity selects a unique solution out of the family $H(\zeta)=m + a \cos(\frac{\ell}{\pi} \zeta)$ of steady states for $\delta=0$.

Furthermore, in Theorem \ref{thm:exp-stability-perturbed-steady-state} these positive steady states are shown to be exponentially stable for small cylinder lengths $\ell < \pi$. As expected, in the case of large cylinder lengths $\ell > \pi$, the instability of steady states is shown to persist for small $\delta$ (Theorem \ref{thm:instability_H_delta}).

Finally, in Section \ref{sec:slow_manifold} we investigate the dynamics in the critical case $\ell = \pi$ for small but positive gravitational influence $0<\delta\ll 1$. Here, we can observe dynamics on two distinct time scales: We prove in Theorem \ref{thm:convergence-to-manifold} that solutions that are bounded away from zero reach a $\delta$-neighbourhood of the three-dimensional manifold
\begin{equation*}
    \Mcal(m) \coloneq \setb{m + a \cos(\theta) + b \sin(\theta) + c \cos(z)}{a,b,c\in\R,\ \sqrt{a^2 + b^2} + \abs{c} < m}.
\end{equation*}
The manifold $\Mcal(m)$ corresponds to fluid profiles whose cross-sections at $z = z_0$ are circles centred at $(-a, -b, z_0)$ and with radius $1 - m - c \cos(z_0)$ (cf. Remark \ref{rem:exp-conv-to-travelling-wave}). In this $\delta$-neighbourhood of $\Mcal(m)$ the dynamics are essentially governed by an ODE in the slow time variable $\tau = \delta^2 t$, see Subsection \ref{sec:dynamics_close_to_slow_manifold}.\\
In Subsection \ref{sec:linearise-ode} we compute a linearisation for the derived ODE, and in Subsection \ref{sec:numerics} we present some numerical results on the full nonlinear ODE.


\section{Physical model and derivation of the rimming-flow equation} \label{sec:modelling}

In this section, we introduce the physical setting of the three-dimensional rimming-flow problem and derive the thin-film type evolution equation \eqref{eq:thin-film-equation} for the free surface of the thin fluid film. To this end, we model the fluid flow by the incompressible Navier--Stokes equations and apply lubrication approximation to obtain an asymptotic equation in the limit of a vanishing film height. 

We refer to the work \cite{pukhnachov2005capillary}, where equation \eqref{eq:thin-film-equation} is studied and to the work \cite{jlv2023} for an analogue derivation of the corresponding two-dimensional rimming-flow problem, where the case of a cross section of the cylinder is considered. Moreover, in \cite{pernas_castano_analysis_2020, lienstromberg_analysis_2022,LV2024} the derivation of similar equations in the two-fluid Taylor--Couette setting is presented.


\medskip

\noindent\textsc{\textbf{The incompressible Navier--Stokes equations. }}
Consider a cylinder of radius $R>0$ rotating counter-clockwise with constant angular velocity $\omega > 0$
around a horizontal axis of length $L$. The inner surface of the cylinder is covered by a thin layer of viscous Newtonian fluid, see figure \ref{fig:rimming-flow}. The governing partial differential equations for the fluid motion are the incompressible Navier–Stokes equations
\begin{equation}\label{eq:Navier-Stokes_modelling}
	\begin{cases}
		\rho \left( \tilde{\textbf{u}}_{\tilde{t}} + (\tilde{\textbf{u}} \cdot \nabla) \tilde{\textbf{u}}\right)
		=
		- \nabla \tilde{p} + \tilde{\mu} \Delta \tilde{\textbf{u}} - \rho\tilde{\textbf{g}}
		&
		\text{in } \tilde{\Omega}(\tilde{t})
		\\
		\nabla\cdot \tilde{\textbf{u}} 
		=
		0
		&
		\text{in } \tilde{\Omega}(\tilde{t}),
	\end{cases}
\end{equation}
where the vector $\tilde{\textbf{u}}(\tilde{t},\tilde{\mathbf{x}}) \in \R^3$ 
denotes the velocity field at time $\tilde{t} > 0$ and position $\tilde{\textbf{x}} \in \R^3$, $\tilde{p}(\tilde t,\tilde{\mathbf{x}}) \in \R$ is the pressure, and $\rho \geq 0$ is the constant fluid density. Furthermore, $\tilde{\mu} > 0$ denotes the constant dynamic fluid viscosity and the vector $\tilde{\textbf{g}} = (0,0,\tilde{g})$ with $\tilde{g} > 0$ includes the gravitational acceleration pointing downwards in the direction perpendicular to the rotation axis.

In order to describe spatial positions in the cylinder, it is natural to introduce cylindrical coordinates $\tilde{\textbf{x}} = (\tilde{x}, \tilde{y},\tilde{z}) = (\tilde{r} \cos \theta, \tilde{r} \sin \theta,\tilde z) \in \R^3$. Denoting by $\tilde h(\tilde{t},\theta,\tilde z) > 0$ the function describing the height of the free surface of the fluid film, the area $\tilde{\Omega}(\tilde{t})$ filled by the fluid is given by
\begin{equation*}
    \tilde{\Omega}(\tilde{t}) 
    = 
    \set{(\tilde{r} \cos\theta,\tilde{r}\sin\theta,\tilde z) \in \R^3}{0\leq \theta < 2\pi,\ R -  \tilde h(\tilde{t},\theta, \tilde z) < \tilde{r} < R, 0 < \tilde z < L}
\end{equation*}
for every $\tilde t > 0$.
The boundary
\begin{equation*}
    \partial\tilde{\Omega}(\tilde{t}) = \bigl(\partial B_R(0) \times (0,L)\bigr) \cup \tilde\Gamma_0(\tilde t) \cup \tilde\Gamma_L(\tilde t) \cup \tilde{\Gamma}(\tilde{t})
\end{equation*}
of the region $\tilde{\Omega}(\tilde{t})$ consists
of the cylinder wall $\partial B_R(0) \times (0,L)$, where $B_R(0)\subset\R^2$ denotes the ball of radius $R>0$ around zero, the lateral cylinder covers
\begin{equation*}
    \tilde\Gamma_{\tilde c}(\tilde t) = \set{(\tilde r\cos\theta, \tilde r\sin\theta, \tilde c)}{0\leq\theta<2\pi, R - \tilde h(\tilde t, \theta,\tilde c) \leq \tilde r\leq R}, \quad \tilde c\in\{0,L\},
\end{equation*}
and the free surface
\begin{equation*}
    \tilde{\Gamma}(\tilde{t})
    =
    \set{\bigl(R - \tilde h(\tilde{t},\theta, \tilde z)\bigr)\bigl(\cos\theta,\sin\theta,\tfrac{\tilde z}{R - \tilde h(\tilde{t},\theta,\tilde z)}\bigr)}{0 \leq \theta < 2\pi,\, 0 < \tilde z < L}
\end{equation*}
 of the fluid.
Note that we always assume the function $\tilde h(\tilde{t},\theta,\tilde z) > 0$ to be strictly positive, i.e. that the cylinder wall is everywhere covered by fluid. 

It remains to complement Navier--Stokes equations \eqref{eq:Navier-Stokes_modelling} by suitable boundary conditions. 
First, we assume that the fluid velocity at the cylinder wall equals the rotation speed $\omega > 0$ of the cylinder. This results
in the no-slip condition
\begin{equation}\label{eq:no-slip}
	\tilde{\textbf{u}}
	=
	\omega (-\tilde{y},\tilde{x}, 0), 
	\quad
	(\tilde x,\tilde y,\tilde z) \in \partial B_{R}(0)\times (0,L), 
\end{equation}
where again $B_R(0) \subset \R^2$. 
At the free surface $\tilde{\Gamma}(\tilde{t})$, we assume the normal velocity $\tilde{V}_{\tilde{\mathbf{n}}}$ of the free surface to coincide with the bulk velocity of the fluid in normal direction, i.e.
\begin{equation}\label{eq:normal_velocity}
    \tilde{\textbf{u}}\cdot \tilde{\textbf{n}}
    =
	\tilde{V}_{\tilde{\mathbf{n}}}, 
	\quad 
	\tilde{\textbf{x}} \in \tilde{\Gamma}(\tilde{t}).
\end{equation}
Here, $\tilde{\textbf{n}}$ denotes the outer pointing normal vector at the free surface $\tilde{\Gamma}(\tilde{t})$, i.e. $\tilde{\textbf{n}}$ points from the region $\tilde{\Omega}$ filled by the fluid towards the inside of the cylinder. 
Moreover, at the free surface $\tilde{\Gamma}(\tilde{t})$ we prescribe the stress balance condition
\begin{equation}\label{eq:stress_balance}
    \tilde{\sigma}(\tilde{\textbf{u}},\tilde{p}) \,  \tilde{\textbf{n}}
    =
    \tilde{\gamma} \tilde{\kappa} \tilde{\textbf{n}},
    \quad
     \tilde{\textbf{x}} \in \tilde{\Gamma}(\tilde{t}).
\end{equation}
Here, $\tilde{\gamma} > 0$ denotes the positive constant surface tension, $\tilde{\kappa}$ the mean curvature of the free surface $\tilde \Gamma(\tilde t)$, and
\begin{equation*}
    \tilde{\sigma}(\tilde{\textbf{u}},\tilde{p})
    =
    -\tilde{p} I + 2 \tilde{\mu} \tilde{\epsilon}(\tilde{\textbf{u}})
    \quad \text{with} \quad
    \tilde{\epsilon}(\textbf{u}) 
    = 
    \tfrac{1}{2} 
    \bigl(\nabla \tilde{\textbf{u}} + \nabla \tilde{\textbf{u}}^T\bigr)
\end{equation*}
is the stress tensor of the fluid.
As usual, we may multiply the stress balance condition \eqref{eq:stress_balance}  by the tangential vectors $\tilde{\textbf{t}}_1$, $\tilde{\textbf{t}}_2$ and the normal vector $\tilde{\textbf{n}}$, respectively, to obtain the so-called tangential and normal stress balance conditions
\begin{equation}\label{eq:stress_balance_single}
    \begin{cases}
        \tilde{\textbf{t}}_1\cdot \tilde{\sigma}(\tilde{\textbf{u}},\tilde{p})\, \tilde{\textbf{n}} = 0,
	    &
	    \tilde{\textbf{x}}\in \tilde{\Gamma}(\tilde{t})
	    \\
        \tilde{\textbf{t}}_2\cdot \tilde{\sigma}(\tilde{\textbf{u}},\tilde{p})\, \tilde{\textbf{n}} = 0,
	    &
	    \tilde{\textbf{x}}\in \tilde{\Gamma}(\tilde{t})
	    \\
	    \tilde{\textbf{n}}\cdot \tilde{\sigma}(\tilde{\textbf{u}},\tilde{p})\, \tilde{\textbf{n}} = \tilde{\gamma}\tilde{\kappa},
	    & 
	    \tilde{\textbf{x}} \in \tilde{\Gamma}(\tilde{t}).
    \end{cases}
\end{equation}
The tangential stress balance conditions $\eqref{eq:stress_balance_single}_1$ and $\eqref{eq:stress_balance_single}_2$ signify that there is no tangential stress associated with gradients in $\tilde{\gamma}$ which is due to the assumption of constant surface tension.
The normal stress balance condition $\eqref{eq:stress_balance_single}_3$ states that the stress exerted by the fluid that acts perpendicular (normal) to the interface must balance the curvature force per unit area. 

\medskip

\noindent\textsc{\textbf{Navier--Stokes system in dimensionless variables. }}
We now rewrite the Navier--Stokes equations \eqref{eq:Navier-Stokes_modelling} and the corresponding boundary conditions \eqref{eq:no-slip}, \eqref{eq:normal_velocity}, and \eqref{eq:stress_balance_single} in dimensionless variables. For this purpose, we set $L = \ell R$ and introduce the rescaled variables
\begin{equation}
\label{eq:scaling}
    \begin{cases}
        \tilde x = R x,
        \quad
        \tilde y = R y, 
        \quad
        {\tilde z = R z,} 
        \quad
	    \tilde{t}= \frac{t}{\omega},
        \quad
         \tilde h = R h,
	    \quad
        \tilde{\mathbf{u}} = \omega R \textbf{u},
	    \quad
	    \tilde{p} = \rho \omega^2 R^2 p,
	    &
	    \\
        \We = \frac{\rho R^3 \omega^2}{\tilde{\gamma}},
        \quad 
        \Ca = \frac{R\omega \tilde{\mu}}{\tilde{\gamma}},
        \quad
        \Bo = \frac{\rho \tilde{g} R^2}{\tilde{\gamma}}
        .
    \end{cases}
\end{equation}
In the dimensionless problem, the cylinder has radius $1$ and length $\ell$ and rotates at constant angular velocity $1$.
The positive dimensionless quantities $\We > 0, \Ca > 0$, and $\Bo > 0$ denote the Weber number, the Capillary number, and the Bond number, respectively. The Weber number $\We $ describes the ratio of inertial forces and surface forces, whereas the Capillary number $\Ca$, describes the ratio of viscous forces and surface forces and the Bond number $\Bo$ encodes the ratio of gravity and surface tension. 
We refer to Figure \ref{fig:rimming-flow_dimless} for a sketch of the dimensionless problem.
\begin{center}
\begin{figure}[ht]
	\begin{tikzpicture}[xscale=0.2, yscale=0.3]
	\begin{scope}[xshift=-12,yshift=5cm]
		\draw [thick] (0,0) -- (18cm,2cm) node[midway,label=\(\ell\)] {};
	\end{scope}
	\begin{scope}[xshift=12,yshift=-5cm]
	   \draw [thick] (0,0) -- (18cm,2cm);
	\end{scope}
	\begin{scope}
		\draw [thick, fill=luh-dark-blue!50] circle (5);
		\fill [fill=white] (0:4.25) 
		[out=90,in=300] to (30:4) 
		[out=120,in=330]   to (60:4.25) 
		[out=150,in=0]  to (90:4) 
		[out=180,in=30] to (120:4.25) 
		[out=210,in=60]  to (150:4) 
		[out=240,in=90] to (180:4.25) 
		[out=270,in=120] to (210:4) 
		[out=300,in=150] to (240:4.25) 
		[out=330,in=180] to (270:4) 
		[out=0,in=210] to (300:4.25) 
		[out=30,in=240] to (330:4) 
	    [out=60,in=270] to (0:4.25);
		\filldraw circle (0.05);
		\draw [thick,->,color=magenta] (5.5,0) arc (0:270:5.5) node[label=below:\(1\)] {}; 
		\draw [thick] (0,0) -- (135:5) node[midway,label=\(1\)] {}; 
		\draw [thick] (210:5) -- (210:4) node[midway,label=right:\(\varepsilon\)] {}; 
		\draw[thick,->] (5,-5) node[below] {\textcolor{luh-dark-blue}{${\Omega}({t})$}} -- (4,-1.9) {};
	\end{scope}
	\begin{scope}[xshift=18cm,yshift=2cm]
		\draw [thick, fill=luh-light-blue] circle (5);
		\fill [fill=white] (0:4.25) 
		[out=90,in=300] to (30:4) 
		[out=120,in=330]   to (60:4.25) 
		[out=150,in=0]  to (90:4) 
		[out=180,in=30] to (120:4.25) 
		[out=210,in=60]  to (150:4) 
		[out=240,in=90] to (180:4.25) 
		[out=270,in=120] to (210:4) 
		[out=300,in=150] to (240:4.25) 
		[out=330,in=180] to (270:4) 
		[out=0,in=210] to (300:4.25) 
		[out=30,in=240] to (330:4) 
		[out=60,in=270] to (0:4.25);
		\filldraw circle (0.05);
		\draw[thick,->] (6.5,3) -- (6.5,0) node[right] {\footnotesize{${g}$}} -- (6.5,-3) {};
	\end{scope}
		\draw [thick,dashed,->] (0,0) -- (18cm,2cm) node[label=\(z\)] {}; 
	\end{tikzpicture}
\caption{The rimming-flow problem in dimensionless variables.}
\label{fig:rimming-flow_dimless}
\end{figure}
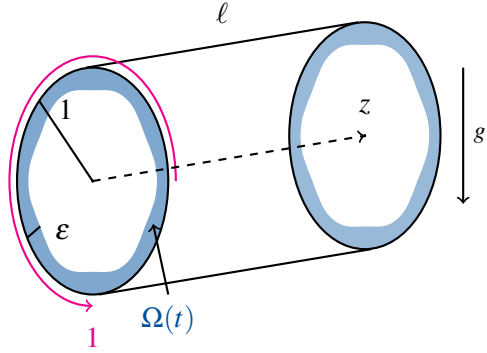 
\end{center}
The Navier--Stokes system in the dimensionless variables \eqref{eq:scaling} then reads
\begin{equation} \label{eq:Navier-Stokes_dimless}
	\begin{cases}
		\We \bigl(\textbf{u}_t + (\textbf{u}\cdot \nabla) \textbf{u}\bigr)
		=
		- \We \nabla p 
		+\Ca \Delta \textbf{u}
		- \Bo
 		\textbf{e}_2
		&
		\text{in } \Omega(t)
		\\
		\nabla\cdot \textbf{u}
		=
		0
		&
		\text{in } \Omega(t),
	\end{cases}
\end{equation}
where the region $\Omega(t)$ occupied by the fluid is given by
\begin{equation*}
    \Omega(t) = \set{(r\cos\theta,r\sin\theta,z) \in \R^3}{0\leq \theta < 2\pi,\, 1 -  h(t,\theta,z) < r < 1, 0 < z<\ell}.
\end{equation*}
The cylinder wall is given by $\partial B_1(0) \times (0,\ell)$, the spatial cylinder covers by
\begin{equation*}
    \Gamma_c(t) = \set{(r\cos\theta,r\sin\theta,c) \in \R^3}{0\leq \theta < 2\pi,\, 1 -  h(t,\theta,c) \leq r \leq 1}, \quad c\in\{0,\ell\},
\end{equation*}
and free surface by
\begin{equation*}
    \Gamma(t)
    =
    \set{(1- h(t,\theta,z))\bigl(\cos\theta,\sin\theta,\tfrac{z}{1 -  h(t,\theta,z)}\bigr)}{0\leq \theta < 2\pi,\, 0 < z<\ell}.
\end{equation*}
The boundary conditions in the dimensionless variables \eqref{eq:scaling} read
\begin{equation} \label{eq:BC_dimless}
    \begin{cases}
    	\textbf{u}
    	=
    	(-y,x,0),
	    &
	    \textbf{x} \in \partial B_1(0) \times (0,\ell)
	    \\
	    \textbf{u}\cdot \textbf{n}
	    =
	    V_{\mathbf{n}},
	    &
	    \textbf{x} \in \Gamma(t)
	    \\
	    \textbf{t}_1\cdot \sigma(\textbf{u},p)\, \textbf{n} = 0,
	    &
	    \textbf{x} \in \Gamma(t)
	    \\
        \textbf{t}_2\cdot \sigma(\textbf{u},p)\, \textbf{n} = 0,
	    &
	    \textbf{x} \in \Gamma(t)
	    \\
	    \textbf{n}\cdot \sigma(\textbf{u},p)\, \textbf{n} = \kappa,
	    &
	    \textbf{x} \in \Gamma(t).
	\end{cases}
\end{equation}

\medskip

\noindent\textbf{\textsc{The dimensionless Navier--Stokes System in Cylindrical Coordinates. }} 
It is natural to describe spatial positions in the cylinder by means of cylindrical coordinates 
\begin{equation*}
    \mathbf{x} = (r \cos(\theta), r \sin(\theta),z),
    \quad \text{where} \quad 
    \ 0 < r < 1,\ 0 \leq \theta < 2\pi \quad \text{and} \quad 0<z<\ell,
\end{equation*}
and to transform the dimensionless Navier--Stokes problem \eqref{eq:Navier-Stokes_dimless}--\eqref{eq:BC_dimless} accordingly.
Consequently, for $0 \leq \theta < 2\pi$ and $0<z<\ell$ we introduce the orthonormal basis
\begin{equation*}
    \mathbf e_r(\theta)
    \coloneq
    \begin{pmatrix}
        \cos\theta\\
        \sin\theta \\
        0
    \end{pmatrix},
    \quad
    \mathbf e_\theta(\theta)
    \coloneq
    \begin{pmatrix}
        -\sin\theta\\
        \cos\theta \\
        0
    \end{pmatrix}
    \quad \text{and} \quad
    \mathbf e_z(\theta) 
    \coloneq
    \begin{pmatrix}
        0\\ 0 \\ 1
    \end{pmatrix}
\end{equation*}
and write $\mathbf{x} = r \mathbf{e_r} + z \mathbf{e_z}$. Moreover, the velocity field and the pressure can be written as
\begin{equation*}
    \mathbf{u}(t,\mathbf{x})
    \coloneqq 
    u_r(r,\theta,z) \mathbf{e_r}(\theta) + u_\theta(r,\theta,z) \mathbf{e_\theta}(\theta) + u_z(r,\theta,z) \mathbf{e_z}(\theta)
    \quad \text{and} \quad
    \hat p(t,r,\theta)
    \coloneqq
    p\left(t,\mathbf{x}\right).
\end{equation*}
Finally, we introduce the (cylindrical) vector field
\begin{equation*}
    \hat{\mathbf{u}} \coloneq \begin{pmatrix}
        u_r \\ u_\theta \\ u_z
    \end{pmatrix},
\end{equation*}
where $u_r$ denotes the radial component, $u_\theta$ the angular component, and $u_z$ the component in the direction of the rotation axis. We frequently call $u_z$ the \emph{transverse component of the velocity field}. The dimensionless Navier--Stokes equations in cylindrical coordinates $(r,\theta,z)$ read
\begin{equation}\label{eq:NS_polar}
    \begin{cases}
		\We\left(\partial_t u_r + u_r \partial_r u_r + \frac{1}{r} u_\theta \partial_\theta u_r 
		- \frac{1}{r} u_\theta^2
        + u_z \partial_z u_r\right) \\
		\quad=
		- \We\partial_r \hat p 
		+ \Ca
		\left(\partial_r\left[\frac{1}{r} \partial_r(r u_r)\right] + \frac{1}{r^2} \partial_\theta^2 u_r 
		- \frac{2}{r^2} \partial_\theta u_\theta + \partial_z^2 u_r\right)
		- \Bo \sin(\theta)
		&
		\text{in } \hat \Omega(t)
		\\
		\We \left(\partial_t u_\theta + u_r \partial_r u_\theta + \frac{1}{r} u_\theta \partial_\theta u_\theta 
		+ \frac{1}{r} u_r u_\theta 
        + u_z \partial_z u_\theta\right) \\
		\quad=
		- \We (\frac{1}{r} \partial_\theta \hat p) 
		+ \Ca
		\left(\partial_r\left[\frac{1}{r} \partial_r(r u_\theta)\right] + \frac{1}{r^2} \partial_\theta^2 u_\theta 
		+ \frac{2}{r^2} \partial_\theta u_r + \partial_z^2 u_\theta\right)
		- \Bo \cos(\theta)
		&
		\text{in } \hat\Omega(t)
		\\
        \We(\partial_t u_z + u_r \partial_r u_z + \tfrac{1}{r} u_\theta \partial_\theta u_z + u_z \partial_z u_z)
        \\
        \quad=-\We \partial_z \hat{p} + \Ca \left(\partial_r\left[\frac{1}{r} \partial_r(r u_z)\right] + \frac{1}{r^2} \partial_\theta^2 u_z  + \partial_z^2 u_z
        \right)
        &
		\text{in } \hat\Omega(t)
        \\
		{\partial_r (r u_r) + \partial_\theta u_\theta + r \partial_z u_z}
		=
		0
		&\text{in } \hat\Omega(t),
	\end{cases}
\end{equation}
where $\hat \Omega(t) \coloneq \set{(r,\theta,z)}{1- h(t,\theta,z),\ 0 \leq \theta < 2\pi,\, 0 < z<\ell}$. In the remainder of this section, we use $u_r, u_\theta$ and $u_z$ for the components of $u$ and write $\partial_t, \partial_r, \partial_\theta$ and $\partial_z$ for the respective partial derivatives.
It remains to rewrite the boundary conditions \eqref{eq:BC_dimless} at the cylinder wall $\partial B_1(0) \times \{0<z<\ell\}$ and at the free surface
\begin{equation*}
    \hat \Gamma(t) \coloneq \set{(1- h(t,\theta,z),\theta,z)}{0 \leq \theta < 2\pi,0<z<\ell}
\end{equation*}
in cylindrical coordinates. The transformed no-slip condition at the cylinder wall reads
\begin{equation*}
    \begin{cases}
        u_r(1,\theta,z) = 0, & 0 \leq \theta < 2\pi, 0 < z<\ell
        \\
        u_\theta(1,\theta,z) = 1, & 0 \leq \theta < 2\pi, 0 < z<\ell \\
        u_z(1,\theta,z) = 0, & 0 \leq \theta < 2\pi, 0 < z<\ell,
    \end{cases}
\end{equation*}
while for the boundary conditions at the free surface, we first parametrise $\hat \Gamma(t)$ by
\begin{equation*}
    \mathbf{c}_{\hat \Gamma}(t,\cdot,\cdot)\colon [0,2\pi) \times (0,\ell) \longrightarrow \hat\Gamma(t),
    \quad
    \mathbf{c}_{\hat \Gamma}(t,\theta,z) 
    = 
    \bigl(1 -  h(t,\theta,z)\bigr) \textbf{e}_{r} + z\textbf{e}_{z}.
\end{equation*}
By means of $\mathbf{c}_{\hat \Gamma}$ we may express the unit tangent vectors at the surface $\hat \Gamma(t)$ as
\begin{equation*}
    \mathbf{\hat t}_1(t,\theta,z)
    =
    \frac{-  h_\theta \textbf{e}_r + \bigl(1 -  h\bigr) \textbf{e}_\theta}{\sqrt{(1 -  h)^2(1+ h_z^2)  +  h_\theta^2}}
    \quad \text{and}  \quad \mathbf{\hat t}_2(t,\theta,z)
    =
    \frac{-  h_z \textbf{e}_r +\textbf{e}_z}{\sqrt{(1 -  h)^2(1+ h_z^2)  +  h_\theta^2}}.
\end{equation*}
The outer pointing unit normal vector $\mathbf{\hat n}$ is given by
\begin{equation*}
    \mathbf{\hat n}(t,\theta,z)
    =
    \frac{-\bigl(1 -  h\bigr) \textbf{e}_r -  h_\theta \textbf{e}_\theta - (1- h) h_z}{\sqrt{(1 -  h)^2(1+ h_z^2)  + h_\theta^2}}.
\end{equation*}
The surface $\hat{\Gamma}(t)$ has dimensionless mean curvature $\hat{\kappa}$ given by
\begin{align*}
   \hat{\kappa} &= 
   \Big(h_\theta^2 + (1- h)^2(1+ h_z^2)\Big)^{-3/2}  \Big( (1- h)\left(h_{zz} h_\theta^2 + h_{\theta\theta} h_z^2\right)  -2(1- h)h_z h_\theta h_{\theta z}  +(1- h)^2\\
    &\qquad \qquad \qquad \quad \qquad \qquad \qquad+ \left(h_z^2(1- h)^2 + 2 h_\theta^2\right) +  (1- h)\left((1- h)^2 h_{zz} + h_{\theta \theta} \right)\Big).
\end{align*}
Finally, the normal velocity of the interface $\hat{\Gamma}$ and that of the fluid at the interface are given by
\begin{equation*}
    \hat V_{\mathbf{\hat n}} 
    = 
    \partial_t \mathbf{c}_{\hat\Gamma}\cdot \mathbf{\hat n}
    = 
    \frac{ (1 -  h) h_t}{\sqrt{(1 -  h)^2(1+ h_z^2) + (h_\theta)^2}}
    \quad \text{and} \quad
    \mathbf{\hat u}\cdot \mathbf{\hat n}
    =
    -\frac{(1- h) u_r +  h_\theta u_\theta + (1- h) h_z u_z}{\sqrt{(1 -  h)^2(1+ h_z^2) + h_\theta^2}},
\end{equation*}
respectively.
Therewith, the condition of the normal velocity $ \hat V_{\mathbf{\hat n}}$ of the free surface being equal to the normal velocity of the fluid at the interface $\hat \Gamma(t)$ reads
\begin{equation*}
     (1 -  h) h_t
    =
    -u_r (1- h) -  u_\theta h_\theta - (1- h) u_z h_z ,
    \quad
    r = 1- h,\, 0 \leq \theta < 2\pi,\, 0<z<\ell.
\end{equation*}
In order to reformulate the stress balance conditions (componentwise), we need the transformed stress tensor given by
\begin{equation*}
    \hat \sigma(\mathbf{\hat u},\hat p)
    =
    \begin{pmatrix}
        -\We \hat p + 2 \Ca \partial_r u_r
        &
        \Ca
        \left(\tfrac{1}{r} \partial_\theta u_r + \partial_r u_\theta - \tfrac{1}{r} u_\theta\right) & \Ca(\partial_z u_r + \partial_r u_z)
        \\
        \Ca
        \left(\tfrac{1}{r} \partial_\theta u_r + \partial_r u_\theta - \tfrac{1}{r} u_\theta\right)
        &
        -\We \hat p + \frac{2}{r} \Ca 
        \left(\partial_\theta u_\theta + u_r\right) & \Ca\left(\partial_z u_\theta + \tfrac{1}{r} \partial_\theta u_z \right) \\
        \Ca\left(\partial_r u_z + \partial_z u_r \right) & \Ca \left(\tfrac{1}{r}\partial_\theta u_z + \partial_z u_\theta \right) & - \We \hat p + 2 \Ca \partial_z u_z 
    \end{pmatrix}.
\end{equation*}
The tangential stress balance conditions $\mathbf{\hat t}_1\cdot \hat\sigma(\mathbf{\hat u},\hat p) \mathbf{\hat n} = 0$ and $\mathbf{\hat t}_2\cdot \hat\sigma(\mathbf{\hat u},\hat p) \mathbf{\hat n} = 0$ as well as the normal stress balance condition $\mathbf{\hat n}\cdot \hat\sigma(\mathbf{\hat u},\hat p) \mathbf{\hat n} =  \hat \kappa$ on $\hat \Gamma(t)$ may then be rewritten (componentwise) as
\begin{equation*}
    \begin{cases}
        2 (1- h) h_\theta 
        \left(\partial_r u_r - \frac{1}{1- h} \partial_\theta u_\theta - \frac{1}{1- h} u_r\right)
        -
        \bigl((1- h)^2 -  h_\theta^2\bigr)
        \left[\frac{1}{1- h} \partial_\theta u_r + \partial_r u_\theta - \frac{1}{1- h} u_\theta\right] \\
        \quad + (1- h) h_z \bigl( h_\theta (\partial_z u_r + \partial_r u_z) - (1- h)(\partial_z u_\theta + \tfrac{1}{1- h}\partial_\theta u_z) \bigr) 
        =
        0
        &\quad \text{on } \hat \Gamma(t)
        \\
          2 (1- h) h_z
        \left(\partial_r u_r - \partial_z u_z \right)
        -
        \bigl((1- h)(1- h_z^2)\bigr)
        \left[ \partial_z u_r + \partial_r u_z\right] &\\
        \quad+  h_\theta \bigl[\tfrac{1}{1- h} h_z(\partial_\theta u_r - u_\theta) + h_z\partial_r u_\theta - \left(\tfrac{1}{1- h}\partial_\theta u_z + \partial_z u_\theta\right)\bigr]
        =
        0
        &\quad \text{on } \hat \Gamma(t)
        \\
        2 \Ca \Bigg[ (1- h)^2 \left( \partial_r u_r\right) +(1- h) h_\theta
        \left(\frac{1}{1- h} \partial_\theta u_r + \partial_r u_\theta - \frac{1}{1- h} u_\theta\right) \\ \quad+  h_\theta^2 \left( \frac{1}{1- h} \partial_\theta u_\theta + \frac{1}{1- h} u_r\right) + (1- h)^2  h_z (\partial_z u_r + \partial_r u_z) \\
        \quad+   h_\theta h_z (1- h)\bigl(\partial_z u_\theta + \tfrac{1}{1- h} \partial_\theta u_z \bigr) + (1- h)^2  h_z^2  \partial_z u_z) \Bigg]\\
        =
        \We \Big((1-h^2)+ h_\theta^2+(1-h^2)h^2_z) \Big)\hat p\\
        \quad+ \tfrac{(1- h)\left(h_{zz} h_\theta^2 + h_{\theta\theta} h_z^2\right)  -2(1- h)h_z h_\theta h_{\theta z}  +(1- h)^2 + \left(h_z^2(1- h)^2 + 2 h_\theta^2\right) +  (1- h)\left((1- h)^2 h_{zz} + h_{\theta \theta} \right)}{\sqrt{ h_\theta^2 + (1- h)^2(1+ h_z^2)}} &\quad \text{on } \hat \Gamma(t). 
    \end{cases}
\end{equation*}

\medskip

\noindent\textsc{\textbf{Thin-film scaling and lubrication approximation.}}
We now consider the small-aspect ratio limit $\eps \ll 1$, i.e. the situation of the film height being small compared to the cylinder radius. We assume $L$ and $R$ to be of the same order of magnitude and introduce the scaling
\begin{equation*}
\begin{cases}
    \eps = \frac{d}{R}, \quad r = 1 - \eps \xi,
    \quad \tilde h = \eps^{-1} h,\\
    \We = \eps^2, \quad \gamma \Ca = \eps^3, \quad \gamma \Bo = g\eps^{1+\alpha},\ 0<\alpha < 1,
    \\
    \tilde u_\xi(t,\xi,\theta,z)= -\eps^{-1} u_r(t,1-\eps\xi,\theta,z),
    \quad  
    \tilde u_\theta{(t,\xi,\theta,z)} = 1 -  u_\theta(t,1-\eps\xi,\theta,z), 
    \\
    \tilde u_z{(t,\xi,\theta,z)} =  \tilde u_z(t,1-\eps\xi,\theta,z),
\end{cases}
\end{equation*}
in $\tilde \Omega(t) \coloneq \set{(\xi,\theta,z)}{0 \leq \xi \leq \tilde h(t,\theta,z),\, 0\leq \theta < 2\pi,0<z<\ell}$. That is,  we assume that the radial velocity of the fluid scales like the film height. For the angular velocity we only consider the deviation to the velocity at the cylinder wall. Furthermore, the transversal velocity is not affected by the scaling. Here, $\gamma$ and $g$ are constants of order one that appear later in the closed evolution equation for the film height in the role of surface tension and gravity.
Moreover, we consider only the case of intermediate angular velocities to avoid Taylor instabilities.
In order to obtain non-trivial dynamics we further introduce a suitable scaling for various parameters of the equations.
Indeed, we define the function
\begin{equation*}
    \tilde p\left(t,\xi,\theta,z)\coloneq \eps  (\hat p(t,1-\eps\xi,\theta,z)+\tfrac{1}{\We}\right)
\end{equation*}
in $\tilde \Omega(t)$.
Inserting these expressions into the Navier--Stokes system \eqref{eq:NS_polar}, we obtain
\begin{equation} \label{eq:NS_after_scaling}
    \begin{cases}
        -\eps\partial_t\tilde u_\xi+\eps\tilde u_\xi\partial_\xi\tilde u_\xi-\frac{\eps}{1-\eps\xi}\tilde u_\theta\partial_\theta\tilde u_\xi - \tfrac{\eps
}{1-\eps \xi}  \partial_\theta \tilde u_\xi-\frac{1}{1-\eps\xi} ( \tilde u_\theta + 1)^2
        \\=    
        \frac{1}{\eps^2}\partial_\xi\tilde p+\frac{\Ca}{\We}\Big(-\eps^{-1}\partial_\xi^2\tilde u_\xi+\frac{1}{1-\eps\xi}\partial_\xi\tilde u_\xi+\frac{\eps}{(1-\eps\xi)^2}\tilde u_\xi-\frac{\eps}{(1-\eps\xi)^2}\partial_\theta^2\tilde u_\xi\\
        \quad-\frac{2}{(1-\eps\xi)^2}\partial_\theta\tilde u_\theta\Big)-\frac{g\eps^{-1+\alpha}}{\gamma}\sin(\theta)  
        & \quad \text{in } \tilde \Omega(t) 
        \\
         \partial_t\tilde u_\theta+ \tilde u_\xi\partial_\xi\tilde u_\theta+\frac{1}{1-\eps\xi}\tilde u_\theta\partial_\theta\tilde u_\theta-\frac{\eps}{1-\eps\xi}\tilde u_\xi\tilde u_\theta+\frac{1}{1-\eps \xi}\partial_\theta  \tilde u_\theta - \frac{\eps}{1-\eps \xi} \tilde u_\xi +  \tilde u_z \partial_z \tilde u_\theta\\
        = 
        -\frac{1}{\eps(1-\eps\xi)}\partial_\theta\tilde p+\frac{\Ca}{\We}\Big(\frac{1}{\eps^2}\partial_\xi^2\tilde u_\theta-\frac{1}{\eps(1-\eps\xi)}\partial_\xi\tilde u_\theta-\frac{1}{(1-\eps\xi)^2}(\tilde u_\theta + 1)\\
        \quad+\frac{1}{(1-\eps\xi)}\partial_\theta^2\tilde u_\theta-\frac{2\eps}{(1-\eps\xi)^2}\partial_\theta\tilde u_\xi\Big)-\frac{g\eps^{-1+\alpha}}{\gamma}\cos(\theta) 
        &\quad \text{in } \tilde \Omega(t) 
        \\
        \partial_t\tilde u_z+\tilde u_\xi \partial_\xi\tilde u_z+\frac{1}{1-\eps\xi}\tilde u_\theta\partial_\theta\tilde u_z + \tfrac{1}{1-\eps \xi}  \partial_\theta \tilde u_z +  \tilde u_z \partial_z \tilde u_z \\
        =  
        -\frac{1}{\eps}\partial_z \tilde p+\frac{\Ca}{\We}\left(\eps^{-2}\partial_\xi^2\tilde u_z-\frac{1}{\eps(1-\eps\xi)}\partial_\xi\tilde u_z+ \partial_\theta ^2\tilde u_z +\partial^2_z \tilde u_z\right) 
        &\quad \text{in } \tilde \Omega(t) 
        \\
        -\eps\tilde u_\xi+(1-\eps\xi)\partial_\xi\tilde u_\xi+\partial_\theta\tilde u_\theta + \partial_z \tilde u_z=0 
        &\quad \text{in } \tilde \Omega(t). 
    \end{cases}
\end{equation}
The no-slip conditions at the cylinder, i. e. at $\{\xi=0\}$, transform to
\begin{equation*}
    \begin{cases}
        \tilde u_\xi = 0,
        \quad 
        t > 0,\, 0 \leq \theta < 2\pi, 0<z<\ell
        \\
        \tilde u_\theta = 0,
        \quad 
        t > 0,\, 0 \leq \theta < 2\pi,0<z<\ell\\
        \tilde u_z = 0,
        \quad 
        t > 0,\, 0 \leq \theta < 2\pi,0<z<\ell.
    \end{cases}
\end{equation*}
The conditions on the surface $\tilde\Gamma(t) = \set{(h(t,\theta,z),\theta)}{0 \leq \theta < 2\pi,0<z<\ell}$ can be written as
\begin{equation}\label{eq:bc_tf_scaled}
    \begin{cases}
       (1-\eps \tilde h) \eps \tilde h_{\tilde t} - (1-\eps \tilde h) \eps \tilde h_{\tilde \theta} = (1-\eps \tilde h) \eps \tilde u_\xi -\eps \tilde h_{\tilde \theta}(1+  \tilde u_{\tilde \theta}) -(1-\eps \tilde h) \eps \tilde u_z \tilde h_z
        &\quad \text{on } \tilde \Gamma(t)
        \\
        2\eps(1-\eps \tilde{h}) \tilde{h}_\theta \left( \partial_\xi\tilde u_\xi-\frac{1 }{1-\eps \tilde{h}}\partial_\theta\tilde u_\theta +\frac{\eps}{1-\eps \tilde{h}}\tilde u_\xi\right)
        \\
        \quad+\big(\eps^2 \tilde{h}_\theta^2-(1-\eps \tilde{h})^2\big)\left(-\eps^{-1}
        \partial_\xi\tilde u_\theta-\frac{\eps}{1-\eps \tilde{h}}\partial_\theta\tilde u_\xi-\frac{1}{1-\eps \tilde{h}}(\tilde u_\theta+1)\right) \\
        \quad+ (1-\eps \tilde{h}) \eps^2 \tilde h_z \tilde h_{\tilde \theta} (\eps \partial_z \tilde u_\xi - \eps^{-1} \partial_\xi \tilde u_z)  - (1-\eps \tilde h)(\partial_z \tilde u_{\tilde \theta} + \frac{1}{1-\eps \xi} \partial_\theta \tilde u_z) = 0
        &\quad\text{on } \tilde\Gamma(t)
        \\
        2(1-\eps \tilde{h})\eps \tilde{h}_{\tilde{z}} (-\partial_\xi \tilde{u}_\xi -  \partial_{\tilde{z}} \tilde{u}_{\tilde{z}}) -(1-\eps \tilde{h})(1-\eps^2 \tilde{h}_{\tilde{z}}^2) [\eps \partial_{\tilde{z}} \tilde{u}_\xi - \eps^{-1} \partial_\xi \tilde{u}_{\tilde{z}}]\\
         \quad+ \eps \tilde{h}_\theta \left[\tfrac{\eps}{1-\eps \tilde{h}}\tilde{h}_{\tilde{z}} (\eps \partial_\theta \tilde{u}_\xi - 1 +  \tilde{u}_\theta)  + \tilde h_{\tilde{z}}\partial_\xi \tilde{u}_\theta  - \left(\tfrac{1}{1-\eps \tilde{h}}\partial_\theta \tilde{u}_{\tilde{z}} - \partial_{\tilde{z}} \tilde{u}_\theta\right)
        \right] = 0 
        &\quad\text{on } \tilde\Gamma(t)
        \\
         2  \Bigg[ (1-\eps \tilde{h})^2 \left( -\partial_\xi \tilde{u}_\xi\right) +(1- \eps \tilde{h}) \eps \tilde{h}_\theta
        \left(\frac{\eps}{1- \eps \tilde{h}} \partial_\theta \tilde{u}_\xi + \eps^{-1}\partial_\xi \tilde{u}_\theta - \frac{1}{1- \eps \tilde{h}} (1- \tilde{u}_\theta\right) \\ \quad+  \eps^2 \tilde{h}_\theta^2 \left(- \frac{1}{1- \eps \tilde{h}} \partial_\theta \tilde{u}_\theta + \frac{\eps}{1- \eps \tilde{h}} \tilde{u}_\xi\right) + (1-\eps \tilde{h})^2 \eps \tilde{h}_{\tilde{z}} (\eps \partial_{\tilde{z}} \tilde{u}_\xi -\eps^{-1} \partial_\xi \tilde{u}_{\tilde{z}}) \\
        \quad+   \eps^2 \tilde{h}_\theta \tilde{h}_{\tilde{z}} (1- \eps \tilde{h})\bigl(- \partial_{\tilde{z}} \tilde{u}_\theta + \tfrac{1}{1- \eps \tilde{h}} \partial_\theta \tilde{u}_{\tilde{z}} \bigr) + (1- \eps \tilde{h})^2  \eps^2 \tilde{h}_{\tilde{z}}^2  \partial_{\tilde{z}} \tilde{u}_{\tilde{z}}) \Bigg]\\
        =
         \Big((1-\eps \tilde{h})^2+ \eps^2 \tilde{h}_\theta^2+(1-\eps \tilde{h})^2 \eps^2 \tilde{h}^2_{\tilde{z}}) \Big)\hat p \sqrt{\eps^2 \tilde{h}_\theta^2 + (1- \eps \tilde{h})^2(1+ \eps^2  \tilde{h}_{\tilde{z}}^2)}\\
         \quad+ 
         \Bigl[(1-\eps  \tilde{h})\left(\eps^3 \tilde{h}_{\tilde{z}\tilde{z}} \tilde{h}_\theta^2 +\eps^3 \tilde{h}_{\theta\theta} \tilde{h}_{\tilde{z}}^2\right)  -2(1-\eps \tilde{h})\eps^3 \tilde{h}_{\tilde{z}} \tilde{h}_\theta \tilde{h}_{\theta {\tilde{z}}}  +(1- \eps \tilde{h})^2 \Bigr.\\
        \Bigl.\quad+ \left(\eps^2 \tilde{h}_{\tilde{z}}^2(1-\eps \tilde{h})^2 + 2\eps^2 \tilde{h}_\theta^2\right) +  (1- \eps \tilde{h})\left((1- \eps \tilde{h})^2 \eps \tilde{h}_{\tilde{z}\tilde{z}} + \eps \tilde{h}_{\theta \theta} \right)\Bigr] &\quad \text{on } \hat \Gamma(t). 
    \end{cases}
\end{equation}

\medskip

\noindent\textsc{\textbf{Derivation of the thin-film equation. }}
We aim to derive a closed evolution equation for the film height $\tilde h$ in the case of the height of the liquid film being rather small compared to the cylinder radius. Thus, we multiply the first equation in \eqref{eq:NS_after_scaling} by $\frac{\eps^2}{\We}$ and keep only terms of order one. Similarly, we multiply the second equation in \eqref{eq:NS_after_scaling} by $\frac{\eps}{\We}$ and keep all terms up to order $\eps^\alpha$. The third equation is again multiplied by $\frac{\eps}{\We}$ and all terms up to order one are kept. Accordingly, the Navier--Stokes equations reduce to the system
\begin{equation*}
\begin{cases}
     \tilde p_\xi = \mathcal{O}(\eps^{1+\alpha})
    &\quad \text{in } \tilde \Omega(t)  
    \\
    \partial_\xi^2\tilde u_\theta = \gamma \tilde p_\theta + \eps^\alpha\tilde g\cos(\theta) + \mathcal{O}(\eps)
    &\quad \text{in } \tilde \Omega(t) 
    \\
     \partial_\xi^2\tilde u_z = \gamma \tilde p_z  + \mathcal{O}(\eps)
    &\quad \text{in } \tilde \Omega(t) 
    \\
    \partial_\xi\tilde u_\xi+\partial_\theta\tilde u_\theta + \partial_z \tilde u_z = \mathcal{O}(\eps)
    &\quad \text{in } \tilde \Omega(t) .
\end{cases}
\end{equation*}
To obtain the reduced boundary conditions, we multiply the kinematic boundary condition $\eqref{eq:bc_tf_scaled}_1$ by $\frac{1}{\eps}$ and the tangential stress balance conditions $\eqref{eq:bc_tf_scaled}_2$ and $\eqref{eq:bc_tf_scaled}_3$ by $\eps$. Moreover, for the normal stress balance condition $\eqref{eq:bc_tf_scaled}_4$, we use the Taylor expansion 
\begin{equation*}
    \sqrt{(1-\eps h)^2(1+\eps^2 \tilde h_z^2) + \eps^2 h_\theta^2}
    =
    1- \eps h + \Ocal(\eps^2).
\end{equation*}
Then, keeping only terms up to order $\eps^\alpha$, we end up with
\begin{equation}
    \begin{cases}
        \tilde u_\xi = 0
        &\quad \text{on } \{\xi=0\}
        \\
        \tilde u_\theta = 0
        &\quad \text{on } \{\xi=0\}
        \\
        \tilde u_z = 0
        &\quad \text{on } \{\xi=0\}
        \\
        \tilde h_{\tilde t} + \tilde h_\theta= \tilde u_\xi-\tilde u_\theta \tilde h_\theta - \tilde u_z \tilde h_z + \Ocal(\eps)
        &\quad \text{on } \tilde\Gamma(t)
        \\
       \gamma  \tilde p_\xi = \Ocal(\eps) &\quad \text{in } \tilde \Omega(\tilde t)
        \\
        \partial^2_\xi \tilde u_{\tilde \theta} =  \gamma \tilde p_{\tilde \theta} + \eps^{\alpha} g \cos(\tilde \theta) +\Ocal(\eps)&\quad \text{in } \tilde \Omega(\tilde t)
        \\
        \partial^2_\xi \tilde u_{z} =  \gamma \tilde p_{z} +\Ocal(\eps)&\quad \text{in } \tilde \Omega(\tilde t)
        \\
        \partial_\xi \tilde u_\xi + \partial_\theta \tilde u_{\tilde \theta } + \partial_z \tilde u_z = \Ocal(\eps)&\quad \text{in } \tilde \Omega(\tilde t)
        \\
        \partial_\xi \tilde u_{\tilde \theta} = \Ocal(\eps)  &\quad \text{on } \tilde\Gamma(t)
        \\
        \partial_\xi \tilde u_z = \Ocal(\eps) &\quad \text{on } \tilde\Gamma(t)
        \\
        -\tilde p = \tilde h+ \tilde h_{\theta\theta}+\tilde h_{zz} + \Ocal(\eps)
        &\quad \text{on } \tilde\Gamma(t).
    \end{cases}
\end{equation} 
As usual, by integrating this system, we can derive the closed evolution equation 
\begin{equation*}\label{eq:derived-thin-film-equation}
     \tilde h_{ t} +  \tilde h_\theta + \frac{\gamma}{3} \div\big( \tilde h^3 \nabla (\Delta  \tilde h +  \tilde h)\big) - \frac{g}{3}\eps^{\alpha} \big({\tilde h}^3\cos( \theta)\big)_{ \theta} = 0,
    \quad
    t > 0,\, 0 \leq \theta < 2\pi, 0<z<\ell,
\end{equation*}
for the height $\tilde h$ of the fluid film covering the inner cylinder wall.
Absorbing the factor $\frac{1}{3}$ into the constants, we obtain
\begin{equation}\label{eq:thin-film-equation}
     \tilde h_{ t} +  \tilde h_\theta + \gamma\, \div\big( \tilde h^3 \nabla (\Delta  \tilde h +  \tilde h)\big) - \delta \big(\tilde h^3 \cos( \theta)\big)_{ \theta} = 0,
    \quad
    t > 0,\, 0 \leq \theta < 2\pi, 0<z<\ell,
\end{equation}
with $\delta \coloneq \frac{g}{3} \eps^\alpha$.

For equation \eqref{eq:thin-film-equation} to be well-posed we need to impose boundary conditions for $z\in\{0,\ell\}$. We impose a zero-contact angle condition, meaning that the free surface at the contact points is perpendicular to the cylinder covers:
\begin{equation}\label{eq:bound-zero-angle}
    \tilde h_z = 0, \quad t>0,\ 0\leq \theta < 2\pi,\ z\in\{0,\ell\}.
\end{equation}
Moreover, we assume that there is no flux through the cylinder covers, i. e.
\begin{equation*}
    \tilde h^3 \left(\Delta \tilde h + \tilde h\right)_z = 0, \quad t>0,\ 0\leq \theta < 2\pi,\ z\in\{0,\ell\}.
\end{equation*}
This condition also guarantees naturally that the fluid mass inside the cylinder is conserved in time.
Since we only consider positive solutions and using the zero-contact angle condition \eqref{eq:bound-zero-angle}, the no-flux condition may be rewritten as
\begin{equation*}
    0 = \Delta \tilde h_z + \tilde h_z = \tilde h_{zzz}, \quad t>0,\ 0\leq \theta < 2\pi,\ z\in\{0,\ell\},
\end{equation*}
such that the boundary conditions at the lateral covers reduce to the Neumann-type conditions
\begin{equation*}
    \tilde h_z = \tilde h_{zzz} = 0, \quad t>0, 0\leq \theta < 2\pi, z\in\{0,\ell\}.
\end{equation*}
We point out that the zero-contact angle condition and the no-flux condition are also classically imposed in the usual flat thin-film case (cf. for instance \cite{BF1990,dalpasso_giacomelli_shishkov_2001}) both in 1D and 2D equations. However, we mention that these boundary conditions are somehow artificial. In order to take into account no-slip boundary conditions at the rear covers, a careful boundary layer analysis would have to be carried out.


\section{Notation and functional setting}\label{sec:notation_basics}

Much of the following analysis relies on Fourier methods. For this reason we would like to work with periodic boundary conditions instead of the Neumann-type conditions imposed above. This can indeed be done without loss of generality by means of the following reflection principle.

\medskip

\noindent \textsc{Reflection. } We define a reflected version $v \colon \Tbb^2 \to \R$ of a function $\tilde v\colon \Tbb \times (0,\ell) \to \R$ by
\begin{equation}
    v(\theta,\zeta) \coloneqq
    \begin{cases}
        \tilde v(\theta,\frac{\ell}{\pi} \zeta), & 0 \leq \zeta < \pi
        \\
        \tilde v(\theta, 2\ell-\frac{\ell}{\pi} \zeta), &
        \pi \leq \zeta < 2\pi
    \end{cases} \label{eq:reflectedfunction}
\end{equation}
See Figure \ref{fig:reflectedfunction} for a visualisation of this definition.

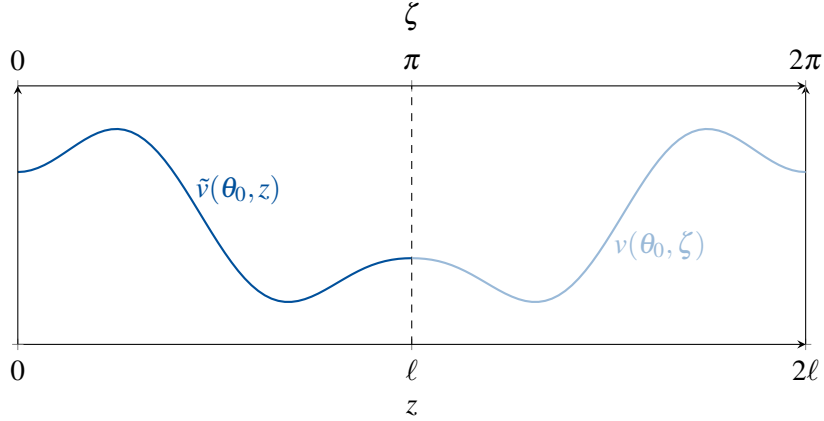
\begin{figure}[h!]
\centering
\begin{tikzpicture}

  \begin{axis}[
    domain=0:2,
    samples=60,
    axis x line=bottom,
    axis y line=right,
    xlabel={$z$},
    xtick={0,1,2},
    xticklabels={$0$,$\ell$,$2\ell$},
    ytick=false,
    yticklabels={},
    xmin=0, xmax=2,
    ymin=0, ymax=3,
    width=12cm,
    height=5cm
  ]
    \addplot[thick, luh-dark-blue, domain=0:1]
      {-467.9402*x^9 + 1893.5004*x^8 - 2771.9731*x^7 + 1560.8559*x^6 + 1.7132*x^5 - 236.0593*x^4 + 18.9031*x^2 + 2} node[right, pos=0.4] {$\tilde v (\theta_0,z)$};
  \end{axis}
  \begin{axis}[
    domain=0:2,
    samples=60,
    axis y line=left,
    axis x line=top,
    xlabel={$\zeta$},
    xtick={0,1,2},
    xticklabels={$0$,$\pi$,$2\pi$},
    ytick=false,
    yticklabels={},
    xmin=0, xmax=2,
    ymin=0, ymax=3,
    width=12cm,
    height=5cm,
  ]
    \draw[dashed] (1, 0) -- (1, 3);
    \addplot[thick, luh-light-blue, domain=1:2]
      {-467.9402*(2-x)^9 + 1893.5004*(2-x)^8 - 2771.9731*(2-x)^7 + 1560.8559*(2-x)^6 + 1.7132*(2-x)^5 - 236.0593*(2-x)^4 + 18.9031*(2-x)^2 + 2} node[right, pos=0.4] {$v(\theta_0,\zeta)$};
  \end{axis}
\end{tikzpicture}
\caption{A function $\tilde v(\theta_0,z)$ (for fixed $\theta_0\in\mathbb T$) with $z\in (0,\ell)$ and its reflected continuation $v(\theta_0,\zeta)$ rescaled to $\zeta\in (0,2\pi)$. The function $v$ can then be periodically extended to $\mathbb T^2$.}
\label{fig:reflectedfunction}
\end{figure}

Moreover, we introduce the differential operators
\begin{equation*}
    \nabla_{\ell} v(\theta,\zeta)
    \coloneqq 
    \begin{pmatrix}
        v_\theta
        \\
        \frac{\pi}{\ell} v_{\zeta}
    \end{pmatrix},
    \quad 
    \div_{\ell} 
    \begin{pmatrix}
        v_1
        \\
        v_2
    \end{pmatrix}
    \coloneq
    v_{1,\theta} + \frac{\pi}{\ell} v_{2,\zeta},
    \quad \text{and} \quad
    \Delta_{\ell} v
    \coloneqq
    v_{\theta\theta} + \frac{\pi^2}{\ell^2} v_{\zeta\zeta}.
\end{equation*}
We can then write a reflected function $v\in L_2(\Tbb^2)$ in terms of its Fourier series
\begin{equation*}
    v(\theta,\zeta) = \sum_{k,l\in\Zbb} v_{k,l}\, e^{\ii k\theta} e^{\ii l\zeta},
\end{equation*}
where the Fourier coefficients $v_{k,l}$ are defined by
\begin{equation*}
    v_{k,l} = \frac{1}{4\pi^2}\int_{\Tbb^2} v(\theta,\zeta)\, e^{-\ii k\theta}\, e^{-\ii l\zeta}\dd\theta\dd\zeta.
\end{equation*}
Under this reflection principle the thin-film equation \eqref{eq:thin-film-equation} accordingly transforms to
\begin{equation}\label{eq:transformed-thin-film-equation}
    h_t + h_\theta + \gamma\, \div_\ell \left(h^3 \nabla_\ell(\Delta_\ell h + h)\right) = \delta \left(h^3 \cos(\theta)\right)_\theta,
    \quad\theta,\zeta\in\mathbb T
\end{equation}
with periodic boundary conditions now in $\theta$ and $\zeta$. Since $h$ is real-valued its Fourier coefficients must satisfy $h_{-k,-l} = \bar{h}_{k,l}$ and since solutions $h$ to \eqref{eq:transformed-thin-film-equation} are formally constructed from solutions $\tilde h$ of the original equation \eqref{eq:thin-film-equation} by the above reflection principle, the Fourier coefficients must also obey the symmetry condition $h_{k,l} = h_{k,-l}$, allowing effectively only functions even in the coordinate $\zeta$.

\bigskip
\noindent\textsc{Function Spaces. } The reflection principle allows us to define (fractional) Sobolev norms in terms of a function's Fourier representation:
We denote by $\Lev{p}(\T^2;\Cbb)$ the subspace of $L_p(\T^2;\Cbb)$ of functions $v$ that are even in $\zeta$, that is, their Fourier coefficients satisfy the symmetry condition $v_{k,l} = v_{k,-l}$.
For $s\geq 0$ we then define the fractional Sobolev space of functions that are even in $\zeta$ by
\begin{equation*}
    \Hev{s}(\Tbb^2;\Cbb) \coloneq \setB{v\in \Lev{2}(\Tbb^2;\Cbb)}{\sum_{k,l\in\Zbb} \big(1+k^2+l^2\big)^s \abs{v_{k,l}}^2 < \infty}
\end{equation*}
with the natural norm
\begin{equation*}
    \norm{v}_{H^s} \coloneq \Big(\sum_{k,l\in\Zbb} \big(1+k^2+l^2\big)^s \abs{v_{k,l}}^2\Big)^{1/2}.
\end{equation*}
We identify the corresponding real-valued spaces $\Lev{2}(\Tbb^2)$ and $\Hev{s}(\Tbb^2)$ with closed subspaces of $\Lev{2}(\T^2;\Cbb)$ and $\Hev{s}(\Tbb^2;\Cbb)$ respectively, where in addition to the symmetry condition $v_{k,l} = v_{k,-l}$ we have $v_{k,l} = \bar v_{-k.-l}$.

Analogously, we define the corresponding homogeneous spaces with zero mean:
\begin{equation*}
\begin{aligned}
    \Hevd{s}(\T^2;\Cbb) &\coloneq \setb{v\in\Hev{s}(\T^2;\Cbb)}{v_{0,0} = 0},\\
    \Levd{p}(\T^2;\Cbb) &\coloneq \setb{v\in\Lev{p}(\T^2;\Cbb)}{v_{0,0} = 0}
\end{aligned}
\end{equation*}
with norms
\begin{equation*}
    \norm{v}_{\dot H^s} \coloneq \Big(\sum_{k,l\in\Zbb} \big(k^2+l^2\big)^s \abs{v_{k,l}}^2\Big)^{1/2}
\end{equation*}
and the usual $L_p$-norm respectively.

The even fractional Sobolev spaces defined above form an interpolation scale:
\begin{lemma}[cf. \cite{lions_magenes_1972, chandler-wilde_interpolation_2015}]
For all $s_1,s_2\geq 0$ it holds for the real interpolation space
\begin{equation*}
    \left(\Hev{s_1}(\Tbb^2), \Hev{s_2}(\Tbb^2)\right)_{\alpha, 2} = \Hev{s}(\Tbb^2),
\end{equation*}
where $0<\alpha<1$, $s = s_2\alpha+s_1(1-\alpha)$ and with equivalent norms.
\end{lemma}
Note that we would need some normalisation factor depending on $\alpha$ for the norms to be equal \cite{chandler-wilde_interpolation_2015}. We follow the same strategy as \cite{chandler-wilde_interpolation_2015} and reduce the interpolation to the case of weighted $L_2$-spaces, see also e. g. \cite{lions_magenes_1972, triebel1978}.

\begin{proof}
$\Hev{s_1}(\Tbb^2)$ and $\Hev{s_2}(\Tbb^2)$ are (real) Hilbert spaces and functions in either one can be written as
\begin{equation*}
\begin{aligned}
    v(\theta,\zeta) =&\ 4\sum_{k,l>0} a_{k,l} \cos(k\theta)\cos(l\zeta) - 4\sum_{k,l>0} b_{k,l} \sin(k\theta)\cos(l\zeta)\\
    &+ 2\sum_{k>0} a_{k,0} \cos(k\theta) - 2\sum_{k>0} b_{k,0}\sin(k\theta) + 2\sum_{l>0} a_{0,l} \cos(l\zeta) + a_{0,0}\\
    =& \sum_{k\in\Zbb}\sum_{l\in\Nbb_0} c_{k,l}\,\phi_{k,l}(\theta,\zeta)
\end{aligned}
\end{equation*}
where $v_{k,l} = a_{k,l} + \ii b_{k,l}$. Furthermore, $c_{k,l}$ and $\phi_{k,l}$ are appropriately chosen real coefficients and real-valued basis functions. This defines a linear map $\Phi:L_{2,\mathrm{ev}}(\Tbb^2)\to \{c:\Zbb\times\Nbb_0\to\Rbb\}$ between the even $L_2$ functions and the coefficient space. In fact, it turns out that $\Phi$ is an orthogonal isomorphism between $\Hev{s_i}(\Tbb^2)$ and $L_2(\Zbb\times\Nbb_0, w^{(i)})$ for $i=1,2$ where we denote by $L_2(\Zbb\times\Nbb_0, w)$ the weighted $L_2$-space equipped with the counting measure and weight $w$.

The weights $w^{(i)}$ are hereby defined as follows:
\begin{equation*}
    w^{(i)}_{k,l} =
    \begin{cases}
        1 \quad & k=l=0,\\
        \frac{\big(1+k^2\big)^{s_i}}{2} \quad & k\neq 0, l=0,\\
        \frac{\big(1+l^2\big)^{s_i}}{2} \quad & k=0, l\neq 0,\\
        \frac{\big(1+k^2+l^2\big)^{s_i}}{4} \quad & k\neq 0, l\neq 0.
    \end{cases}
\end{equation*}
We can interpolate between these weighted $L_2$-spaces by adjusting the weights according to the rule $w^{(\alpha)} = \big(w^{(1)}\big)^{1-\alpha} \big(w^{(2)}\big)^{\alpha}$, see e. g. \cite{chandler-wilde_interpolation_2015, triebel1978, bergh_lofstrom_1976}. Pulling these newly weighted $L_2$-spaces back along $\Phi$, we conclude the desired result.
\end{proof}


\section{Local well-posedness for \texorpdfstring{$\delta \geq 0$}{delta greater than zero}} \label{sec:well-posedness_and_stat_sol}

In this section we prove local existence and uniqueness of positive strong solutions to the reflected rimming-flow problem (see Section \ref{sec:notation_basics}).

That is, we look for a function
\begin{equation*}
    h\in C\big([0,T); \Lev{2}(\Tbb^2)\big) \cap C^1\big((0,T);\Lev{2}(\Tbb^2)\big)
\end{equation*}
with $h(t)\in \Hev{4}(\Tbb^2)$ for all $0<t<T$ satisfying
\begin{equation}\label{eq:reform-PDE}
\left\{
\begin{aligned}
    &h_t + h_\theta + \gamma\, \div_\ell \left(h^3 \nabla_\ell(\Delta_\ell h + h)\right) - \delta \left(h^3 \cos(\theta)\right)_\theta = 0,
    \quad & t>0,\ \theta,\zeta\in\mathbb T,\\
    &h(0) = h_0, \quad & t=0,\ \theta,\zeta\in\mathbb T.
\end{aligned}
\right.
\end{equation}

For the remainder of the paper we are going to investigate this reformulated problem \eqref{eq:reform-PDE} instead of the original problem \eqref{eq:thin-film-equation}.


\begin{theorem}[Local well-posedness for positive initial values] \label{thm:local_existence}
Let $\delta \geq 0$ be fixed and let $3<\alpha<4$. Given a positive initial value $h_0 \in \Hev{\alpha}(\T^2)$ with $h_0(\theta,\zeta) > 0$ for all $\theta, \zeta \in \T$, problem \eqref{eq:reform-PDE} admits a unique maximal solution
\begin{equation*}
    h \in C\bigl([0,T);\Hev{\alpha}(\T^2)\bigr) 
    \cap
    C^1\bigl((0,T);\Lev{2}(\T^2)\bigr) 
    \cap
    C^{\alpha/4}\bigl([0,T);\Lev{2}(\T^2)\bigr),
\end{equation*}
satisfying $h(t) \in \Hev{4}(\T^2)$ for all $0 < t < T$. We call $T=T(h_0) \in (0,\infty]$ the maximal time of existence. Moreover, the maximal solution has the following properties
\begin{itemize}
    \item[(i)] (positivity) $h$ is positive, 
    \begin{equation*}
        h(t,\theta,\zeta) > 0,
        \quad
        0 \leq t < T, \quad \theta, \zeta \in \T^2;
    \end{equation*}
    \item[(ii)] (conservation of mass) $h$ conserves its mass in the sense that
    \begin{equation*}
        \|h(t)\|_{L_1(\T^2)}
        =
        \|h_0\|_{L_1(\T^2)},
        \quad 
        0 \leq t < T; \quad 
    \end{equation*}
    \item[(iii)] (maximal existence) At least one of the following statements is true:
    \begin{itemize}
        \item $T=\infty$ (global existence);
        \item $\lim_{t \nearrow T} \|h(t)\|_{H^\alpha} = \infty$ (blow up);
        \item $\lim_{t \nearrow T} \min_{\theta, \zeta \in \T} h(t,\theta,\zeta) = 0$ (film rupture).
    \end{itemize}
    \item[(iv)] (energy dissipation) For $\delta = 0$ (no gravity) $h$ dissipates the (non-negative) energy
    \begin{equation*}
        E[h] \coloneq \frac 12\int_{\T^2} |\nabla h|^2 - |h|^2 \dd\vartheta\dd\zeta + \frac{m^2}{2},
    \end{equation*}
    i.e. we have $\tfrac{\dd}{\dd t} E[h(t)] \leq 0$.
\end{itemize}
\end{theorem}


\begin{remark}
    We briefly comment on points (ii) and (iv) of the theorem:
    \begin{itemize}
        \item We denote in the following always
        \begin{equation*}
            m \coloneq \tfrac{1}{4\pi^2}\norm{h(t)}_{L_1} = h_{0,0}.
        \end{equation*}
        For the sake of simplicity, in subsequent calculations, we call $m$ the mass of the solution.
        \item For positive $\delta$ equation \eqref{eq:reform-PDE} does not seem to follow an energy dissipation principle. However, when $\delta$ is small the energy from property (iv) (for $\delta = 0$) remains an important tool to the subsequent sections.
    \end{itemize}
\end{remark}

The existence and uniqueness in classical periodic Hölder spaces have already been stated in \cite{pukhnachov2005capillary}.

In order to prove Theorem \ref{thm:local_existence} we reformulate problem \eqref{eq:reform-PDE} as an abstract Cauchy problem. To this end, we extract the highest order differential operator from the equation which is of the form
\begin{equation*}
    \Acal (b) \coloneq \gamma b^3 \Delta_\ell^2 : \Hev{4}(\T^2;\Cbb) \longrightarrow \Lev{2}(\T^2;\Cbb),
\end{equation*}
for some coefficient function $b\in\Lev{\infty}(\T^2)$.
For all remaining lower-order terms we define the right-hand side
\begin{equation*}
    \Fcal(u) \coloneq \delta (u^3 \cos(\theta))_\theta - u_\theta - 3 \gamma\, u^2 \nabla_\ell u \cdot \nabla_\ell (\Delta_\ell u + u) - \gamma\, u^3 \Delta_\ell u
\end{equation*}
such that we can rewrite our problem as follows:
\begin{equation}\label{eq:amann-formulation}
    \left\{
    \begin{aligned}
        h_t + \Acal(h)h &= \Fcal(h) \quad &\text{in }\Hev{4}(\T^2),\\
        h(0) &= h_0 \quad & \text{in }\Hev{\alpha}(\T^2).
    \end{aligned}
    \right.
\end{equation}


\begin{proposition}\label{prop:generator}
Let $\gamma>0$ and $b\in\Lev{\infty}(\T^2)$ with $b(\theta,\zeta) \geq b_0>0$ almost everywhere in $\T^2$ be a positive function bounded away from zero. Then, the operator $-\Acal(b)$ is the generator of an analytic semigroup, in symbols:
\begin{equation*}
     \Acal(b)
     \in 
     \Hcal\bigl(\Hev{4}(\T^2;\Cbb);\Lev{2}(\T^2;\Cbb)\bigr).
\end{equation*}
\end{proposition}

\begin{proof}
We use the resolvent characterisation in \cite[Equation (3.1)]{amann_nonhomogeneous_1993}.
We start by showing that the operator $\lambda + \Acal(b)$ is bijective for all $\lambda \in \Cbb$ with $\Re \lambda > 0$. We fix such a $\lambda$ and let $f \in \Lev{2} (\T^2, \Cbb)$ be an arbitrary function that is even in $\zeta$. We define the sesquilinear form $a: \Hevd{2}(\T^2;\Cbb) \times \Hevd{2}(\T^2;\Cbb) \to \Cbb$ by
\begin{equation*}
a(\phi,v) = \bar\lambda \int_{\T^2} \phi \bar v \dd\theta \dd\zeta + \gamma \int_{\T^2} b^3 \Delta_\ell \phi \Delta_\ell \bar v \dd\theta \dd\zeta
\end{equation*}
and the linear form $F: \Hevd{2}(\T^2;\Cbb) \to \Cbb$ by
\begin{equation*}
F \phi = \int_{\T^2} \Delta_\ell \phi \, \bar f \dd\theta \dd\zeta.
\end{equation*}
We check the conditions for the Lax--Milgram lemma: The continuity of $a$ and $F$, the linearity of $F$ and the sesquilinearity of $a$ are obvious. For the coercivity we find
\begin{equation*}
\Re(a(v, v)) \geq \Re(\lambda) \norm{v}_{L_2}^2 + \gamma b_0^3 \norm{\Delta_\ell v}_{L_2}^2 \geq c_0 \norm{v}_{\Hevd{2}}^2.
\end{equation*}
Thus, there exists a unique $v \in \Hevd{2}(\T^2; \Cbb)$ such that $a(\phi, v) = F \phi$ holds for all $\phi \in \Hevd{2}(\T^2; \Cbb)$. This $v$ satisfies automatically the estimate
\begin{equation*}
\norm{v}_{\dot H^2} \leq C \norm{f}_{L_2}.
\end{equation*}
Then, we define $\tilde u \in \Hevd{4}(\T^2)$ as the unique solution of the equation $\Delta_\ell \tilde u = v$, i.e.
\begin{equation*}
\tilde u_{k,l} =
\begin{cases}
    0 \quad & \text{if }k=l=0,\\
    -\frac{v_{k,l}}{k^2 + \frac{\pi^2}{\ell^2} l^2} \quad & \text{else}
\end{cases}
\end{equation*}
and
\begin{equation*}
u = \tilde u + \frac{1}{\lambda} \int_{\T^2}(f - \Acal(b)u) \dd\theta \dd\zeta.
\end{equation*}
We claim that $u \in \Hev{4}(\T^2; \Cbb)$ satisfies $\lambda u + \Acal(b) u = f$. Indeed, we have $\Delta_\ell u = \Delta_\ell \tilde u = v$ and thus $a(\phi, \Delta_\ell u) = F \phi$ for all $\phi \in \Hevd{2}(\T^2;\Cbb)$. We integrate by parts to deduce
\begin{equation*}
\int_{\T^2} \Delta_\ell \phi \overline{(\lambda u + \Acal(b) u - f)} \dd\theta \dd\zeta = 0.
\end{equation*}
Since $\Delta_\ell:\Hevd{2}(\T^2;\Cbb) \to \Levd{2}(\T^2;\Cbb)$ is bijective, we conclude that
\begin{equation*}
\lambda u + \Acal(b) u - f \equiv C_0
\end{equation*}
is constant. Integrating over $\T^2$ we find that $C_0$ must be zero. This proves that $\lambda + \Acal(b)$ is surjective. Moreover, for $u$ we have the estimate
\begin{equation*}
\norm{u}_{H^4} \leq C_1 \norm{\tilde u}_{\dot H^4} + \tfrac{C_2}{\abs{\lambda}} \bigl(\norm{f}_{L_2} + \norm{\tilde u}_{\dot H^4}\bigr) \leq \Bigl(C_3 + \tfrac{C_4}{\abs{\lambda}}\Bigr) \norm{v}_{\dot H^2} + \tfrac{C_2}{\abs{\lambda}} \norm{f}_{L_2} \leq \Bigr(C_5 + \tfrac{C_6}{\abs{\lambda}}\Bigl) \norm{f}_{L_2}.
\end{equation*}
This proves that $\lambda + \Acal(b)$ is also injective.

Finally, we observe that for $\Re \lambda \geq \omega > 0$ the estimate above becomes uniform in $\lambda$. Testing the equation
\begin{equation*}
\lambda u + \Acal(b) u = f
\end{equation*}
with $\overline{\lambda u}$, we find
\begin{equation*}
\abs{\lambda}^2 \norm{u}_{L_2}^2 \leq \norm{f}_{L_2} \abs{\lambda} \norm{u}_{L_2} + \gamma \norm{b}^3_{L_\infty} \abs{\lambda} \norm{u}_{L_2} \norm{\Delta_\ell^2 u}_{L_2}
\end{equation*}
and thus, together with the $H^4$-estimate we obtain
\begin{equation*}
\abs{\lambda} \norm{u}_{L_2} \leq C \norm{f}_{L_2}.
\end{equation*}
Putting those two estimates together, we obtain the single estimate
\begin{equation*}
\abs{\lambda} \norm{u}_{L_2} + \norm{u}_{H^4} \leq C \norm{f}_{L_2}
\end{equation*}
for all $\lambda \in \Cbb$ with $\Re \lambda \geq 1$. By \cite[Equation (3.1)]{amann_nonhomogeneous_1993} we conclude $\Acal(b) \in \Hcal \bigl(\Hev{4} (\T^2; \Cbb); \Lev{2} (\T^2; \Cbb)\bigr)$. 
\end{proof}


With this result, we are ready to prove the local existence of Theorem \ref{thm:local_existence} by considering the formulation \eqref{eq:amann-formulation} and the theory developed by Amann \cite{amann_nonhomogeneous_1993}.

\begin{proof}[Proof of Theorem \ref{thm:local_existence}]
We fix $\beta\in (3,\alpha)$ and denote by $\Ocal$ the set of strictly positive functions in $\Hev{\beta}(\T^2)$. This set is well-defined and open since $\T^2$ is compact and $\Hev{\beta}(\T^2)$ continuously embeds into the space of continuous functions by the usual Sobolev embeddings.

We claim that $\Acal:\Ocal\to\Hcal\bigl(\Hev{4};\Lev{2}\bigr)$ and $\Fcal:\Ocal\to\Hev{\beta-3}(\T^2)$ as defined above are well-defined and locally Lipschitz continuous. Here, we interpret $\Hcal\bigl(\Hev{4};\Lev{2}\bigr)$ as a subset of all bounded linear maps equipped with the operator norm. Indeed, it follows from Proposition \ref{prop:generator} that $\Acal$ is well-defined. To prove that it is locally Lipschitz continuous we fix $u_0 \in \Ocal$ and choose $\eps>0$ small enough that $B_\eps(u_0;H^\beta)$ is uniformly bounded away from zero. We find for all $u,v \in B_\eps(u_0;H^\beta)$ and $h \in \Hev{4}(\T^2)$:
\begin{equation*}
\begin{aligned}
    \normb{(\Acal(u) - \Acal(v))h}_{L_2}
    &= \normb{\gamma (u^3 - v^3) \Delta_\ell^2 h}_{L_2} \leq \gamma \norm{u-v}_{L_\infty} \norm{u^2 + u v + v^2}_{L_\infty} \norm{\Delta_\ell^2 h}_{L_2}\\
    &\leq C(u_0, \eps) \gamma \norm{u-v}_{H^\beta} \norm{h}_{H^4}.
\end{aligned}
\end{equation*}

Now, we consider the right-hand side of \eqref{eq:amann-formulation}, $\Fcal:\Ocal\to\Hev{\beta-3}(\T^2)$. We use that the (pointwise) product of two Sobolev functions $u\in\Hev{s}(\T^2)$ and $v\in\Hev{r}(\T^2)$ with $s\leq r$ is an element of the weaker space $\Hev{s}(\T^2)$ as long as $r>1$. Recall that $r>1$ is exactly the condition required for $\Hev{r}(\T^2)$ to embed continuously into a Hölder space. Moreover, the multiplication is continuous in such a case \cite[Theorem 2.3]{Amann91_Multiplication}. It follows immediately that $\Fcal$ is well-defined.
We check the Lipschitz continuity only for the most critical third-derivative term: For all $u, v \in B_\eps(u_0;H^\beta)$ we find since $\beta-2>1$:
\begin{equation*}
\begin{aligned}
    & \normb{u^2 \nabla_\ell u \cdot \nabla_\ell (\Delta_\ell u + u) - v^2 \nabla_\ell v \cdot \nabla_\ell (\Delta_\ell v + v)}_{H^{\beta-3}} \\
    \leq\, & \normb{(u-v) (u+v) \nabla_\ell u \cdot \nabla_\ell (\Delta_\ell u + u)}_{H^{\beta-3}} + \normb{v^2 \nabla_\ell (u-v) \cdot \nabla_\ell (\Delta_\ell u + u)}_{H^{\beta-3}} \\
    &+ \normb{v^2 \nabla_\ell v \cdot \nabla_\ell \bigl(\Delta_\ell (u-v) + (u-v)\bigr)}_{H^{\beta-3}} \\
    \leq\, & C \Bigl(\norm{u-v}_{H^{\beta-2}} \norm{u+v}_{H^{\beta-2}} \normb{\nabla_\ell u}_{H^{\beta-2}} \normb{\nabla_\ell (\Delta_\ell u + u)}_{H^{\beta-3}} \\
    & + \norm{v}_{H^{\beta-2}}^2 \normb{\nabla_\ell (u-v)}_{H^{\beta-2}} \normb{\nabla_\ell (\Delta_\ell u + u)}_{H^{\beta-3}} \\
    & + \norm{v}_{H^{\beta-2}}^2 \normb{\nabla_\ell v}_{H^{\beta-2}} \normb{\nabla_\ell (\Delta_\ell (u-v) + (u-v))}_{H^{\beta-3}}\Bigr) \\
    \leq\, & C\left(\norm{u}_{H^\beta}^3 + \norm{u}_{H^\beta}^2 \norm{v}_{H^\beta} + \norm{u}_{H^\beta} \norm{v}_{H^\beta}^2 + \norm{v}_{H^\beta}^3\right) \norm{u-v}_{H^\beta}.
\end{aligned}
\end{equation*}

With the Lipschitz continuity set up, we can apply \cite[Theorem 12.1]{amann_nonhomogeneous_1993} to obtain a unique maximal solution which is positive in its existence interval.

Since our solution is positive, we have
\begin{equation*}
    \norm{h(t)}_{L_1} = \int_{\T^2} h(t,\theta,\zeta) \dd\theta \dd\zeta.
\end{equation*}
Testing equation \eqref{eq:reform-PDE} by the constant $1$ function and integrating by parts with respect to $\theta$ and $\zeta$ we find
\begin{equation*}
    \frac{\dd}{\dd t}\int_{\T^2} h(t,\theta,\zeta) \dd\theta\dd\zeta = 0,
\end{equation*}
thus we can conclude that solutions conserve mass.

We point out that $\Hev{4}(\T^2)$ is densely embedded in $\Lev{2}(\T^2)$ and hence the maximal existence property (iii) is an immediate consequence of \cite[Theorem 12.5]{amann_nonhomogeneous_1993}.

Finally, using equation \eqref{eq:reform-PDE} and setting $\delta = 0$ we obtain
\begin{equation*}
    \frac{\dd}{\dd t} E[h(t)] = -\int_{\T^2} h_t (\Delta_\ell h + h)\dd\theta\dd\zeta = -\gamma\int_{\T^2} h^3 \absb{\nabla_\ell (\Delta_\ell h + h)}^2 \dd\theta\dd\zeta,
\end{equation*}
thus proving property (iv).
\end{proof}

We point out that the blow-up result in our theorem is not optimal: In fact, a blow-up occurs in the space $\Hev{\mu}$ for every $\mu>\beta$. However, the existence time might depend on the size of $\beta$ and $\mu$.


\section{Positive stationary solutions} \label{sec:steady_states}

In this section we investigate positive stationary solutions of the rimming-flow problem with small gravitational influence $\delta \ll 1$ 
\begin{equation} \label{eq:PDE_full_stationary}
    \begin{cases}
        h_t + h_\theta + \gamma\,\div_\ell\bigl(h^3 \nabla_\ell(\Delta_\ell h + h)\bigr) = \delta\big(h^3 \cos(\theta)\big)_\theta,
        \quad t > 0,\, \theta,\zeta \in \Tbb &
        \\
        h_{k,-l} = h_{k,l}, \quad k,l\in\Zbb
        \\
        h(0,\theta,\zeta) = h_0(\theta,\zeta), \quad \theta,\zeta \in \Tbb
    \end{cases}
\end{equation}
with periodic boundary conditions in $\theta \in \Tbb$ and $\zeta\in\Tbb$. We study their existence and stability properties. That is, we are interested in time-independent functions $H\in\Hev{4}(\T^2)$ with $H(\theta,\zeta) > 0$ that solve the time-independent problem
\begin{equation} \label{eq:steady-states}
    H_\theta + \gamma\, \div_{\ell}\bigl(H^3 \nabla_{\ell}(\Delta_{\ell} H + H)\bigr) = 0,
    \quad \theta,\zeta \in \Tbb,
\end{equation}
with periodic boundary conditions in $\theta \in \Tbb$ and $\zeta\in\Tbb$.

\subsection{Positive stationary solutions for $\delta=0$} \label{sec:stability_delta=0}

By testing the equation with $\phi = \Delta_{\ell} H + H$, we can characterise positive steady states of fixed mass $m>0$, i.e. solutions of \eqref{eq:steady-states}, depending on the length of the cylinder. This has essentially already been observed by Pukhnachov in \cite[Proposition 1]{pukhnachov2005capillary}.


\begin{lemma}[Characterisation of positive steady states for $\delta=0$, \cite{pukhnachov2005capillary}]\label{lem:steady-states_delta=0}
Let $\delta=0$ and fix a positive mass $m>0$. Then, a function $H \in \Hev{4}(\T^2)$ is a positive stationary solution of \eqref{eq:PDE_full_stationary} if and only if
\begin{equation}\label{eq:H_steady-states_delta=0}
    \begin{cases}
        H(\theta,\zeta) = m + a \cos\left(\frac{\ell}{\pi}\zeta\right) \text{ with } |a| < m & \text{for }\frac{\ell}{\pi} \in \Z,
        \\
        H \equiv m, & \text{for }\frac{\ell}{\pi}\not\in\Z. 
    \end{cases}
\end{equation}
\end{lemma}


\begin{proof}
\noindent\textbf{(i) } Let $H$ be defined as in \eqref{eq:H_steady-states_delta=0}. Since $m > 0$ and $|a| < m$, we have $H > 0$ on $\T^2$. If $H\equiv m$, then obviously $H \in \Hev{4}(\T^2)$ solves \eqref{eq:steady-states}. If $H = H(\zeta) = m + a \cos(\frac{\ell}{\pi}\zeta)$, then $H$ is independent of $\theta$ and hence $H_\theta = 0$. This in turn implies
\begin{equation*}
    \nabla_\ell (\Delta_\ell H + H) =
    \begin{pmatrix}
        0\\ \tfrac{\pi^3}{\ell^3}H_{\zeta\zeta\zeta} + \tfrac{\pi}{\ell}H_\zeta
    \end{pmatrix}
    = 0.
\end{equation*}
Moreover, we have $H \in \Hev{4}(\T^2)$ since $\tfrac{\ell}{\pi}\in\Z$.

\medskip
\noindent\textbf{(ii) } Conversely, let $H \in \Hev{4}(\T^2)$ be an arbitrary positive solution of \eqref{eq:steady-states}. We test the equation with $\varphi = (\Delta_{\ell} H + H)$ and observe that
\begin{equation*}
    \begin{split}
        0
        &=
        \int_{\T^2} H_\theta (\Delta_{\ell} H + H)\, \dd \theta\, \dd\zeta
        +
        \gamma \int_{\T^2} \div_{\ell}(H^3 \nabla_{\ell}(\Delta_{\ell} H + H)\bigr) (\Delta_{\ell} H + H)\, \dd \theta\, \dd\zeta
        \\
        &=
        - \frac{1}{2} \int_{\T^2} \left(\bigl(|\nabla_{\ell} H|^2\bigr)_\theta + (H^2)_\theta\right)\, \dd\theta\, \dd\zeta
        -
        \gamma \int_{\T^2} H^3 |\nabla_{\ell}\Delta_{\ell} H + \nabla_{\ell} H|^2\, \dd \theta\, \dd \zeta
        \\
        &=
        - \gamma \int_{\T^2} H^3 |\nabla_{\ell}\Delta_{\ell} H + \nabla_{\ell} H|^2\, \dd \theta\, \dd\zeta.
    \end{split}
\end{equation*}
Since the integrand on the right-hand side is non-negative and $H > 0$ in $\T^2$, we must have $\nabla_\ell\Delta_\ell H + \nabla_\ell H=0$ in $\T^2$. Inserting this into the equation yields $H_\theta = 0$ in $\T^2$ and consequently, $\nabla_{\ell} \Delta_{\ell} H + \nabla_{\ell} H = 0$ simplifies to $(\frac{\pi}{\ell})^2H_{\zeta \zeta \zeta} + H_\zeta = 0$. This implies that $H$ is of the form
\begin{equation*}
    H(\zeta) = m + a \cos\left(\tfrac{\ell}{\pi}\zeta\right) + b \sin\left(\tfrac{\ell}{\pi}\zeta\right), \quad \zeta \in \T,
\end{equation*}
with $a, b \in \R$ such that $|a|+|b| < m$. 
Since $H \in \Hev{4}(\T^2)$ is periodic in $\zeta \in \T$, the ratio $\frac{\ell}{\pi} \in \Z$ must be an integer or $a=b=0$. Moreover, since $H \in \Hev{4}(\T^2)$ is even in $\zeta \in \T$, the coefficient $b$ must be zero and we conclude the desired result.
\end{proof}

\begin{remark}
    Note that solutions of the form $H \equiv m$ are cylinders around the $\zeta$-axis. \\
    Cross-sections of solutions of the form \eqref{eq:H_steady-states_delta=0} are always circles centred at the $\zeta$-axis, however, in the latter case its radius may vary with $\zeta$.
\end{remark}

We first investigate the stability and instability of the constant stationary solutions $H\equiv m>0$ both in the case of $\frac{\ell}{\pi}\not\in\Zbb$ and $\frac{\ell}{\pi}\in\Zbb$ (with $a=0$).
As it has already been pointed out in \cite{pukhnachov2005capillary},
for $\delta = 0$ we cannot expect the constants (or any other positive steady state for that matter) to be asymptotically stable: There will always exist a travelling wave solution of the form
\begin{equation*}
    H(t,\theta,\zeta) = m + a \cos\left(\tfrac{\ell}{\pi}\zeta\right) + \kappa_{-1} \cos(\theta-t) + \kappa_1 \sin(\theta-t)
\end{equation*}
with $a=0$ if $\frac{\ell}{\pi}\not\in\Zbb$, which for small $\kappa_{-1}$ and $\kappa_1$ can be arbitrarily close to the given steady state but will never converge to the same. This travelling wave phenomenon is of course caused by the transport term $H_\theta$ in the equation and is due to the rotation of the cylinder. We eliminate such dynamics from the problem by switching to a rotating coordinate frame instead:
\begin{equation*}
    H(t,\theta,\zeta) = \tilde H(t,\theta-t,\zeta) \quad\text{for } t>0 \text{ and } \theta,\zeta\in\T.
\end{equation*}
In this co-moving reference frame the transport term disappears as expected and our equation transforms to
\begin{equation}\label{eq:co-moving_reference_frame}
    \tilde H_t + \gamma\ \div_\ell \bigl(\tilde H^3 \nabla_\ell(\Delta_\ell \tilde H + \tilde H)\bigr) = 0.
\end{equation}
Since our steady states for the original equation are independent of $\theta$, they are automatically also steady states for this new equation. We linearise \eqref{eq:co-moving_reference_frame} around $\tilde H\equiv m$, i.e. we set $\tilde H(t) = m + u(t)$ with $u(t)\in\Hevd{4}(\T^2)$ for all $t$ and obtain
\begin{equation}\label{eq:stability-linearised}
    u_t = L u + R(u)
\end{equation}
with
\begin{align*}
    L u &\coloneq -\gamma m^3 \Delta_\ell(\Delta_\ell u + u),\\
    R(u) &\coloneq -\gamma\ \div_\ell \bigl((3m^2 u + 3m u^2 + u^3) \nabla_\ell (\Delta_\ell u + u)\bigr).
\end{align*}
In Fourier representation the linear part $Lu$ reads
\begin{equation*}
    (Lu)_{k,l} = -\gamma m^3 \bigl(k^2+\tfrac{\pi^2}{\ell^2}l^2\bigr)\bigl(k^2 + \tfrac{\pi^2}{\ell^2}l^2 - 1\bigr) u_{k,l}.
\end{equation*}
Hence, depending on the value of $\ell$ (the ratio of the cylinder length to its radius) we observe different stability behaviour:
The surface tension always acts to minimise the area of the free surface which is constrained by the zero-contact-angle and no-flux boundary conditions. While for short cylinders profiles that are independent of $\zeta$ turn out to be optimal, for long cylinders a lower surface area (compared to solutions independent of $\zeta$) can be achieved if the fluid breaks into "bubbles" along the $\zeta$-axis. The length scale of these bubbles is then expected to match the critical length of the cylinder, where the constant (i.e. fully cylindrical) steady states $H\equiv m$ transition from being stable to unstable. We refer to the book \cite[Chapter V, \emph{The forms of cells of liquid films}]{darcy_thompson} for a discussion of this topic.
Indeed, for $\ell>\pi$ we find the unstable eigenvalue $\gamma m^3 \tfrac{\pi^2}{\ell^2}\bigl(1-\tfrac{\pi^2}{\ell^2}\bigr)>0$ corresponding to the eigenfunction $u(\theta,\zeta) = \cos(\zeta)$.
As expected the constant steady states turn out to be unstable in this case (cf. \cite{lunardi_1995}):

\begin{proposition}[Instability of constant steady states for $\delta=0$ and $\ell > \pi$]\label{prop:const-unstable-large-ell}
    Let $\ell>\pi$. Then, there exist constants $C,\omega>0$ depending on $\ell$ such that equation \eqref{eq:stability-linearised} admits a non-trivial backward solution $u\in C\bigl((-\infty,0];\Hevd{4}(\T^2)\bigr)\cap C^1\bigl((-\infty,0];\Levd{2}(\T^2)\bigr)$ which satisfies
    \begin{equation*}
        \normb{u(t)}_{H^4} \leq C\, \ee^{\omega t}
    \end{equation*}
    for all $t<0$.
\end{proposition}

Proving the existence of a non-trivial backward solution $u$ on $(-\infty,0]$ converging to zero as $t\to-\infty$ asserts that we can always find an initial value arbitrarily close to zero (by going far enough back in time) which will eventually (at $t=0$) result in a solution of constant distance to zero. We expect such solutions to break into "bubbles" of size $\pi$ along the $\zeta$-axis in finite time, as described above. However, at this point no rigorous proof of this phenomenon is available.

\begin{proof}[Proof of Proposition \ref{prop:const-unstable-large-ell}]
    This is an application of \cite[Theorem 9.1.3]{lunardi_1995}. In fact, $L$ is a sectorial operator by Proposition \ref{prop:generator} and the perturbation result \cite[Proposition 2.4.1]{lunardi_1995}; its graph norm is equivalent to the $\dot H^4$-norm. From the observations above we know that $L$ has at least one unstable eigenvalue and the set of unstable eigenvalues is bounded away from the imaginary axis. Moreover, we have $R(0)=0$ and $R$ is Fréchet differentiable with locally Lipschitz continuous derivative
    \begin{equation*}
        R'(u)v = -\gamma\ \div_\ell \bigl((3m^2 v + 6m u v + 3 u^2 v)\nabla_\ell (\Delta_\ell u + u)\bigr) - \gamma\ \div_\ell\bigl((3m^2 u + 3m u^2 + u^3)\nabla_\ell(\Delta_\ell v + v)\bigr)
    \end{equation*}
    and $R'(0)=0$.
\end{proof}

For $\ell\leq\pi$ the situation is quite different: There is no unstable part of the spectrum of $L$, only a non-trivial kernel. We decompose the function space $\Hevd{4}(\T^2)$ into the kernel of $L$ and its orthogonal complement, and define subspaces $V,W\subset \Hevd{4}(\T^2)$ by
\begin{equation*}
    V \coloneq \ker L \quad\text{and} \quad W \coloneq (\ker L)^\perp.
\end{equation*}
From the Fourier representation of $L$ it is easy to see that a basis of $V$ is given by $\bigl\{\sin(\theta), \cos(\theta)\bigr\}$ if $\ell<\pi$ and $\bigl\{\sin(\theta),\cos(\theta),\cos(\zeta)\bigr\}$ if $\ell=\pi$. Either way, $V$ is a finite-dimensional closed subspace of $\Hevd{4}(\T^2)$ and we denote the corresponding orthogonal projections by $P_V$ and $P_W=\mathrm{Id}-P_V$. Before we can employ this decomposition to prove a stability result for positive constant steady states we need the following technical lemma:

\begin{lemma}\label{lem:equiv-norms}
    For each $k\geq 2$ the functionals 
    \begin{equation*}
        \norm{w}_{H^{2k}},
        \quad
        \normb{\Delta_\ell^k w}_{L_2},
        \quad
        \sqrt{\normb{\Delta_\ell^k w}_{L_2}^2 - \normb{\nabla_\ell \Delta_\ell^{k-1} w}_{L_2}^2}
        \quad \text{and} \quad
        \normb{\nabla_\ell (\Delta_\ell w + w)}_{H^{2k-3}}.
    \end{equation*}
    constitute equivalent norms for the subspace $W\cap \Hevd{2k}(\T^2)$.
\end{lemma}

\begin{proof}
    We observe that the Fourier coefficient $w_{m,n}$ of $w$ vanishes whenever $m^2+\tfrac{\pi^2}{\ell^2}n^2 \leq 1$. For all other pairs $(m,n)$ we have $m^2+\tfrac{\pi^2}{\ell^2}n^2 - 1 \geq c_0 (1 + m^2 + n^2)$ with some constant $c_0>0$ depending on $\ell$. Rewriting (i)-(iv) in their Fourier representation, we thus obtain
    \begin{equation*}
        \norm{w}_{H^{2k}}^2 \leq \tfrac{1}{c_0^3} \normb{\nabla_\ell (\Delta_\ell w + w)}_{H^{2k-3}} \leq \tfrac{1}{c_0^{2k}} \Bigl(\normb{\Delta_\ell^k w}_{L_2}^2 - \normb{\nabla_\ell \Delta_\ell^{k-1} w}_{L_2}^2\Bigr) \leq \tfrac{1}{c_0^{2k}} \normb{\Delta_\ell^k w}_{L_2}^2 \leq \bigl(\tfrac{\pi^2}{c_0 \ell^2}\bigr)^{2k} \norm{w}_{H^{2k}}^2.
    \end{equation*}
\end{proof}

We are going to prove that the zero solution of \eqref{eq:stability-linearised} is stable for small cylinder lengths. In fact, we are going to see that solutions starting close enough to zero are not only going to stay close to zero but will also converge exponentially to a single periodic orbit $t\mapsto v_\ast(\theta-t,\zeta)$ where $v_\ast\in V$ depends on the initial value. For our original equation without gravity ($\delta = 0$) this will yield the following stability result:

\begin{theorem}[Stability of constant steady states for $\delta=0$ and $\ell \leq \pi$]\label{thm:stability}
    Let $\ell\leq\pi$ and $\delta = 0$. Then, for every $\eps>0$ there exists an $\eps_0$ such that for every $H_0\in\Hev{4}(\T^2)$ satisfying $\norm{H_0-m}_{H^4}<\eps_0$ and $\tfrac{1}{4\pi^2}\int_{\T^2} H_0 \dd\theta\dd\zeta = m>0$ the full equation \eqref{eq:PDE_full_stationary} without gravity admits a unique global positive smooth solution $H\in C^\infty(\T^2)$ which satisfies
    \begin{equation}\label{eq:const-stability}
        \norm{H(t)-m}_{H^4} < \eps \quad\text{for all }t>0.
    \end{equation}
    Moreover, the constant steady state is orbitally exponentially stable, meaning that there exists a travelling wave profile $v_\ast\in V$ depending on $H_0$ satisfying necessarily $\norm{v_\ast}_{H^4}\leq C \eps_0$ such that
    \begin{equation*}
        \normb{H(t,\theta,\zeta) - m - v_\ast(\theta-t,\zeta)}_{H^4} \leq C\,\eps_0\,\ee^{-\omega t}
    \end{equation*}
    for constants $\omega,C>0$ and all $t>0$.
\end{theorem}

Theorem \ref{thm:stability} is actually stronger than simple stability of constant steady states as it proves exponential convergence to a single travelling wave. The same is also true for non-constant steady states $m + a \cos\bigl(\tfrac{\ell}{\pi}\zeta\bigr)$ (in the case $\tfrac{\ell}{\pi}\in\Zbb$) if $a$ is small enough: If we start close enough to such a steady state, we will also be close to the constant steady state and Theorem \ref{thm:stability} applies.

A similar theorem has been proved in \cite[Theorem 5.4]{pernas_castano_analysis_2020} for a corresponding 1D equation. We follow essentially the same strategy: In the subspace $W$ corresponding to the stable spectrum of $L$ we prove exponential decay in $\Hevd{4}(\T^2)$. This will be sufficient to control the right-hand side of the equation when projected onto the kernel space $V$.

\begin{proof}[Proof of Theorem \ref{thm:stability}]
    For $\eps>0$ small enough inequality \eqref{eq:const-stability} ensures that $H$ remains bounded away from zero and does not blow up in $H^4$. Thus, choosing $\eps_0<\eps$ there exists a unique smooth solution at least as long as \eqref{eq:const-stability} remains true. We define
    \begin{equation*}
        T_\eps \coloneq \sup\setB{t\geq 0}{\normb{H(s)-m}_{H^4} < \eps \text{ for all }s\leq t}
    \end{equation*}
    and want to show that $T_\eps=\infty$ if we choose $\eps_0>0$ small enough. With the projections $P_V$ and $P_W$ defined above we can reformulate equation \eqref{eq:stability-linearised} to
    \begin{equation*}
        \left\{
        \begin{aligned}
            v_t &= -\gamma\, P_V\left[\div_\ell\bigl((m+u)^3 \nabla_\ell(\Delta_\ell w + w)\bigr)\right],\\
            w_t &= Lw -\gamma\, P_W\left[\div_\ell\bigl((3m^2 u + 3m u^2 + u^3)\nabla_\ell(\Delta_\ell w + w)\bigr)\right],
        \end{aligned}
        \right.
    \end{equation*}
    where $u(t,\theta,\zeta) = H(t,\theta+t,\zeta)-m$, $v(t) = P_V u(t)$ and $w(t)=P_W u(t)$. We test the equation for $w$ with $\Delta_\ell^4 w$ and obtain in the time interval $(0,T_\eps)$
    \begin{equation*}
    \begin{aligned}
        \frac{\dd}{\dd t} \frac{1}{2} \normb{\Delta_\ell^2 w}_{L_2}^2
        &= \left(\Delta_\ell^4 w, Lw\right)_{L_2} - \gamma \Bigl(\Delta_\ell^3 w, \Delta_\ell\bigl(\div_\ell\bigl((3m^2 u + 3m u^2 + u^3) \nabla_\ell (\Delta_\ell w + w)\bigr)\bigr)\Bigr)_{L_2}\\
        &\leq -\gamma m^3 \Bigl(\normb{\Delta_\ell^3 w}_{L_2}^2 - \normb{\nabla_\ell \Delta_\ell^2 w}_{L_2}^2\Bigr) + C \gamma \normb{\Delta_\ell^3 w}_{L_2} \norm{u}_{H^3} \normb{\nabla_\ell(\Delta_\ell w + w)}_{H^3},
    \end{aligned}
    \end{equation*}
    whence, using that $\norm{u}_{H^4}<\eps$ by the definition of $T_\eps$ and applying Lemma \ref{lem:equiv-norms}, we can choose $\eps>0$ small enough such that
    \begin{equation*}
        \frac{\dd}{\dd t} \frac{1}{2} \normb{\Delta_\ell^2 w}_{L_2}^2 \leq -\gamma c_0 m^3 \normb{\Delta_\ell^3 w}_{L_2}^2 \leq -\gamma c_0 m^3 \normb{\Delta_\ell^2 w}_{L_2}^2
    \end{equation*}
    for some constant $c_0>0$. By Gronwall's lemma we conclude that
    \begin{equation}\label{eq:stable-eigenspace-exp-decay}
        \normb{\Delta_\ell^2 w(t)}_{L_2} \leq \normb{\Delta_\ell^2 w(0)}_{L_2} \ \ee^{-\gamma c_0 m^3 t} \leq \eps_0\ \ee^{-\gamma c_0 m^3 t}
    \end{equation}
    for all $0\leq t\leq T_\eps$. With this exponential decay in mind we now turn to the evolution equation for $v$. By the fundamental theorem of calculus we can estimate for all $0\leq t\leq T_\eps$:
    \begin{equation}\label{eq:stab-kernel-estimate}
    \begin{aligned}
        \norm{v(t)}_{L_2}
        &\leq \norm{v(0)}_{L_2} + \gamma\int_0^t \normb{\div_\ell\bigl((m+u(s))^3 \nabla_\ell(\Delta_\ell w(s) + w(s))\bigr)}_{L_2} \dd s\\
        &\leq \eps_0 + \gamma \int_0^t \norm{m + u(s)}_{H^2}^3 \normb{\nabla_\ell (\Delta_\ell w(s) + w(s))}_{H^1} \dd s\\
        &\leq \eps_0 + C \gamma \eps_0 \int_0^t \ee^{-\gamma c_0 m^3 s} \dd s \leq (1+\tilde C)\eps_0.
    \end{aligned}
    \end{equation}
    Here, we used the continuous multiplication of 
    functions $\Hev{2}(\Tbb^2)\cdot \Hev{1}(\Tbb^2)\hookrightarrow\Hev{1}(\Tbb^2)$
    for the second inequality and, for the third inequality, again Lemma \ref{lem:equiv-norms} to be able to apply the exponential decay from \eqref{eq:stable-eigenspace-exp-decay}.
    
    Note that the kernel space $V$ is always finite-dimensional and hence all norms on $V$ are equivalent. Therefore, we may as well take the $H^4$-norm on the left-hand side of \eqref{eq:stab-kernel-estimate} while still ending up with the same right-hand side up to a different constant.
    
    Putting \eqref{eq:stable-eigenspace-exp-decay} and \eqref{eq:stab-kernel-estimate} together, we end up with an estimate of the form
    \begin{equation*}
        \normb{H(t)-m}_{H^4} = \norm{u(t)}_{H^4} \leq C \eps_0 < \eps,
    \end{equation*}
    provided $\eps_0>0$ is small enough. But the continuity of $H$ in $\Hev{4}(\T^2)$ would imply $\normb{H(T_\eps)-m}_{H^4} = \eps$ if we had $T_\eps<\infty$, leading to a contradiction. Thus, we proved $T_\eps = \infty$ if $\eps_0>0$ is chosen small enough.

    It remains to prove the second part of the theorem. To this end, we integrate the evolution equation for $v$ in time from $s\geq 0$ to $t\geq s$. Again, by the fundamental theorem of calculus and using the same estimates as before in \eqref{eq:stab-kernel-estimate} we find
    \begin{equation*}
    \begin{aligned}
        \normb{v(t)-v(s)}_{L_2}
        &\leq \gamma \int_s^t\normb{P_V\left[\div_\ell\bigl((m+u(\tau))^3 \nabla_\ell(\Delta_\ell w(\tau) + w(\tau))\bigr)\right]}_{L_2} \dd \tau\\
        &\leq C \eps_0 \left(\ee^{-\gamma c_0 m^3 s} - \ee^{-\gamma c_0 m^3 t}\right),
    \end{aligned}
    \end{equation*}
    whence we conclude that $v(t)$ must converge to a single function $v_\ast\in V$ by Cauchy's convergence criterion. Moreover, we obtain the error estimate
    \begin{equation*}
        \normb{v(t)-v_\ast}_{H^4} \leq C \eps_0\,\ee^{-\gamma c_0 m^3 t}
    \end{equation*}
    for all $t>0$. Here, for convenience we again replaced the $L_2$-norm on the finite-dimensional subspace $V$ by the equivalent $H^4$-norm. Together with \eqref{eq:stable-eigenspace-exp-decay} we finally conclude
    \begin{equation*}
        \normb{H(t,\theta,\zeta) - m - v_\ast(\theta-t,\zeta)}_{H^4} = \norm{u(t) - v_\ast}_{H^4} \leq \norm{v(t)-v_\ast}_{H^4} + \norm{w(t)}_{H^4} \leq C \eps_0 \ee^{-\gamma c_0 m^3 t}.
    \end{equation*}
\end{proof}

\begin{remark}\label{rem:exp-conv-to-travelling-wave}
    Theorem \ref{thm:stability} shows that solutions starting close to constants converge exponentially fast to a travelling wave with the profile
    \begin{equation*}
        H_\ast(\theta,\zeta) = m + v_\ast(\theta,\zeta) = m + \eps_0\kappa_{-1} \cos(\theta) + \eps_0\kappa_1 \sin(\theta) + \eps_0 c \cos\bigl(\tfrac{\ell}{\pi}\zeta\bigr)
    \end{equation*}
    for some $\kappa_{-1},\kappa_1,c\in\Rbb$ and with $c=0$ if $\ell < \pi$. In the case of $c\neq 0$ we interpret $\zeta$ as a parameter and consider only cross-sections of the cylinder. We denote the cross-section profile at $\zeta$ by
    \begin{equation*}
        H_\ast(\theta) = c_\zeta + \eps_0\kappa_{-1} \cos(\theta) + \eps_0\kappa_1 \sin(\theta),
    \end{equation*}
    where $c_\zeta = m + \eps_0 c \cos\bigl(\tfrac{\ell}{\pi}\zeta\bigr)$.

    This cross-section profile $H_\ast$ describes almost a perfect circle around the shifted centre point $(-\eps_0 \kappa_{-1}, -\eps_0 \kappa_1)$ with radius $1-c_\zeta$ in the sense that the $H^4(\T)$-distance of $H_\ast$ to the shifted circle is of order $\eps_0^2$.
    In order to compute the $H^4$-distance we first assume without loss of generality that $\kappa_{-1}\geq 0$ and $\kappa_1=0$. Otherwise we could always change our coordinate system to
    \begin{equation*}
        \tilde\theta \coloneq
        \begin{cases}
            \theta - \tfrac{\pi}{2} + \arctan\bigl(\tfrac{\kappa_{-1}}{\kappa_1}\bigr) \quad &\text{if }\kappa_1>0,\\
            \theta + \tfrac{\pi}{2} + \arctan\bigl(\tfrac{\kappa_{-1}}{\kappa_1}\bigr) \quad &\text{if }\kappa_1<0,
        \end{cases}
    \end{equation*}
    where $\tilde H_\ast(\tilde\theta) = c_\zeta + \eps_0 \sqrt{\kappa_{-1}^2 + \kappa_1^2} \cos(\tilde\theta)$.

    Therefore, we consider a limit profile $H_\ast$ of the form
    \begin{equation*}
        H_\ast(\theta) \coloneq c_\zeta + \eps_0 \kappa_{-1}\cos(\theta),
    \end{equation*}
    parametrised by the angle with respect to the centre of the cylinder. Recall that $H_\ast(\theta)$ is the fluid height above the cylinder wall and that the cylinder has radius $1$.
    
    We want to compare this profile to the shifted circle with radius $1 - c_\zeta$ around $(-\eps_0\kappa_{-1}, 0)$ given by the equation
    \begin{equation*}
        (x+\eps_0\kappa_{-1})^2 + y^2 = (1-c_\zeta)^2.
    \end{equation*}
    A natural parametrisation of this circle would, of course, be given by
    \begin{equation*}
        c'(\theta') = \big(-\eps_0\kappa_{-1} + (1-c_\zeta)\cos(\theta'), (1-c_\zeta)\sin(\theta')\big),
    \end{equation*}
    however, this is parametrised by the angle $\theta'$ with respect to the shifted centre $M' = (-\eps_0\kappa_{-1}, 0)$. In order to make it comparable to $H_\ast$ we must first reparametrise it in terms of $\theta$.
    To this end, we consider the triangle between the centre point $M=(0,0)$ of the cylinder cross-section, the centre point of the shifted circle $M' \coloneq (-\eps_0\kappa_{-1}, 0)$ and an arbitrary point $P$ on the shifted circle, see Figure \ref{fig:shiftedcircles}.
\begin{figure}[ht]
    \centering
    \begin{tikzpicture}[scale=1.2]
        \coordinate (M) at (0.0, 0.0);
        \coordinate (Mprime) at (-0.8, 0.0);
        \coordinate (P) at ({-0.8+2*cos(50)}, {2*sin(50)});
        \coordinate (dummy) at (2.0, 0.0);
        \coordinate (Q) at ({(3-1-0.8*0.3021)*0.3021}, {(3-1-0.8*0.3021)*0.9533});

        \draw[black, thick] (M) circle (3);

        \draw[luh-light-blue] (Mprime) -- (P) node[left, pos=0.6] {$1-c_\zeta$};
        \draw[luh-light-blue] (Mprime) -- (M);
        \draw[luh-dark-blue] (M) -- (P) node[right, pos=0.5] {$r$};
        \draw[luh-dark-blue] (M) -- (dummy) node[below, pos=0.55] {$1-c_\zeta$};
        \draw[luh-dark-blue] (P) -- (Q);
        \pic[draw, luh-dark-blue, angle eccentricity=1.0, angle radius=0.7cm] {angle = dummy--M--P} node[anchor=south west, luh-dark-blue, xshift=1.5, yshift=-1.0] {$\theta$};
        \pic[draw, luh-light-blue, angle eccentricity=1.0, angle radius=0.8cm] {angle = M--Mprime--P} node[anchor=south east, luh-light-blue, xshift=-4.5, yshift=-1.0] {$\theta'$};
        
        \fill[luh-dark-blue] (M) circle (0.05) node[below] {$M$};
        \fill[luh-light-blue] (Mprime) circle (0.05) node[below] {$M'$};
        \fill[luh-light-blue] (P) circle (0.05) node[left, yshift=-1.5] {$P$};
        \fill[magenta] (Q) circle (0.05);

        \draw[domain=0:360, samples=60, smooth, variable=\theta, luh-dark-blue]
            plot ({(3-1)*cos(\theta)}, {(3-1)*sin(\theta)});
        \draw[domain=0:264, samples=60, smooth, variable=\theta, magenta]
            plot ({(3-1-0.8*cos(\theta))*cos(\theta)}, {(3-1-0.8*cos(\theta))*sin(\theta)})
            node[below] {$H_\ast$}
        [domain=264:360, samples=60, smooth, variable=\theta, magenta]
            plot ({(3-1-0.8*cos(\theta))*cos(\theta)}, {(3-1-0.8*cos(\theta))*sin(\theta)});
        \draw[domain=0:300, samples=60, smooth, variable=\theta, luh-light-blue]
            plot ({-0.8+2*cos(\theta)}, {2*sin(\theta)})
            node[above] {$C'$}
        [domain=300:360, samples=60, smooth, variable=\theta, luh-light-blue]
            plot ({-0.8+2*cos(\theta)}, {2*sin(\theta)});
    \end{tikzpicture}
    \caption{Cylinder cross-section (black) with limit profile $H_\ast$ (red) parametrised in $\theta$. The shifted circle $C'$ around $M'$ (light blue) must be reparametrised in terms of $\theta$ in order to compare it to the profile $H_\ast$. We assume the distance between $M$ and $M'$ to be small but for illustrational purposes it is exaggerated in this figure.}
    \label{fig:shiftedcircles}
\end{figure}
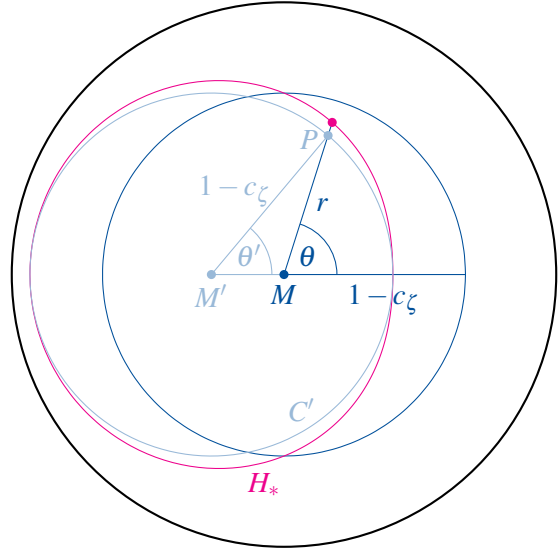

    By the law of cosines the distance $r$ between $M$ and $P$ satisfies
    \begin{equation*}
        (1-c_\zeta)^2 = r^2 + \eps_0^2 \kappa_{-1}^2 - 2r\eps_0\kappa_{-1}\cos(\pi - \theta) = r^2 + \eps_0^2\kappa_{-1}^2 + 2r\eps_0\kappa_{-1}\cos(\theta)
    \end{equation*}
    and solving for $r$ we find
    \begin{equation*}
        r(\theta) = -\eps_0\kappa_{-1}\cos(\theta) + \sqrt{(1-c_\zeta)^2 - \eps_0^2\kappa_{-1}^2\sin^2(\theta)}.
    \end{equation*}
    Functions in the sense of this paper such as $H_\ast(\theta)$ measure the height of the fluid surface above the cylinder wall rather than the distance to the centre. In order to compare $H_\ast$ to the shifted circle, we must therefore consider the function
    \begin{equation*}
        C'(\theta) \coloneq 1 - r(\theta) = 1 + \eps_0\kappa_{-1}\cos(\theta) - \sqrt{(1-c_\zeta)^2 - \eps_0^2\kappa_{-1}^2\sin^2(\theta)}.
    \end{equation*}
    For the $H^4$-distance we consequently obtain
    \begin{equation*}
    \begin{aligned}
        \normb{C' - H_\ast}_{H^4(\T)}
        &= \normB{1-c_\zeta - \sqrt{(1-c_\zeta)^2 - \eps_0^2\kappa_{-1}^2\sin^2(\theta)}}_{H^4(\T)}\\ &= \eps_0^2 \kappa_{-1}^2\normB{\tfrac{\sin^2(\theta)}{1-c_\zeta + \sqrt{(1-c_\zeta)^2 - \eps_0^2\kappa_{-1}^2\sin^2(\theta)}}}_{H^4} \leq C \eps_0^2.
    \end{aligned}
    \end{equation*}
    For the last inequality we used the estimate
    \begin{equation*}
        \normb{\tfrac{f}{g}}_{H^4} \leq C \norm{f}_{H^4} \normb{\tfrac{1}{g}}_{H^4},
    \end{equation*}
    where
    \begin{equation*}
        \bigl(\tfrac{1}{g}\bigr)'''' = \tfrac{6 (g'')^2}{g^3} + \tfrac{8 g' g'''}{g^3} - \tfrac{36 (g')^2 g''}{g^4} + \tfrac{24 (g')^4}{g^5} - \tfrac{g''''}{g^2}
    \end{equation*}
    is bounded in $L_2(\T)$ if $g$ is bounded away from zero and bounded in $H^4(\T)$. To see that the denominator $1-c_\zeta + \sqrt{(1-c_\zeta)^2 - \eps_0^2\kappa_{-1}^2\sin^2(\theta)}$ is indeed bounded in $H^4$ we use the triangle inequality and the formula
    \begin{equation*}
        \bigl(\sqrt{h}\bigr)'''' = \tfrac{1}{2} \tfrac{h''''}{\sqrt{h}} - \tfrac{3}{4} \tfrac{(h')^2}{\sqrt{h}^3} - \tfrac{h' h'''}{\sqrt{h}^3} + \tfrac{9}{4} \tfrac{(h')^2 h''}{\sqrt{h}^5} - \tfrac{15}{16}\tfrac{(h')^4}{\sqrt{h}^7}.
    \end{equation*}
    Since also the discriminant $(1-c_\zeta)^2 - \eps_0^2\kappa_{-1}^2\sin^2(\theta)$ is bounded away from zero for $\eps_0$ small enough, we conclude the result.
\end{remark}

In particular, if $\eps_0$ is of the same order of magnitude as the ratio of the mean fluid height to the cylinder radius, the error will be negligible and after the lubrication approximation they become indistinguishable.

Therefore, a solution $H$ starting close to a constant converges exponentially to a situation where in each cross-section the free surface is a circle with constant radius whose centre point rotates around the cylinder axis exactly once per full rotation of the cylinder. The asymptotic dynamics in a different cross-section will be exactly the same if $\ell<\pi$ or differ only by the radius of the circle according to $1-c_\zeta$ if $\ell=\pi$.


\subsection{Positive stationary solutions for small $\delta>0$} \label{sec:stability_delta_small}

Given the results of the previous subsection for $\delta = 0$ (no gravity), we now investigate stability and instability properties of perturbed steady states obtained for positive but small $\delta >0$, that is, for solutions $H\in\Hev{4}(\T^2)$ to
\begin{equation}\label{eq:perturbed-steady-states-equ}
    H_{\theta} + \gamma\,\div_\ell \big(H^3 \nabla_\ell (\Delta_\ell H + H)\big) = \delta \big(H^3 \cos(\theta)\big)_\theta.
\end{equation}
For the one-dimensional analogue
\begin{equation*}
    H_{\theta} + \gamma \big(H^3 (H_{\theta\theta\theta} + H_\theta)\big)_\theta = \delta \big(H^3 \cos(\theta)\big)_\theta
\end{equation*}
it has been proved in \cite{jlv2023} by means of the implicit function theorem, that for positive but small $\delta$ there exist locally unique stationary solutions $H_\delta(\theta) = m + \delta m^3 \cos(\theta) + \Ocal(\delta^2)$. It turns out that these $H_\delta$ are also solutions to the two-dimensional problem \eqref{eq:perturbed-steady-states-equ}. Remarkably, according to the following theorem, no further solutions can exist in a local neighbourhood of the constant solutions $H\equiv m$ in the case $\delta = 0$. In the case $\tfrac{\ell}{\pi}\not\in\Z$, this is a direct consequence of the implicit function theorem. On the other hand, for $\tfrac{\ell}{\pi}\in\Z$, the stationary solutions for $\delta = 0$ are not unique (cf. Lemma \ref{lem:steady-states_delta=0}) and hence the implicit function theorem is no longer applicable. However, since for $\tfrac{\ell}{\pi}\in\Z$, the (linearised) operator $u\mapsto u_\theta + \gamma m^3 \Delta_\ell (\Delta_\ell u + u)$ has a one-dimensional kernel $\mathrm{span}\big\{\cos\big(\tfrac{\ell}{\pi}\zeta\big)\big\}$ and an image with co-dimension one, we can instead apply the Lyapunov--Schmidt reduction. Hence, in a neighbourhood of the constant solutions $H\equiv m$ the problem is reduced to a one-dimensional equation on $\mathrm{span}\big\{\cos\big(\tfrac{\ell}{\pi}\zeta\big)\big\}$. The uniqueness of solutions to this one-dimensional problem can only be observed by a Taylor expansion to third order.

\begin{theorem}[Positive steady states for $0<\delta\ll 1$]\label{thm:ift_steady_states}\label{thm:perturbed-steady-states}
     For a fixed mass $m>0$, there exists $\delta_0>0$ such that for all $0\leq\delta<\delta_0$ there exists a stationary solution $H_\delta \in \Hev{4}(\T^2)$ satisfying $\tfrac{1}{4\pi^2}\int_{\T^2} H_\delta \dd\theta\dd\zeta = m$. Moreover, it has the following properties:
    \begin{enumerate}
        \item[(i)] $H_\delta$ is unique in a sufficiently small neighbourhood of $m$,
        \item[(ii)] $H_\delta$ is independent of $\zeta$,
        \item[(iii)] The curve $\delta\mapsto H_\delta\in H^4(\T)$ is continuously differentiable,
        \item[(iv)] $H_\delta$ can be expanded in $\delta$ to
        \begin{equation*}
            H_\delta(\theta) = m + \delta m^3 \cos(\theta) + \delta^2 \left(\frac{3m^5}{2(1 + 36 \gamma^2 m^6)} \cos(2\theta) - \frac{9\gamma m^8}{1 + 36\gamma^2 m^6}\sin(2\theta)\right) + \Ocal(\delta^3).
        \end{equation*}
    \end{enumerate}  
\end{theorem}

We divide the proof into two parts as announced above: In the first part, we prove existence and properties (i)--(iv) by the implicit function theorem in the case $\tfrac{\ell}{\pi}\not\in\Z$. In the second part, we use a Lyapunov--Schmidt reduction to extend the results to the case $\tfrac{\ell}{\pi}\in\Z$.

\begin{proof}[Proof of Theorem \ref{thm:perturbed-steady-states} in the case $\tfrac{\ell}{\pi}\not\in\Z$]
We start by proving existence and uniqueness in $\Hev{4}(\T^2)$ by means of the implicit function theorem and then refer to the one-dimensional problem to conclude the independence of $\zeta$. Finally, we compute the $\delta$-expansion.

\medskip
{\noindent\bfseries Step 1: Existence and uniqueness in $\Hev{4}(\T^2)$. } The proof makes use of the implicit function theorem. We separate the fixed mass $m$ from the solution and define the function $F:\Hevd{4}(\T^2)\times\R\to\Levd{2}(\T^2)$ by
    \begin{equation*}
        F(u,\delta) = u_\theta + \gamma\,\div_\ell \left((m+u)^3 \nabla_\ell(\Delta_\ell u + u)\right) - \delta\left((m+u)^3 \cos(\theta)\right)_\theta.
    \end{equation*}
    This function satisfies $F(0,0) = 0$ and it is continuously Fréchet differentiable with
    \begin{equation*}
    \begin{aligned}
        D_u F(u,\delta) \phi =&\, \phi_\theta + \gamma\,\div_\ell\left((m+u)^3 \nabla_\ell(\Delta_\ell \phi + \phi)\right)\\
        &+ 3\gamma\,\div_\ell\left((m+u)^2\phi\nabla_\ell(\Delta_\ell u + u)\right) - 3 \delta \left((m+u)^2 \phi \cos(\theta)\right)_\theta.
    \end{aligned}
    \end{equation*}
    At $u=0$, $\delta=0$ this linear operator becomes
    \begin{equation*}
        D_u F(0,0)\phi = \phi_\theta + \gamma m^3 \Delta_\ell(\Delta_\ell \phi + \phi)
    \end{equation*}
    with Fourier representation
    \begin{equation*}
        \left(D_u F(0,0)\phi\right)_{k,l} = \left(\ii k + \gamma m^3\bigl(k^2 + \tfrac{\pi^2}{\ell^2}l^2\bigr)\bigl(k^2 + \tfrac{\pi^2}{\ell^2}l^2 - 1\bigr)\right) \phi_{k,l}.
    \end{equation*}
    Since $\tfrac{\ell}{\pi}\not\in\Z$ the factor in front of $\phi_{k,l}$ vanishes only for $k=l=0$ and consequently $F_u(0,0):\Hevd{4}(\T^2)\to\Levd{2}(\T^2)$ is bijective with bounded inverse. Now, by the implicit function theorem, there exists a $\delta_0>0$ and a continuously differentiable curve $u\in C^1\bigl((-\delta_0,\delta_0);\Hevd{4}(\T^2)\bigr)$ such that
    \begin{equation*}
        F(u(\delta),\delta)=0
    \end{equation*}
    for all $\delta\in(-\delta_0,\delta_0)$. Defining $H_\delta\coloneq m + u(\delta)$ we have proved existence and properties (i) and (iii).

    \medskip
    {\noindent\bfseries Step 2: Independence of $\zeta$. } We show property (ii). It has already been proved in the one-dimensional case (cf. \cite{jlv2023}) that there exist solutions $H_\delta = H_\delta(\theta)$ of the one-dimensional analogue
    \begin{equation*}
        H_\theta + \gamma \big(H^3 (H_{\theta\theta\theta} + H_\theta)\big)_\theta = \delta \big(H^3 \cos(\theta)\big)_\theta,
        \quad \theta\in\T.
    \end{equation*}
    These are also solutions for the two-dimensional steady-state equation \eqref{eq:PDE_full_stationary}. By the uniqueness proved above it follows that the general solutions $H_\delta$ obtained in Step 1 coincide, in fact, with the one-dimensional solutions and hence they are independent of $\zeta$.

    \medskip
    {\noindent\bfseries Step 3: Expansion of $H_\delta$. } We plug the ansatz $H_\delta = m + \delta H_1 + \delta^2 H_2 + \Ocal(\delta^3)$ into
    \begin{equation*}
        H_{\theta} + \gamma\,\div_\ell\big(H^3\nabla_\ell (\Delta_\ell H + H)\big) - \delta\big(H^3\cos(\theta)\big)_\theta = 0,
        \quad \theta\in\T.
    \end{equation*}
    To order $\delta$ we hence obtain
    \begin{equation*}
        H_{1,\theta} + \gamma m^3 \Delta_\ell (\Delta_\ell H_1 + H_1) - \big(m^3 \cos(\theta)\big)_\theta = 0, \quad \theta\in\T,
    \end{equation*}
    with solution
    \begin{equation*}
        H_1(\theta) = m^3 \cos(\theta)
    \end{equation*}
    Then, to order $\delta^2$ we find
    \begin{equation*}
        H_{2,\theta} + \gamma m^3 \Delta_\ell (\Delta_\ell H_2 + H_2 ) + 6 m^5 \sin(\theta)\cos(\theta) = 0, \quad\theta\in\T,
    \end{equation*}
    with solution
    \begin{equation*}
        H_2(\theta) = \frac{3m^5}{2(1 + 36 \gamma^2 m^6)} \cos(2\theta) - \frac{9\gamma m^8}{1 + 36\gamma^2 m^6}\sin(2\theta).
    \end{equation*}    
\end{proof}

For the case $\tfrac{\ell}{\pi}\in\Z$ we carry out a Lyapunov--Schmidt reduction. To this end, we start by outlining the strategy and introducing the functional setting. We define a function $F:\Hevd{4}(\T^2)\times\R\to\Levd{2}(\T^2)$ by
    \begin{equation*}
        F(u,\delta) \coloneq u_\theta + \gamma \div_\ell \big((m+u)^3 \nabla_\ell (\Delta_\ell u + u)\big) - \delta \big((m+u)^3 \cos(\theta)\big)_\theta
    \end{equation*}
    such that $m + u$ is a steady state of \eqref{eq:perturbed-steady-states-equ} if and only if $F(u,\delta) = 0$. Then, we apply a Lyapunov--Schmidt reduction: Linearising $F$ around $u=0$ and $\delta=0$ we obtain the linearised operator
    \begin{equation} \label{eq:linearised_op}
        L u \coloneq u_\theta + \gamma m^3 \Delta_\ell (\Delta_\ell u + u).
    \end{equation}
    Then, $L$ has kernel $\mathrm{span}\{\cos\big(\tfrac{\ell}{\pi}\zeta\big)\}$ and its image is orthogonal to $\cos\big(\tfrac{\ell}{\pi}\zeta\big)$. This motivates the decomposition
    \begin{equation*}
    \begin{aligned}
        \Hevd{4}(\T^2) &= \mathrm{span}\big\{\cos\big(\tfrac{\ell}{\pi}\zeta\big)\big\} \oplus \big(\cos\big(\tfrac{\ell}{\pi}\zeta\big)\big)^\perp \simeq: \R\times V\\
        \Levd{2}(\T^2) &= \mathrm{span}\big\{\cos\big(\tfrac{\ell}{\pi}\zeta\big)\big\} \oplus \big(\cos\big(\tfrac{\ell}{\pi}\zeta\big)\big)^\perp \simeq: \R\times W.
    \end{aligned}
    \end{equation*}
    We denote by $P_V$ and $P_W$ the orthogonal projections onto $V$ and $W$, respectively (in the $\Levd{2}$-space).
    The idea is to solve the problem first in the infinite dimensional space $W$ by means of the implicit function theorem. Subsequently, the problem can be reduced to a single equation in the one-dimensional space $\mathrm{span}\left\{\cos\big(\tfrac{\ell}{\pi}\zeta\big)\right\}$.
    
    Indeed, we observe that a solution $u$ to $F(u,\delta) = 0$ must satisfy
    \begin{equation} \label{eq:decomposed_equ}
        P_W F(u,\delta) = 0
        \quad \text{and} \quad
        (\mathrm{Id}-P_W) F(u,\delta) = 0.
    \end{equation}
    Any solution $u$ to the first equation in \eqref{eq:decomposed_equ} can be decomposed into $u = a \cos\big(\frac{\ell}{\pi}\zeta\big) + v$ with $a \in \R$ and $v \in V$.
    Since the linearised operator $L$ in \eqref{eq:linearised_op} is a Banach space isomorphism between $V$ and $W$, we may apply the implicit function theorem to conclude that there exists $\eps_0 > 0$ such that for every $a, \delta \in (-\eps_0,\eps_0)$, the equation
    \begin{equation}\label{eq:unique_v}
        P_W F\big(a \cos\big(\tfrac{\ell}{\pi}\zeta\big) + v, \delta\big) = 0
    \end{equation}
    admits a unique solution $v=v(a,\delta)$. Here, the function $v\colon (-\eps_0,\eps_0)\times(-\eps_0,\eps_0) \to V$ is  analytic. With this decomposition, the problem $F(u, \delta) = 0$ reduces to finding $a \in (-\eps_0,\eps_0)$ such that $u = a \cos\big(\frac{\ell}{\pi}\zeta\big) + v(a,\delta)$ and $u$ satisfies the second equation of \eqref{eq:decomposed_equ}, i.e.
    \begin{equation}\label{eq:def_f}
        (\mathrm{Id} - P_W)u = a \quad\text{and}\quad f(a,\delta) \coloneq (\mathrm{Id}-P_W)F(a\cos(z)+v(a,\delta),\delta) = 0.
    \end{equation}
    In particular, for fixed $\delta \in (-\eps_0,\eps_0)$  there exists a solution $u$ to $F(u,\delta) = 0$ if and only if
    \begin{equation*}
        \exists a\in(-\eps_0,\eps_0)\text{ s.t. } f(a,\delta) = 0.
    \end{equation*}
    We will show below that, for fixed $0 < \delta < \eps_0$, this can only be true if $a=0$.

\medskip
Before we are in a position to complete the proof of Theorem \ref{thm:perturbed-steady-states}, we need to prove the following auxiliary results on the functions $v$ and $f$, introduced in \eqref{eq:unique_v} and \eqref{eq:def_f}, respectively.

\begin{lemma}\label{lem:properties-v}
    Let $v\colon (-\eps_0,\eps_0)^2 \to V$ be the function defined in defined in \eqref{eq:unique_v}. Then, for all $k\in\N_0$ we have:
    \begin{enumerate}
        \item[(i)] $\partial_a^k v(0,0) \equiv 0$,
        \item[(ii)] $\partial_\delta^k v(0,0)$ is independent of $\zeta$,
        \item[(iii)] $v(-a,\delta)(\theta,\zeta) = v(a,\delta)\big(\theta,\zeta + \tfrac{\pi^2}{\ell}\big)$,
        \item[(iv)] there exists a function $\phi = \phi(\zeta)$ depending on $\zeta$ only such that
        \begin{equation*}
            \partial_a^k\partial_\delta v(0,0)(\theta,\zeta) = \phi(\zeta)\ee^{\ii\theta} + \bar\phi(\zeta)\ee^{-\ii\theta}.
        \end{equation*}
    \end{enumerate}
\end{lemma}

\begin{proof}
    {\noindent\bfseries Property (i): }
    This is a direct consequence of the characterisation of steady states in Lemma \ref{lem:steady-states_delta=0} for $\delta = 0$. Indeed, since $m + a\cos\big(\tfrac{\ell}{\pi}\zeta\big)$ is a steady state, we have $F\big(a\cos\big(\tfrac{\ell}{\pi}\zeta\big), 0\big) = 0$ for all $a$ and hence $v(a, 0)\equiv 0$ solves equation \eqref{eq:unique_v}.

    \medskip
    {\noindent\bfseries Property (ii): }
    This follows from Theorem \ref{thm:ift_steady_states} on steady states in the case $\tfrac{\ell}{\pi}\not\in\Z$. Note that the stationary solution $H_\delta$ constructed there is independent of $\zeta$ and thus solves the steady-state equation \eqref{eq:perturbed-steady-states-equ} also in the case $\tfrac{\ell}{\pi}\in\Z$. Hence, we conclude $v(0,\delta) = H_\delta - m$, the function being in particular independent of $\zeta$.

    \medskip
    {\noindent\bfseries Property (iii): } By definition of $v$ we have
    \begin{equation*}
        P_W F\big(-a \cos\big(\tfrac{\ell}{\pi} \zeta\big) + v(-a,\delta)(\theta,\zeta),\delta\big) 
        =
        0.
    \end{equation*}
    Since $-a \cos\big(\tfrac{\ell}{\pi}\zeta\big) = a \cos\big(\tfrac{\ell}{\pi}\big(\zeta + \tfrac{\pi^2}{\ell}\big)\big)$, we find
    \begin{equation*}
        P_W F\big(a\cos\big(\tfrac{\ell}{\pi}\big(\zeta + \tfrac{\pi^2}{\ell}\big)\big) + v(-a,\delta)(\theta,\zeta),\delta\big)
        =
        0.
    \end{equation*}
    Again by the definition of $v$ we have
    \begin{equation*}
         P_W F\big(a\cos\big(\tfrac{\ell}{\pi}\big(\zeta + \tfrac{\pi^2}{\ell}\big)\big) + v(a,\delta)\big(\theta,\zeta+\tfrac{\pi^2}{\ell}\big),\delta\big)
        =
        0.
    \end{equation*}
    and we conclude 
    $v(-a,\delta)(\theta,\zeta) = v(a,\delta)\big(\theta,\zeta + \tfrac{\pi^2}{\ell}\big)$.

    \medskip
    {\noindent\bfseries Property (iv): } We prove this property by induction over $k$. For $k=0$ we differentiate
    \begin{equation*}
        P_W F\big(a \cos\big(\tfrac{\ell}{\pi}\zeta\big) + v(a,\delta), \delta\big) = 0
    \end{equation*}
    with respect to $\delta$ and obtain
    \begin{equation}\label{eq:lyapunov-schmidt-deriv-delta}
    \begin{aligned}
        P_W \Big[&\big(\partial_\delta v\big)_\theta + \gamma\,\div_\ell\Big(\big(m + a\cos\big(\tfrac{\ell}{\pi}\zeta\big) + v\big)^3\, \nabla_\ell \big(\Delta_\ell \partial_\delta v + \partial_\delta v\big)\Big)\\
        &+ 3\gamma\,\div_\ell\Big(\big(m + a\cos\big(\tfrac{\ell}{\pi}\zeta\big) + v\big)^2\, \partial_\delta v\, \nabla_\ell \big(\Delta_\ell v + v\big)\Big)\\
        &- 3\delta\Big(\big(m + a\cos\big(\tfrac{\ell}{\pi}\zeta\big) + v\big)^2\, \partial_\delta v \,\cos(\theta)\Big)_\theta\\
        &- \Big(\big(m + a\cos\big(\tfrac{\ell}{\pi}\zeta\big) + v\big)^3 \cos(\theta)\Big)_\theta\Big] = 0.
    \end{aligned}
    \end{equation}
    Then, evaluating at $a = \delta = 0$ yields for $\partial_\delta v(0,0)\in V$ the partial differential equation
    \begin{equation*}
        \big(\partial_\delta v(0,0)\big)_\theta + \gamma\, m^3 \Delta_\ell \big(\Delta_\ell \partial_\delta v(0, 0) + \partial_\delta v(0,0)\big) = - m^3 \sin(\theta)\in W
    \end{equation*}
    with unique solution
    \begin{equation}\label{eq:lyapunov-schmidt-v_delta}
        \partial_\delta v(0,0) = m^3 \cos(\theta).
    \end{equation}
    In particular, property (iv) is satisfied for $k = 0$.

    Now, let $k\in\N_0$ be fixed and we assume that property (iv) is true for all $l\leq k$, i.e.
    \begin{equation*}
        \partial_a^l \partial_\delta v(0,0) = \phi(\zeta)\ee^{\ii\theta} + \bar\phi(\zeta)\ee^{-\ii\theta},
        \quad l\leq k,
    \end{equation*}
    with the function $\phi$ possibly depending on $l$. We differentiate equation \eqref{eq:lyapunov-schmidt-deriv-delta} $(k+1)$-times with respect to $a$ and evaluate at $a = \delta = 0$. Note that the second line of \eqref{eq:lyapunov-schmidt-deriv-delta} can be omitted when differentiating with respect to $a$ since any $a$-derivatives of $\nabla_\ell (\Delta_\ell v + v)$ will always evaluate to zero by property (i). Similarly, the third line may be omitted because of the prefactor $\delta$. For $\partial_a^{k+1} \partial_\delta v(0,0)$ we thus obtain the partial differential equation
    \begin{equation}\label{eq:lyapunov-schmidt-an-delta}
    \begin{aligned}
        &\big(\partial_a^{k+1} \partial_\delta v(0,0)\big)_\theta + \gamma m^3 \Delta_\ell \big(\Delta_\ell \partial_a^{k+1} \partial_\delta v(0,0) + \partial_a^{k+1} \partial_\delta v(0,0)\big)\\
        &= -\gamma\sum_{l=0}^k \binom{k+1}{l} P_W\Big[ \div_\ell\Big(\partial_a^{k+1-l}\big(\big(m + a\cos\big(\tfrac{\ell}{\pi}\zeta\big) + v\big)^3\big) \nabla_\ell\big(\Delta_\ell \partial_a^l \partial_\delta v + \partial_a^l \partial_\delta v\big)\Big)\Big]\bigg\vert_{a=\delta=0}\\
        &+ P_W\Big[\Big(\partial_a^{k+1} \big(\big(m + a\cos\big(\tfrac{\ell}{\pi}\zeta\big) + v\big)^3\big) \cos(\theta)\Big)_\theta\Big]\bigg\vert_{a=\delta=0}.
    \end{aligned}
    \end{equation}
    Since $\partial_a^n v(0,0) = 0$ for all $n\in\N_0$ by property (i), we deduce that
    \begin{equation}\label{eq:lyapunov-schmidt-indep-theta}
        \partial_a^{n} \Big(\big(m + a\cos\big(\tfrac{\ell}{\pi}\zeta\big) + v\big)^3\Big)\bigg\vert_{a=\delta=0} \quad\text{is independent of }\theta\text{ for all }n\in\N_0.
    \end{equation}
    Thus, it follows from the induction hypothesis that the right-hand side of \eqref{eq:lyapunov-schmidt-an-delta} is of the form $\phi(\zeta)\ee^{\ii\theta} + \bar\phi(\zeta)\ee^{-\ii\theta}$ and we conclude the same for the solution $\partial_a^{k+1} \partial_\delta v(0,0)$.
\end{proof}

From the properties of the function $v$ we deduce the following properties for the function $f$, defined in \eqref{eq:def_f}, describing the one-dimensional equation arising in the Lyapunov--Schmidt reduction.

\begin{corollary}\label{cor:properties-f}
    Let $f\colon (-\eps_0,\eps_0)^2 \to \R$ be given by $f(a,\delta) \coloneq (\mathrm{Id}-P_W)F(a\cos(z)+v(a,\delta),\delta) = 0$, as defined in \eqref{eq:def_f}. Then, for all $k, n\in\N_0$ we have:
    \begin{enumerate}
        \item[(i)] $\partial_a^k f(0,0) = 0$,
        \item[(ii)] $\partial_\delta^k f(0,0) = 0$,
        \item[(iii)] $\partial_a^{2n}\partial_\delta^k f(0,0) = 0$,
        \item[(iv)] $\partial_a^k\partial_\delta f(0,0)=0$.
    \end{enumerate}
\end{corollary}

\begin{proof}
Recall that the function $F\colon \Hevd{4}(\T^2)\times\R\to\Levd{2}(\T^2)$ is defined by
\begin{equation*}
    F(u,\delta) \coloneq
    u_\theta + \gamma \div_\ell \big((m+u)^3 \nabla_\ell (\Delta_\ell u + u)\big) - \delta \big((m+u)^3 \cos(\theta)\big)_\theta.
\end{equation*}
Since $\theta$-derivatives of a function are always orthogonal to $\cos\big(\tfrac{\ell}{\pi}\zeta\big)$, the definition of $f$ simplifies to
\begin{equation}\label{eq:simplified-def-f}
    f(a,\delta) = \frac{\gamma\pi^2}{\ell^2} (\mathrm{Id}-P_W)\Big[\Big(\big(m + a\cos\big(\tfrac{\ell}{\pi}\zeta\big) + v(a,\delta)\big)^3 \big(\Delta_\ell v(a,\delta) + v(a,\delta)\big)_\zeta\Big)_\zeta\Big].
\end{equation}

\medskip
{\noindent\bfseries Property (i): }
From the proof of Lemma \ref{lem:properties-v} (i) we know that $v(a,0) = 0$ for all $a \in (-\eps_0,\eps_0)$. Inserting this into \eqref{eq:simplified-def-f}, we conclude $f(a,0) = 0$ for all $a\in(-\eps_0,\eps_0)$.

\medskip
{\noindent\bfseries Property (ii): }
In the proof of Lemma \ref{lem:properties-v} (ii) we have observed that $v(0,\delta)$ is independent of $\zeta$ for all $\delta\in(-\eps_0,\eps_0)$. As a direct consequence of this, we conclude $\big(\Delta_\ell v(0,\delta) + v(0,\delta)\big)_\zeta = 0$ and thus $f(0,\delta) = 0$ for all $\delta\in(-\eps_0,\eps_0)$.

\medskip
{\noindent\bfseries Property (iii): }
We show that $f$ is an odd function in $a$. Indeed, using Lemma \ref{lem:properties-v} (iii), we find that
\begin{equation*}
\begin{aligned}
    f(-a,\delta) &= \frac{\gamma\pi^2}{\ell^2} (\mathrm{Id}-P_W)\Big[\Big(\big(m - a\cos\big(\tfrac{\ell}{\pi}\zeta\big) + v(-a,\delta)\big)^3 \big(\Delta_\ell v(-a,\delta) + v(-a,\delta)\big)_\zeta\Big)_\zeta\Big]\\
    &= \frac{\gamma\pi^2}{\ell^2} (\mathrm{Id}-P_W)\Big[\Big(\big(m + a \cos\big(\tfrac{\ell}{\pi}\big(\zeta + \tfrac{\pi^2}{\ell}\big)\big) + v(a,\delta)\big(\theta,\zeta + \tfrac{\pi^2}{\ell}\big)\big)^3\\
    &\qquad\qquad\qquad\quad\ \big(\Delta_\ell v(a,\delta)\big(\theta,\zeta + \tfrac{\pi^2}{\ell}\big) + v(a,\delta)\big(\theta,\zeta + \tfrac{\pi^2}{\ell}\big)\big)_\zeta\Big)_\zeta\Big]\\
    &= -\frac{\gamma\pi^2}{\ell^2} (\mathrm{Id}-P_W)\Big[\Big(\big(m + a\cos\big(\tfrac{\ell}{\pi}\zeta\big) + v(a,\delta)\big)^3 \big(\Delta_\ell v(a,\delta) + v(a,\delta)\big)_\zeta\Big)_\zeta\Big]\\
    &= -f(a,\delta).
\end{aligned}
\end{equation*}
In particular, we deduce $\partial_a^{2n}\partial_\delta^k f(0,0) = -\partial_a^{2n}\partial_\delta^k f(0,0)$, from where we can conclude property (iii).

\medskip
{\noindent\bfseries Property (iv): }
We differentiate equation \eqref{eq:simplified-def-f} once with respect to $\delta$ and obtain
\begin{equation}\label{eq:deriv_f_delta}
\begin{split}
    \partial_\delta f(a,\delta) &= \frac{\gamma\pi^2}{\ell^2} (\mathrm{Id}-P_W)\Big[\Big(\big(m + a\cos\big(\tfrac{\ell}{\pi}\zeta\big) + v\big)^3 \big(\Delta_\ell \partial_\delta v + \partial_\delta v\big)_\zeta\Big)_\zeta\Big]\\
    &+ \frac{3\gamma\pi^2}{\ell^2} (\mathrm{Id}-P_W)\Big[\Big(\big(m + a\cos\big(\tfrac{\ell}{\pi}\zeta\big) + v\big)^2 \,\partial_\delta v\, \big(\Delta_\ell v + v\big)_\zeta\Big)_\zeta\Big].
\end{split}
\end{equation}
If we now differentiate $\partial_\delta f(a,\delta)$ $k$-times with respect to $a$ and then evaluate at $a=\delta=0$, we can omit the second line of \eqref{eq:deriv_f_delta} since any $a$-derivative of $\Delta_\ell v + v$ will evaluate to zero by Lemma \ref{lem:properties-v} (i). Hence, we find that
\begin{equation*}
    \partial_a^k \partial_\delta f(0,0) = \frac{\gamma\pi^2}{\ell^2} \sum_{l=0}^k \binom{k}{l} (\mathrm{Id}-P_W)\Big[\Big(\partial_a^l\Big(\big(m + a\cos\big(\tfrac{\ell}{\pi}\zeta\big) + v\big)^3\Big) \big(\Delta_\ell \partial_a^{k-l}\partial_\delta v + \partial_a^{k-l}\partial_\delta v\big)_\zeta\Big)_\zeta\Big]\bigg\vert_{a=\delta=0}.
\end{equation*}
Finally, we can conclude from \eqref{eq:lyapunov-schmidt-indep-theta} and Lemma \ref{lem:properties-v} (iv) that
\begin{equation*}
    \partial_a^k \partial_\delta f(0,0) = \frac{\gamma\pi^2}{\ell^2} (\mathrm{Id}-P_W) \big[\phi(\zeta)\ee^{\ii\theta} + \bar\phi(\zeta)\ee^{-\ii\theta}\big] = 0.
\end{equation*}
\end{proof}

With the properties obtained in Corollary \ref{cor:properties-f}, we are finally \textcolor{teal}{in a position} to complete the proof of Theorem \ref{thm:perturbed-steady-states}:

\begin{proof}[Proof of Theorem \ref{thm:perturbed-steady-states} in the case $\frac{\ell}{\pi} \in \Z$]
    We claim that there exists a $\delta_0\leq \eps_0$ such that $f$ has the form
    \begin{equation}\label{eq:expansion-f}
        f(a,\delta) = c a \delta^2 + g(a,\delta) \quad\text{ for all }a,\delta\in(-\delta_0,\delta_0)
    \end{equation}
    where $c>0$ and $\abs{g(a,\delta)}\leq C_1 \abs{a \delta^3} + C_2 \abs{a^3 \delta^2}$.
    For a fixed $\delta > 0$ we assume for a contradiction that there exists a solution to the steady-state equation \eqref{eq:steady-states} other than $H_\delta$ that is still a perturbation of the constant solution $H\equiv m$ for $\delta = 0$. Then, there must exist an $a\in(-\delta_0,\delta_0)\setminus\{0\}$ such that $f(a,\delta) = 0$. Indeed, by the implicit function theorem, $a = 0$ would necessarily imply $u = v(0,\delta) = H_\delta - m$.

    Hence, dividing by $a\delta^2$ we find for such $a$ and $\delta$:
    \begin{equation*}
        0 = \abslr{\frac{f(a,\delta)}{a\delta^2}} = \abslr{c + \frac{g(a,\delta)}{a\delta^2}} \geq c - \abslr{\frac{g(a,\delta)}{a\delta^2}} \geq c - C_1\delta - C_2 a^2.
    \end{equation*}
    This is a contradiction if $\delta_0$ is chosen small enough.

    It remains to prove the expansion \eqref{eq:expansion-f}. To this end, we first note that $f$ is analytic by construction and due to Corollary \ref{cor:properties-f} (i), (ii) and (iv), no constant, linear or quadratic terms appear in the Taylor expansion of $f$ about $a=\delta = 0$. Similarly, to third order we have
    \begin{equation*}
        \partial_a^3 f(0,0) = \partial_a^2 \partial_\delta f(0,0) = \partial_\delta^3 f(0,0) = 0.
    \end{equation*}
    Again, by Corollary \ref{cor:properties-f} (i) and (iv), any further terms in the Taylor expansion of $f$ are at least quadratic in $\delta$, and since we have $\partial_a^2\partial_\delta^2 f(0,0) = 0$ by Corollary \ref{cor:properties-f} (iii), the largest such term is $a^3 \delta^2$. On the other hand, due to Corollary \ref{cor:properties-f} (ii), every term is at least linear in $a$ and all such terms to fourth order and above are bounded by $a\delta^3$. Thus, the expansion \eqref{eq:expansion-f} is satisfied with $c = \tfrac{1}{2}\partial_a\partial_\delta^2 f(0,0)$.

    We differentiate $f(a,\delta)$ from \eqref{eq:simplified-def-f} once with respect to $a$ and twice with respect to $\delta$ to obtain
    \begin{equation*}
    \begin{aligned}
        \partial_a \partial_\delta^2 f(a,\delta)
        &= \frac{\gamma\pi^2}{\ell^2} (\mathrm{Id}-P_W)\Big[\Big(\partial_a\partial_\delta^2\Big(\big(m + a\cos\big(\tfrac{\ell}{\pi}\zeta\big) + v(a,\delta)\big)^3\Big) \big(\Delta_\ell v(a,\delta) + v(a,\delta)\big)_\zeta\Big)_\zeta\Big]\\
        &+ \frac{2\gamma\pi^2}{\ell^2} (\mathrm{Id}-P_W)\Big[\Big(\partial_a\partial_\delta\Big(\big(m + a\cos\big(\tfrac{\ell}{\pi}\zeta\big) + v(a,\delta)\big)^3\Big) \big(\Delta_\ell \partial_\delta v(a,\delta) + \partial_\delta v(a,\delta)\big)_\zeta\Big)_\zeta\Big]\\
        &+ \frac{\gamma\pi^2}{\ell^2} (\mathrm{Id}-P_W)\Big[\Big(\partial_\delta^2\Big(\big(m + a\cos\big(\tfrac{\ell}{\pi}\zeta\big) + v(a,\delta)\big)^3\Big) \big(\Delta_\ell\partial_a v(a,\delta) + \partial_a v(a,\delta)\big)_\zeta\Big)_\zeta\Big]\\
        &+ \frac{\gamma\pi^2}{\ell^2} (\mathrm{Id}-P_W)\Big[\Big(\partial_a\Big(\big(m + a\cos\big(\tfrac{\ell}{\pi}\zeta\big) + v(a,\delta)\big)^3\Big) \big(\Delta_\ell\partial_\delta^2 v(a,\delta) + \partial_\delta^2 v(a,\delta)\big)_\zeta\Big)_\zeta\Big]\\
        &+ \frac{2\gamma\pi^2}{\ell^2} (\mathrm{Id}-P_W)\Big[\Big(\partial_\delta\Big(\big(m + a\cos\big(\tfrac{\ell}{\pi}\zeta\big) + v(a,\delta)\big)^3\Big) \big(\Delta_\ell\partial_a\partial_\delta v(a,\delta) + \partial_a\partial_\delta v(a,\delta)\big)_\zeta\Big)_\zeta\Big]\\
        &+ \frac{\gamma\pi^2}{\ell^2} (\mathrm{Id}-P_W)\Big[\Big(\big(m + a\cos\big(\tfrac{\ell}{\pi}\zeta\big) + v(a,\delta)\big)^3 \big(\Delta_\ell \partial_a\partial_\delta^2 v(a,\delta) + \partial_a\partial_\delta^2 v(a,\delta)\big)_\zeta\Big)_\zeta\Big].
    \end{aligned}
    \end{equation*}
    By Lemma \ref{lem:properties-v} the first four lines evaluate to zero in $a = \delta = 0$ and the last line evaluates to zero, too, because the image of $m^3 \partial_\zeta^2 (\Delta_\ell + \mathrm{Id})$ is orthogonal to $\cos\big(\tfrac{\ell}{\pi}\zeta\big)$. Thus, we find
    \begin{equation}\label{eq:f_a_delta_delta}
        \partial_a \partial_\delta^2 f(0,0) = \frac{6\gamma m^2 \pi^2}{\ell^2} (\mathrm{Id} - P_W) \Big[\Big(\partial_\delta v(0,0) \big(\Delta_\ell \partial_a \partial_\delta v(0,0) + \partial_a \partial_\delta v(0,0)\big)_{\zeta}\Big)_\zeta\Big].
    \end{equation}
    We already know that $\partial_\delta v(0,0) = m^3\cos(\theta)$ from \eqref{eq:lyapunov-schmidt-v_delta}. Differentiating equation \eqref{eq:lyapunov-schmidt-deriv-delta} with respect to $a$ and evaluating at $a = \delta = 0$ we obtain the partial differential equation
    \begin{equation*}
        \big(\partial_a\partial_\delta v(0,0)\big)_\theta + \gamma\, m^3 \Delta_\ell \big(\Delta_\ell \partial_a\partial_\delta v(0, 0) + \partial_a\partial_\delta v(0,0)\big) = -3m^2 \cos\big(\tfrac{\ell}{\pi}\zeta\big) \sin(\theta) \in W.
    \end{equation*}
    This equation has the unique solution
    \begin{equation}\label{eq:lyapunov-schmidt-v_a_delta}
        \partial_a \partial_\delta v(0,0)
        =
        \frac{3m^2 + 6 \ii \gamma m^5}{2 + 8 \gamma^2 m^6} \ee^{\ii \theta} 
        \cos\big(\tfrac{\ell}{\pi} \zeta\big)
        +
        \frac{3m^2 - 6 \ii \gamma m^5}{2 + 8 \gamma^2 m^6} \ee^{-\ii \theta} 
        \cos\big(\tfrac{\ell}{\pi} \zeta\big)\in V.
    \end{equation}
    Finally, we insert \eqref{eq:lyapunov-schmidt-v_delta} and \eqref{eq:lyapunov-schmidt-v_a_delta} into \eqref{eq:f_a_delta_delta} and obtain
    \begin{equation*}
        \partial_a \partial_\delta^2 f(0,0) = \frac{\pi^2}{\ell^2}\cdot\frac{9\gamma m^7}{1 + 4 \gamma^2 m^6} > 0.
    \end{equation*}
    The properties (ii)--(iv) have already been proved in the case $\tfrac{\ell}{\pi}\not\in\Z$.
\end{proof}

We now analyse the stability properties of the perturbed steady states obtained in Theorem \ref{thm:perturbed-steady-states}. To this end, given some $\delta > 0$ we linearise equation \eqref{eq:PDE_full_stationary} around $H_\delta$. Indeed, we write
\begin{equation} \label{eq:linearised_u_delta}
    u_t = L_\delta u + R_\delta(u),
\end{equation}
where $L_\delta:\Hevd{4}(\T^2)\to\Levd{2}(\T^2)$ is defined by
\begin{equation*}
    L_\delta u = -u_\theta - \gamma\,\div_\ell\left(H_\delta^3 \nabla_\ell(\Delta_\ell u + u)\right) - 3 \gamma\,\div_\ell \left(H_\delta^2 u \nabla_\ell(\Delta_\ell H_\delta + H_\delta)\right) + 3 \delta\left(H_\delta^2 u \cos(\theta)\right)_\theta
\end{equation*}
and
\begin{equation*}
    \begin{aligned}
        R_\delta(u) =& -\gamma\,\div_\ell\left((3 H_\delta u^2 + u^3) \nabla_\ell(\Delta_\ell H_\delta + H_\delta)\right)\\
        &- \gamma\,\div_\ell\left((3H_\delta^2 u + 3 H_\delta u^2 + u^3) \nabla_\ell(\Delta_\ell u + u)\right) + \delta\left((3 H_\delta u^2 + u^3) \cos(\theta)\right)_\theta
    \end{aligned}
\end{equation*}
satisfies $R_\delta(0) = 0$ and $R_\delta'(0)=0$. 

\begin{theorem}[Exponential stability of $H_\delta$ for $\ell < \pi$]\label{thm:exp-stability-perturbed-steady-state}
    Let $\ell<\pi$ and $H_\delta$ be the stationary solutions obtained in Theorem \ref{thm:ift_steady_states}. Then, there exists a $\delta_0>0$ such that for all $0<\delta<\delta_0$ the stationary solution $H_\delta$ is exponentially stable: There exist constants $\eps,C,\omega>0$ (depending on $\delta$) such that for all initial values $H_0$ with $\tfrac{1}{4\pi^2}\int_{\T^2} H_0\dd\theta\dd\zeta = m$ and
    \begin{equation*}
        \normb{H_0-H_\delta}_{\Hev{4}} < \eps
    \end{equation*}
    the full rimming-flow equation \eqref{eq:PDE_full_stationary} admits a unique global solution $H$ which satisfies
    \begin{equation*}
        \normb{H(t)-H_\delta}_{\Hev{4}} \leq C \eps \ee^{-\omega t}
    \end{equation*}
    for all $t>0$.
\end{theorem}

\begin{proof}
    Since $H_\delta$ is independent of $\zeta$ by Theorem \ref{thm:ift_steady_states} (ii), the proof is similar to the 1D case \cite{karabash_multi-parameter_2024}. We apply \cite[Theorem 9.1.2]{lunardi_1995}. The operator $L_\delta$ from \eqref{eq:linearised_u_delta} may be viewed as a lower-order perturbation of $-\gamma H_\delta^3 \Delta_\ell^2$ and thus by Proposition \ref{prop:generator} and \cite[Proposition 2.4.1]{lunardi_1995}, $L_\delta$ is sectorial. Moreover, the graph norm of $L_\delta$ is equivalent to the $\Hevd{4}$ norm.

    The spectrum of $L_0$ is given by
    \begin{equation*}
        \sigma(L_0) = \set{\lambda_{k,l} = -\ii k - \gamma m^3 \Big(k^2 + \tfrac{\pi^2}{\ell^2} l^2\Big)\Big(k^2 + \tfrac{\pi^2}{\ell^2}l^2 - 1\Big)}{k, l\in\Z},
    \end{equation*}
    consisting only of isolated simple eigenvalues $\lambda_{k,l}$ with eigenfunctions $\ee^{\ii k\theta}\cos(l\zeta).$
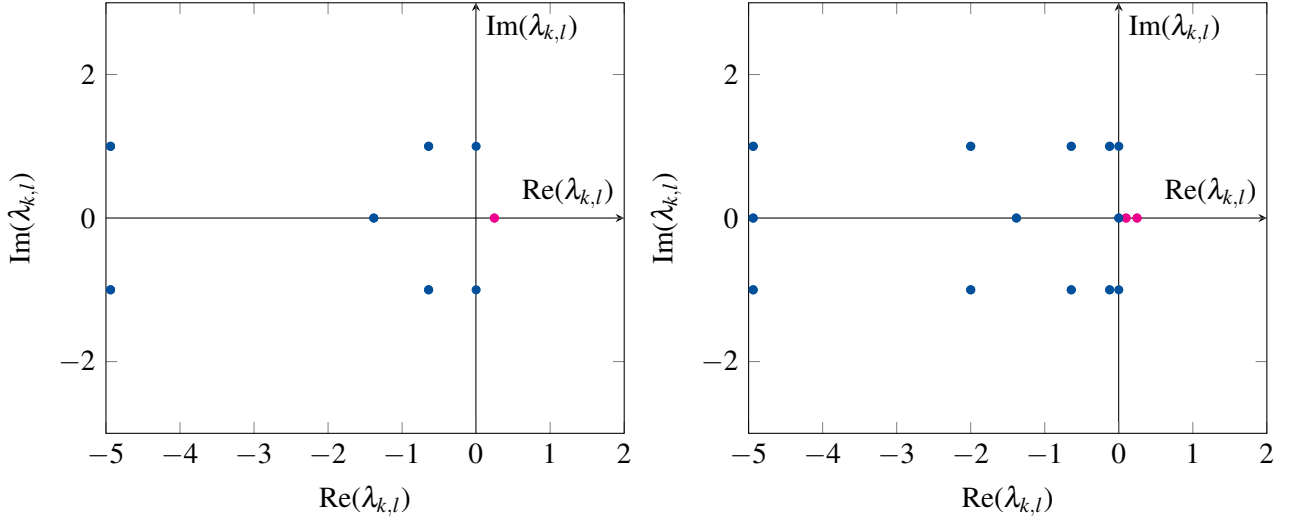
\begin{figure}[h!]
    \centering
    \begin{minipage}[t]{0.5\textwidth}
        \centering
        \begin{tikzpicture}
            \begin{axis}[
                xlabel = {Re($\lambda_{k,l}$)},
                ylabel = {Im($\lambda_{k,l}$)},
                xmin = -5, xmax = 2,
                ymin = -3, ymax = 3,
                axis x line=middle,
                axis y line=middle,
                xtick=\empty,
                ytick=\empty,
                xticklabels={},
                yticklabels={},
            ]
            \end{axis}
            \begin{axis}[
              xlabel = {Re($\lambda_{k,l}$)},
              ylabel = {Im($\lambda_{k,l}$)},
              xmin = -5, xmax = 2,
              ymin = -3, ymax = 3,
            ]
            \pgfmathsetmacro{\ell}{1.5*pi}
            \foreach \k in {-3,...,3} { 
              \foreach \l in {-4,...,4} { 
                \ifthenelse{\k=0 \AND \l=0}{
                }{
                    \pgfmathsetmacro{\xcoord}{-(pow(\k,2) + pow(\l*pi/\ell,2))*(pow(\k,2) + pow(\l*pi/\ell,2) - 1)}
                    \pgfmathsetmacro{\ycoord}{-\k}
                    \pgfmathparse{\xcoord>0}\ifnum\pgfmathresult=1\relax
                        \addplot[only marks, mark=*, mark size=1.5pt, magenta]
                            coordinates {(\xcoord, \ycoord)};
                    \else
                        \addplot[only marks, mark=*, mark size=1.5pt, luh-dark-blue]
                            coordinates {(\xcoord, \ycoord)};
                    \fi
                }
              }
            }
            \end{axis}
        \end{tikzpicture}
    \end{minipage}%
    \begin{minipage}[t]{0.5\textwidth}
    \centering
        \begin{tikzpicture}
            \begin{axis}[
                xlabel = {Re($\lambda_{k,l}$)},
                ylabel = {Im($\lambda_{k,l}$)},
                xmin = -5, xmax = 2,
                ymin = -3, ymax = 3,
                axis x line=middle,
                axis y line=middle,
                xtick=\empty,
                ytick=\empty,
                xticklabels={},
                yticklabels={},
            ]
            \end{axis}
            \begin{axis}[
              xlabel = {Re($\lambda_{k,l}$)},
              ylabel = {Im($\lambda_{k,l}$)},
              xmin = -5, xmax = 2,
              ymin = -3, ymax = 3,
            ]
            \pgfmathsetmacro{\ell}{3*pi}
            \foreach \k in {-3,...,3} { 
              \foreach \l in {-5,...,5} { 
                \ifthenelse{\k=0 \AND \l=0}{
                }{
                    \pgfmathsetmacro{\xcoord}{-(pow(\k,2) + pow(\l*pi/\ell,2))*(pow(\k,2) + pow(\l*pi/\ell,2) - 1)}
                    \pgfmathsetmacro{\ycoord}{-\k}
                    \pgfmathparse{\xcoord>0}\ifnum\pgfmathresult=1\relax
                        \addplot[only marks, mark=*, mark size=1.5pt, magenta]
                            coordinates {(\xcoord, \ycoord)};
                    \else
                        \addplot[only marks, mark=*, mark size=1.5pt, luh-dark-blue]
                            coordinates {(\xcoord, \ycoord)};
                    \fi
                }
              }
            }
            \end{axis}
        \end{tikzpicture}
    \end{minipage}%
    \caption{Eigenvalues with real part $\mathrm{Re}(\lambda_{k,l})\geq -5$ of $L_0$ for length $\ell = \tfrac{3}{2}\pi$ (left) and $\ell=3\pi$ (right). With increasing $\ell$ more and more unstable eigenvalues appear.}
    \label{fig:long-cylinder-instability}
\end{figure}
    
    Thus, $F:\R\times\big(\C\times\Hevd{4}(\T^2)\big)\to\Levd{2}(\T^2)$ defined by
    \begin{equation*}
        F (\delta, \lambda, u) \coloneq L_\delta u - \lambda u
    \end{equation*}
    is analytic and we have $F\big(0, \lambda_{k,l}, \ee^{\ii k\theta}\cos(l \zeta)\big) = 0$. By the implicit function theorem, there are analytic eigenvalue curves $\lambda_{k,l}(\delta)$ with eigenfunctions $u_{k,l}(\delta)$ also depending analytically on $\delta$.
    
    We observe that for $\delta = 0$ all eigenvalues are stable and bounded away from the imaginary axis except for the critical complex-conjugated eigenvalue pair $\lambda = \pm \ii$ with eigenmodes $u = \ee^{\mp \ii\theta}$. This property is conserved for small $\delta$. Indeed, since the kernel of $L_0$ is finite-dimensional, the perturbation $L_\delta - L_0$ is relatively bounded with respect to $L_0$, i.e.
    \begin{equation*}
        \norm{(L_\delta - L_0)u}_{L_2} \leq C_\delta\norm{u}_{L_2} + B_\delta\norm{L_0 u}_{L_2}
    \end{equation*}
    and the bound $B_\delta$ can be made arbitrarily small for small $\delta$. Since $L_0$ is closed it follows from \cite[Theorem IV-1.1]{kato_perturbation_1995} that $L_\delta$ is closed as well for $\delta$ small enough. Thus, by \cite[Theorem VII-1.7, Remark VII-2.9]{kato_perturbation_1995}, except for the two perturbed critical eigenvalues the spectrum of $L_\delta$ is stable and bounded away from the imaginary axis for small $\delta$.
    
    Without loss of generality we consider the only the critical eigenvalue $+\ii$ and expand it (and its eigenmode $\ee^{-\ii\theta}$) in $\delta$ around $\delta = 0$ (cf. also \cite{karabash_multi-parameter_2024}):
    \begin{equation*}
    \begin{aligned}
        \lambda &= \ii + \delta\lambda_1 + \delta^2 \lambda_2 + \Ocal(\delta^3),\\
        u &= \ee^{-\ii\theta} + \delta u_1 + \delta^2 u_2 + \Ocal(\delta^3).
    \end{aligned}
    \end{equation*}
    To first order in $\delta$ the eigenvalue equation $L_\delta u = \lambda u$ reads
    \begin{equation*}
        \lambda_1\ee^{-\ii\theta} + \ii u_1 + u_{1,\theta} + \gamma m^3 \Delta_\ell (\Delta_\ell u_1 + u_1) + 3\ii m^2\ee^{-2\ii\theta} = 0.
    \end{equation*}
    Here, we have used the expansion of the perturbed steady states $H_\delta$ from Theorem \ref{thm:perturbed-steady-states}.
    Projecting this equation onto $\dpr{\ee^{-\ii\theta}}$ we immediately obtain $\lambda_1 = 0$, hence we must proceed expanding to second order. Solving subsequently for $u_1$ we find
    \begin{equation*}
        u_1 = \frac{3m^2 - 36\gamma m^5\ii}{1 + 144\gamma^2 m^6}\ee^{-2\ii\theta}.
    \end{equation*}
    Since we are only interested in the Taylor coefficient $\lambda_2$, we do not compute the full second-order eigenvalue equation. Instead, we only consider its projection onto $\dpr{\ee^{-\ii\theta}}$ and obtain
    \begin{equation*}
        \lambda_2 = -\frac{81\gamma m^7}{1 + 144\gamma^2 m^6} - \ii \frac{108\gamma^2 m^{10} + \tfrac{15}{2}m^4}{1 + 144\gamma^2 m^6}.
    \end{equation*}
    Thus, it follows that all eigenvalues have negative real part in a neighbourhood of $\delta = 0$ when $\delta\neq 0$ and hence we conclude that the steady states $H_\delta$ are exponentially stable.
\end{proof}

\begin{remark}
    The spectral gap computed in the proof of Theorem \ref{thm:exp-stability-perturbed-steady-state} converges to zero as $\delta\to 0$. Thus, the exponential stability of the corresponding steady states $H_\delta$ for $0 < \delta\ll 1$ is non-uniform in $\delta$ and the constants $\eps$, $C$, $\omega$ may depend on $\delta$.
\end{remark}

For \emph{long} cylinders (that is, if $\ell > \pi$) we have already seen that in the absence of gravity ($\delta = 0$) the constant steady states are unstable (cf. Proposition \ref{prop:const-unstable-large-ell}). Of course, this instability persists for the perturbed steady states $H_\delta$ if $\delta$ is small enough:

\begin{proposition}[Instability of $H_\delta$ for $\ell > \pi$] \label{thm:instability_H_delta}
Let $\tfrac{\ell}{\pi}\not\in\Z$, $\ell>\pi$. Then, there exist constants $C, \omega > 0$, depending on $\ell$, such that equation \eqref{eq:linearised_u_delta} admits a non-trivial backward solution $u \in C\bigl((-\infty,0];\Hevd{4}(\T^2)\bigr) \cap C^1\bigl((-\infty,0];\Levd{2}(\T^2)\bigr)$ satisfying
\begin{equation*}
    \|u(t)\|_{H^4} \leq C \ee^{\omega t},
\end{equation*}
for all $t < 0$.
\end{proposition}

\begin{proof}
    Again, we apply \cite[Theorem 9.1.3]{lunardi_1995}. Note that in the case $\ell>\pi$, $L_0$ has got (a finite number of) unstable eigenvalues. We have seen in the proof of Theorem \ref{thm:exp-stability-perturbed-steady-state} that they are a point on an analytic eigenvalue curve. Thus, $L_\delta$ has unstable eigenvalues if $\delta$ is small enough. All other assumptions have already been checked in the proof of Theorem \ref{thm:exp-stability-perturbed-steady-state}.
\end{proof}


\section{Two-time scale dynamics of positive solutions for small gravity} \label{sec:slow_manifold}

In this section, we examine the long-time behaviour of solutions to the rimming-flow equation for the special case in which the length of the cylinder is exactly $\ell = \pi$.
In this case the differential operators $\div_\ell$, $\nabla_\ell$ and $\Delta_\ell$ reduce to the usual divergence, gradient and Laplace operators.

Though on a more formal level, we generalise the results of the cross-sectional case treated in \cite{jlv2023} to the three-dimensional cylinder and the corresponding rimming-flow equation
\begin{equation} \label{eq:PDE_on_two_time_scales}
    \begin{cases}
    h_t + h_\theta + \gamma\, \Acal(h^3)h  
    =
    \delta \left(h^3 \cos(\theta)\right)_\theta,
    &
    t > 0,\, \theta \in \T,\, \zeta \in \T
    \\
    h(0) = h_0, & \zeta \in \T
    \end{cases}
\end{equation}
where the operator $\Acal$ is defined by
\begin{equation*}
    \Acal(u)v = \div \big(u \nabla (\Delta v + v)\big).
\end{equation*}
We investigate this equation near the bifurcation at $\delta=0$ by means of the Poincaré--Lindstedt method of multiple time scales \cite{drazin_nonlinear_1992,HI2011}. More precisely, we show that solutions exhibit dynamics at the \emph{fast time scale} $t$ as well as at the \emph{slow time scale} $\tau\coloneq \delta^2t$. 
For this purpose, we eliminate the transport term $h_\theta$ from the equation by considering a rotating coordinate system the angular velocity of which equals the angular velocity of the cylinder. This corresponds to making the travelling-wave ansatz $\vartheta = \theta-t$ and looking for solutions
\begin{equation*}
    H(t,\vartheta,\zeta) = h(t,\theta,\zeta),
    \quad
    t > 0,\ \vartheta \in \T,\, \zeta \in \T.
\end{equation*}
Plugging this ansatz into the equation \eqref{eq:PDE_on_two_time_scales} we obtain the transformed evolution equation
\begin{equation} \label{eq:PDE_transformed}
    H_t + \gamma\, \div\left(H^3 \nabla (\Delta H + H)\right) 
    =
    \delta \left(H^3 \cos(\vartheta + t)\right)_\vartheta,
    \quad t > 0,\, \vartheta \in \T,\, \zeta \in \T,
\end{equation}
with initial condition $H(0,\vartheta,\zeta)=h(0,\theta,\zeta)=h_0(\theta,\zeta)$.

Note that \eqref{eq:PDE_transformed} is a non-autonomous equation as the term $\cos(\xi+t)$ on the right-hand side depends on both variables $\vartheta$ and $t$.
However, this transformation does neither affect the fluid mass, which we continue to denote by
\begin{equation*}
    m \coloneq
    \int_{\T^2} h_0 \dd\vartheta \dd\zeta
    =
    \int_{\T^2} H(t,\vartheta, \zeta)\dd\vartheta \dd\zeta
\end{equation*}
nor the energy (cf. Theorem \ref{thm:local_existence})
\begin{equation*}
    E[H] = \frac 12\int_{\T^2} |\nabla H|^2 - |H|^2 \dd\vartheta\dd\zeta + \frac{m^2}{2}.
\end{equation*}
Using the energy dissipation equation from the proof of Theorem \ref{thm:local_existence}, namely
\begin{equation*}
    \frac{\dd}{\dd t} E[H(t)] = -\gamma \int_{\T^2} H^3 \absb{\nabla (\Delta H + H)}^2 \dd\theta\dd\zeta,
\end{equation*}
we can easily classify the positive steady states in the case $\delta = 0$ for the transformed equation \eqref{eq:PDE_transformed}:

\begin{lemma}
    Fix a mass $m > 0$. The set of positive steady states of equation\eqref{eq:PDE_transformed} for $\delta = 0$ is a three-dimensional manifold
    \begin{equation*}
        \Mcal(m)\coloneq
        \set{m + a_{-1}e^{-\ii\vartheta} + a_1 e^{\ii\vartheta} + b \cos(\zeta)}{{a_{\pm 1} \in\Cbb, b\in\R,\ a_{-1}=\overline{a_{1}},\ \abs{a_{-1}}+\abs{a_{1}}+|b|<m}}.
    \end{equation*}
\end{lemma}

\begin{proof}
    For steady states the energy dissipation must of course vanish. Since $H$ is assumed to be positive we find $\nabla(\Delta H + H) = 0$. The result follows from its Fourier representation.
\end{proof}

Note that steady states of \eqref{eq:PDE_transformed} are travelling waves of the original rimming-flow equation \eqref{eq:PDE_on_two_time_scales} which have a propagation speed exactly matching the cylinder rotation.

\begin{remark}
    In the case $\ell < \pi$ of a \emph{short cylinder} the eigenmode $\cos(\zeta)$ is stable and the coefficient $b$ becomes of order $\delta$ in finite time. Thus, the problem reduces to the cross-sectional case with only coefficients $a_1$ and $a_{-1}$. See Section \ref{sec:numerics} for numerical simulations and \cite{jlv2023} for a rigorous result in this case.
\end{remark}


\subsection{Convergence to the manifold $\Mcal(m)$} \label{sec:convergence_to_manifold}
We show that solutions of \eqref{eq:PDE_transformed} with mass $m$ that emerge from a positive initial value and are bounded away from zero, converge exponentially fast to a neighbourhood of the scale $\delta$ of the manifold $\Mcal(m)$
consisting exactly of the steady states of \eqref{eq:PDE_transformed} in the case $\delta=0$. Henceforth, we denote by $P_1 : \Lev{2}(\T^2) \to \Mcal(0)$ the projection onto $\dpr{\ee^{\pm\ii\vartheta}, \cos{\zeta}}$ such that $m + P_1 h\in\Mcal(m)$ is the closest point on $\Mcal(m)$ to the function $h\in\Lev{2}(\T^2)$ with mass $m$. By $P_{\geq 2}$ on the other hand we denote the projection onto the subspace orthogonal to $\Mcal(m)$, i.e. $P_{\geq 2} h = h - m - P_1 h$ for $h$ with mass $m$.

Recall that we have the energy dissipation formula 
\begin{equation}\label{eq:energy-dissipation-delta-positive}
    \frac{\dd}{\dd t} E[H]
    =
    -
    \gamma\, \int_{\Tbb^2} H^3 |\nabla(\Delta H + H)|^2 \dd\vartheta \dd \zeta
    -
    \delta \int_{\T^2} \bigl(H^3 \cos(\vartheta + t)\bigr)_\vartheta (\Delta H + H) \dd \vartheta \dd \zeta.
\end{equation}

The idea of the proof is to estimate the solution's distance to the manifold by the energy. 

\begin{theorem}[Convergence to manifold]\label{thm:convergence-to-manifold}
    Let $H$ be a solution to the transformed rimming-flow equation \eqref{eq:PDE_transformed} with initial value $H(0)=h_0$ and suppose $H$ remains bounded away from zero on some positive time interval $[0,T]$, i.e.
    \begin{equation*}
        H(t,\vartheta,\zeta)\geq c_0>0, \quad\text{for all }0\leq t\leq T,\ \vartheta,\zeta\in\T^2.
    \end{equation*}
    Furthermore, suppose $E[h_0]>0$ (otherwise we find immediately $h_0\in\mathcal{M}$). Then, there exist constants $\delta_0>0$ and $C>0$, both depending on $\gamma$, $c_0$, $E[h_0]$ and the mean value of $h_0$, such that for all $0<\delta<\delta_0$ and all $0\leq t\leq T$ the solution $H$ satisfies
    \begin{equation*}
        \dist_{H_1}(H(t),\mathcal{M}) \leq 2\pi\sqrt{3} \dist_{H_1}(h_0,\mathcal{M}) \ee^{-\gamma c_0^3 t} + C\delta.
    \end{equation*}
\end{theorem}

\begin{proof}
The $H^1$-distance of a function to the manifold $\mathcal{M}$ is equivalent to the square root of the energy functional of this function. Indeed, they satisfy the relation
\begin{equation}\label{eq:energy-manifold-distance-equivalence}
    \frac{1}{\pi\sqrt{2}}\sqrt{E[H]} \leq \dist_{H^1}(H,\mathcal{M}) = \norm{P_{\geq 2}H}_{H^1} \leq \sqrt{6 E[H]}.
\end{equation}
The proof consequently relies on the dissipation of energy. While for $\delta>0$ the energy $E$ does not in general dissipate for all times, we will see that the error in the dissipation can be neglected as long as $\delta$ is small enough compared to $E[H]$. This changes as soon as the solution $H$ reaches a certain $\delta$-neighbourhood of $\mathcal{M}$ -- the dissipation and the error term introduced by the positive $\delta$ begin to balance and no further convergence to $\mathcal{M}$ can be expected. This idea is made rigorous by the following estimates:

Differentiating the energy leads to
\begin{equation*}
    \begin{split}
    \frac{\dd}{\dd t} E[H]
    &=
    -
    \gamma\, \int_{\Tbb^2} H^3 |\nabla(\Delta H + H)|^2 \dd\vartheta \dd \zeta
    +
    \delta \int_{\T^2} H^3 \cos(\vartheta + t) (\Delta H + H)_\vartheta \dd \vartheta \dd \zeta
    \\
    & 
    \leq
    -\gamma c_0^3 \|\nabla(\Delta H + H)\|_{L_2}^2
    +
    \delta \|H^3\|_{L_2} \|\nabla(\Delta H + H)\|_{L_2}
    \\
    &
    \leq
    -\gamma c_0^3 \|\nabla(\Delta H + H)\|_{L_2}^2
    +
    C \delta \|H\|_{H^1}^3 \|\nabla(\Delta H + H)\|_{L_2},
    \end{split}
\end{equation*}
where we used $\|H^3\|_{L_2} = \|H\|_{L_6}^{3} \leq C \|H\|_{H^1}^3$.
In order to absorb the term $\|\nabla(\Delta H + H)\|_{L_2}$ of the second summand into the first one, we apply Young's (weighted) inequality to obtain
\begin{equation*}
    C \delta \|H\|_{H^1}^3 \|\nabla(\Delta H + H)\|_{L_2}
    \leq
    \frac{C^2 \delta^2}{2 \gamma c_0^3}
    \|H\|_{H^1}^6
    +
    \frac{\gamma c_0^3}{2}
    \|\nabla(\Delta H + H)\|_{L_2}^2
\end{equation*}
and hence
\begin{equation}\label{eq:energy_derivative_estimate}
    \begin{split}
        \frac{\dd}{\dd t} E[H]
        & \leq
        -\frac{\gamma c_0^3}{2} \|\nabla(\Delta H + H)\|_{L_2}^2
        +
        \frac{C^2 \delta^2}{2 \gamma c_0^3}
        \|H\|_{H^1}^6.
    \end{split}
\end{equation}
The squared $L_2$-norm of $\nabla (\Delta H + H)$ can now be estimated from below by the energy functional:
\begin{equation}\label{eq:estimate-dissipation-against-energy}
    \norm{\nabla (\Delta H + H)}_{L_2}^2 \geq 4 E[H].
\end{equation}
Moreover, we can also control the $H^1$-norm of $H$ by means of the (conserved) zeroth Fourier coefficient $H_{0,0}$, the constant $c_0>0$ and the energy of $H$:
\begin{equation*}
\begin{aligned}
    \norm{H}_{H^1}^2 & \leq 4\pi^2\left(H_{0,0}^2 + 2 \left(\abs{H_{1,0}}^2 + \abs{H_{-1,0}}^2 + \abs{H_{0,1}}^2 + \abs{H_{0,-1}}^2\right)\right) + \norm{P_{\geq 2} H}_{H^1}^2\\
    &\leq 4\pi^2 \left(H_{0,0}^2 + 4 \abs{H_{1,0}}^2 + 4 \abs{H_{0,1}}^2\right) + 6 E[H]
\end{aligned}
\end{equation*}
and since $H(\vartheta,\zeta)\geq c_0$ we have
\begin{equation*}
\begin{aligned}
    \abs{H_{k,l}} &= \abslr{\fint_{\T^2} H(\vartheta,\zeta) \ee^{-\ii k\vartheta} \ee^{-\ii l\vartheta}\dd\vartheta\dd\zeta} = \abslr{\fint_{\T^2} (H(\vartheta,\zeta)-c_0) \ee^{-\ii k\vartheta} \ee^{-\ii l\vartheta}\dd\vartheta\dd\zeta}\\
    &\leq \fint_{\T^2}\abs{H(\vartheta,\zeta) - c_0} \dd\vartheta\dd\zeta 
    = H_{0,0} - c_0,
\end{aligned}
\end{equation*}
hence
\begin{equation}\label{eq:control-H1-norm-by-energy}
    \norm{H}_{H^1}^2 \leq 4\pi^2 \left(H_{0,0}^2 + 8 \big(H_{0,0} - c_0\big)^2\right) + 6 E[H].
\end{equation}
Using \eqref{eq:estimate-dissipation-against-energy} and \eqref{eq:control-H1-norm-by-energy} in \eqref{eq:energy_derivative_estimate}, we conclude the differential inequality
\begin{equation}\label{eq:differential_inequality}
    \frac{\dd}{\dd t} E[H]
    \leq
    -2\gamma c_0^3 E[H]
    +
    \frac{C^2 \delta^2}{2 \gamma c_0^3}
    \left(4\pi^2 \left(H_{0,0}^2 + 8\big(H_{0,0} - c_0\big)^2\right) + 6 E[H]\right)^3.
\end{equation}
Now, suppose $\delta>0$ satisfies
\begin{equation*}
    \delta^2 < \frac{4 \gamma^2 c_0^6 E[h_0]}{C^2\left(4\pi^2\left(H_{0,0}^2 + 8\big(H_{0,0}-c_0\big)^2\right) + 6 E[h_0]\right)^3}.
\end{equation*}
Then, $\frac{\dd}{\dd t}E[H](0)$ is negative and it follows that $E[H](t)\leq E[h_0]$ for all $t\leq T$. Plugging this into the nonlinear term of \eqref{eq:differential_inequality} and applying Gronwall's lemma we find
\begin{equation*}
    E[H(t)] \leq \ee^{-2\gamma c_0^3 t}E[h_0] + \frac{C^2\delta^2}{4\gamma^2 c_0^6}\left(4\pi^2\left(H_{0,0}^2 + 8\big(H_{0,0}-c_0\big)^2\right) + 6 E[h_0]\right)^3.
\end{equation*}
The result follows by taking the square root on both sides and applying estimate \eqref{eq:energy-manifold-distance-equivalence}.
\end{proof}


\subsection{Slow dynamics near the manifold $\Mcal(m)$} \label{sec:dynamics_close_to_slow_manifold}

In the light of Theorem \ref{thm:stability} and following remark \ref{rem:exp-conv-to-travelling-wave} we already know that for $\delta=0$ solutions to the original rimming-flow equation \eqref{eq:reform-PDE} converge exponentially to a \emph{travelling wave with speed one} (matching the rotation speed of the cylinder). Changing to the rotating coordinate system of this section, these travelling wave profiles become \emph{steady states} of the transformed equation \eqref{eq:PDE_transformed}. Thus, solutions to \eqref{eq:PDE_transformed} with $\delta=0$ converge exponentially to single points on $\Mcal$ provided their initial values satisfy appropriate conditions. We suspect that something similar is true in the case of $0<\delta\ll 1$: The exponential convergence holds only up to an error of order $\delta$ and the limit point on $\Mcal$ isn't actually fixed but describes on the time scale $\tau = \delta^2 t$ a slow smooth curve
\begin{equation*}
    H_0(\tau, \vartheta,\zeta) = H_0(\delta^2 t, \vartheta, \zeta) = H_{0,0} + a_{-1}(\delta^2 t) \ee^{-\ii\vartheta} + a_1(\delta^2 t) \ee^{\ii\vartheta} + b(\delta^2 t) \cos(\zeta) \in \Mcal
\end{equation*}
on the manifold $\Mcal$. We expect such dynamics to remain valid on a long time interval of scale $1/\delta^2$. In the case $\ell < \pi$ of a \emph{short cylinder}, the only critical eigenmodes are $\ee^{\pm\ii\vartheta}$. Thus, the problem reduces to the cross-sectional case with only coefficients $a_1$ and $a_{-1}$. See Section \ref{sec:numerics} for numerical simulations and \cite{jlv2023} for a rigorous result in this case. In the present 2D case however, a lot of the estimates from the 1D case break down due to worse Sobolev embeddings for terms appearing due to the quasilinear nature of the equation, and a rigorous justification of what follows remains open. We thus proceed to investigate the dynamics on a purely formal level. In particular, we derive a system of ordinary differential equations for the coefficients $a_{-1}, a_1 \in\C$, $b\in\R$ governing the curve $H_0$. The derivation of this system of ODEs follows the lines of the one-dimensional case in \cite{jlv2023}. We present it again for the convenience of the reader.

We start with the following ansatz:
\begin{equation*}
    H(t,\vartheta, \zeta) = H_0(\tau,\vartheta,\zeta) + \delta R_0(t,\vartheta,\zeta),
\end{equation*}
where $H_0$ is the curve on $\Mcal$ and $\delta R_0$ is some small error of order $\delta$. Inserting this ansatz into equation \eqref{eq:PDE_transformed}, we obtain
\begin{equation}\label{eq:PDE_with_ansatz}
    \delta^2H_{0,\tau}+\delta R_{0,t}+\delta\gamma\Acal\left((H_0+\delta R_0)^3\right)R_0=\delta\left((H_0+\delta R_0)^3\cos(\vartheta+t)\right)_\vartheta.
\end{equation}
Note that the term $\gamma\Acal\left((H_0+\delta R_0)^3\right)H_0$ vanishes since $H_0\in\Mcal$.

The approximation $H_0$ does only depend on the slow time scale $\tau=\delta^2 t$. This can be seen by inserting the more general ansatz $H(t,\vartheta,\zeta) = H_0(t,\vartheta,\zeta) + \delta R_0(t,\vartheta,\zeta)$ instead and obtaining to leading order
\begin{equation}\label{eq:orig-equation-delta0}
    H_{0,t}+\gamma\Acal\left(H_0^3\right)H_0=0,
\end{equation}
i.e. the original equation for $\delta=0$. $\Mcal$ then consists exactly of the steady states of \eqref{eq:orig-equation-delta0} and we know from Theorem \ref{thm:stability} that these steady states are exponentially stable. Thus, we do not expect interesting dynamics on $\Mcal$ in the fast time scale $t$. On the other hand, if we assumed the curve $H_0$ on $\Mcal$ depended only on the intermediate time scale $\tilde\tau = \delta t$ we would formally obtain $H_{0,\tilde\tau} = 0$ and hence we expect interesting dynamics to appear only on an even slower time scale.

We proceed by dividing equation \eqref{eq:PDE_with_ansatz} by $\delta$ to obtain the following equation for the error $R_0$, depending on the curve $H_0$ which is as of yet undetermined:
\begin{equation}\label{eq:error_equation_R0}
\begin{aligned}
   R_{0,t} + \gamma\Acal\left((H_0+\delta R_0)^3\right)R_0 =
   &- \delta H_{0,\tau} + \left(H_0^3\cos(\vartheta+t)\right)_\vartheta\\
   &+ \delta\left(\left(3H_0^2R_0 + 3\delta H_0R_0^2 + \delta^2R_0^3\right)\cos(\vartheta+t)\right)_\vartheta.
\end{aligned}
\end{equation}

\begin{remark}\label{rem:delta-squared-argument}
    In order to motivate the following steps of the derivation let us first consider the simplified linear error equation
    \begin{equation}\label{eq:simplified_error_equation}
        R_{t} + \gamma\Acal(b)R = \delta^\alpha F
    \end{equation}
    with inhomogeneity of order $\delta^\alpha$ and we assume that $b\geq b_0>0$ is bounded away from zero and $\norm{F}_{L_2}\leq C$. Using the same energy method from Theorem \ref{thm:convergence-to-manifold}, we can deduce
    \begin{equation}\label{eq:simplified_functional_estimate}
        \frac{\dd}{\dd t} E[R] \leq -\frac{b_0}{2}\norm{\nabla(\Delta R + R)}_{L_2}^2 + C \delta^{2\alpha}
    \end{equation}
    and hence
    \begin{equation*}
        \norm{P_{\geq 2}R(t)}_{H^1} \leq C(\ee^{-\omega t} + \delta^\alpha)
    \end{equation*}
    for some $\omega>0$.
    Given sufficient bounds on $b$ and $F$ (e.g. $\norm{F}_{H^2}\leq C$) this estimate can be improved to
    \begin{equation*}
        \norm{P_{\geq 2}R(t)}_{H^4} \leq C(\ee^{-\omega t} + \delta^\alpha)
    \end{equation*}
    by testing the $P_{\geq 2}$-projected version of equation \eqref{eq:simplified_error_equation} with $\Delta^4 P_{\geq 2}R_0$ and using the estimate for $\norm{P_{\geq 2}R}_{H^1}$. The main difficulty is thus to control the Fourier coefficients $R_{\pm 1, 0}$ and $R_{0,\pm 1}$. For this we can project with $P_1$ \eqref{eq:simplified_error_equation} onto these critical Fourier modes and use the fundamental theorem of calculus to estimate
    \begin{equation}\label{eq:simplified_eq_estimate}
    \begin{aligned}
        \norm{P_1 R(t)}_{H^4}
        &\leq \norm{P_1 R(0)}_{H^4} + \gamma\int_0^t \norm{P_1\Acal(b)R(s)}_{L_2} \dd s + C \delta^\alpha t\\
        &\leq \norm{P_1 R(0)}_{H^4} + \gamma\int_0^t \norm{P_{\geq 2}R(s)}_{H^4} \dd s + C \delta^\alpha t\\
        &\leq \norm{P_1 R(0)}_{H^4} + C(1 + \delta^\alpha t).
    \end{aligned}
    \end{equation}
    Now, the approximative solution $H_0$ in our original equation evolves on the time slow time scale $\tau=\delta^2 t$. This means we would like to bound the error $R$ on the time scale $\sim 1/\delta^2$. Considering the estimate \eqref{eq:simplified_eq_estimate}, this is only possible if $\alpha\geq 2$.
\end{remark}

Remark \ref{rem:delta-squared-argument} suggests that the right hand side of \eqref{eq:error_equation_R0} is not yet sufficiently small and we need to refine it further. For this purpose our strategy will consist of the following steps:

\begin{enumerate}
    \item[(1)] We consider the error equation \eqref{eq:error_equation_R0} to leading order in $\delta$ and find a particular globally bounded solution $H_1$ with help of the Fredholm alternative.
    \item[(2)] We refine our error $R_0$ by setting $R_0 = H_1 + R_1$ and insert this into the original error equation \eqref{eq:error_equation_R0}. This cancels the leading inhomogeneity term $\big(H_0^3 \cos(\vartheta+t)\big)_\vartheta$, yielding a better error equation for $R_1$.
    \item[(3)] We rewrite the equation obtained in step 2 and handle all remaining terms that are only of order $\delta$.
    \item[(4)] Choose the curve $H_0(\tau)\in\Mcal$ such that all secular terms are cancelled.
\end{enumerate}

{\noindent\bfseries Step 1: } The leading-order equation in $\delta$ of \eqref{eq:error_equation_R0} reads:
\begin{equation}\label{eq:error_equation_R0_leading_order}
   H_{1,t} + \gamma\Acal\left(H_0^3\right)H_1 =
   \left(H_0^3\cos(\vartheta+t)\right)_\vartheta.
\end{equation}
To find a globally bounded particular solution we use the ansatz
\begin{equation}\label{eq:ansatz-H1}
    H_1(t,\vartheta,\zeta) = G_1(\tau,\vartheta,\zeta) \,\ee^{\ii t} + \overline{G_1(\tau,\vartheta,\zeta)} \,\ee^{-\ii t}.
\end{equation}
Inserting this into \eqref{eq:error_equation_R0_leading_order}, we obtain for $G_1$:
\begin{equation}\label{eq:G1-equation}
    \ii \, G_1 + \gamma\,\Acal(H_0^3) G_1 = \dfrac{1}{2}\big(H_0^3\ee^{\ii\vartheta}\big)_\vartheta.
\end{equation}
By means of the Fredholm alternative, we show that this equation admits a unique solution $G_1$:

\begin{lemma}\label{lem:Fredholm-solvability}
    Let $\lambda\in\R\setminus\{0\}$ and let $b\geq b_0>0$ be a smooth function that is bounded away from zero. Then, $\lambda\ii + \Acal(b) : \Hev{3}(\T^2) \to \Hev{-1}(\T^2)$ is invertible. Moreover, if $f\in\Hev{k}(\T^2)$, then we have $(\lambda\ii + \Acal(b))^{-1}f \in \Hev{k+4}(\Tbb)$ for all $k\geq 0$.
\end{lemma}

Note that in equation \eqref{eq:G1-equation} $H_0(\tau)$ is by definition smooth for every $\tau$, and hence by Lemma \ref{lem:Fredholm-solvability} also $G_1(\tau)$ is smooth.

\begin{proof}
    We prove existence by means of the Fredholm alternative. We define the sesquilinear form $B:\Hevd{3}(\T^2)\times\Hevd{3}(\T^2) \to \C$ by
    \begin{equation*}
        B(u,v) \coloneq (\mu + \ii\lambda) \fint_{\T^2} \nabla u\cdot \nabla\bar v \dd\vartheta\dd\zeta + \fint_{\T^2} b\,\nabla\big(\Delta u + u\big)\cdot\nabla\Delta\bar v\dd\vartheta\dd\zeta.
    \end{equation*}
    $B$ is continuous and and for $\mu\geq\frac{B_0^2}{2b_0}$ it is also (uniformly) coercive:
    \begin{equation*}
        \Re B(u,u) \geq \mu \norm{\nabla u}_{L_2}^2 + b_0 \norm{\nabla\Delta u}_{L_2}^2 - B_0\norm{\nabla u}_{L_2} \norm{\nabla\Delta u}_{L_2} \geq \tfrac{b_0}{2} \norm{\nabla\Delta u}_{L_2}^2.
    \end{equation*}
    Thus, we may apply the Theorem of Lax--Milgram to conclude the existence of a Banach space isomorphism $\Phi:\Hevd{3}(\T^2)\to\Hevd{3}(\T^2)$ such that
    \begin{equation*}
        B(\Phi u, v) = \fint_{\T^2} \nabla\Delta u\cdot\nabla\Delta\bar v\dd\vartheta\dd\zeta
    \end{equation*}
    for all $u,v\in\Hevd{3}(\T^2)$. Now, we denote the inverse Laplacian on $\T^2$ by $\Delta^{-1}:\Hevd{-1}(\T^2)\to\Hevd{3}(\T^2)$ and define the compact operator $K\coloneq \Phi\circ \Delta^{-2} : \Hevd{-1}(\T^2)\to\Hevd{-1}(\T^2)$. In order to apply the Fredholm alternative, we show that the homogeneous equation $u - \mu K u = 0$ only admits the trivial solution $u = 0$. Indeed, for any such solution we have $u\in\Hevd{3}(\T^2)$ and
    \begin{equation*}
    \begin{aligned}
        0 &= B(u-\mu Ku,v) = B(u,v) - \mu B\big(\Phi \Delta^{-2} u,v\big) = B(u,v) - \mu \fint_{\T^2} \nabla u \cdot\nabla\bar v\dd\vartheta\dd\zeta\\
        &= \dprlr{\ii\lambda u + \Acal(b)u\,,\,\Delta\bar v}_{\dot H^{-1},\dot H^1}
    \end{aligned}
    \end{equation*}
    for every $v\in\Hevd{3}(\T^2)$. We set $v = u + \Delta^{-1} u \in\Hevd{3}(\T^2)$ and consider only the real part to obtain
    \begin{equation*}
        \int_{\T^2} b\, \abs{\nabla (\Delta u + u)}^2 \dd\vartheta\dd\zeta = 0,
    \end{equation*}
    whence $\nabla(\Delta u + u) = 0$. This, in turn, implies $\Acal(b)u = 0$ and thus $u = 0$.

    Then, the Fredholm alternative implies the existence of a unique solution $u\in\Hevd{-1}(\T^2)$ to $u - \mu Ku = -Kf$ for each $f\in\Hevd{-1}(\T^2).$ Again, it follows actually $u\in\Hevd{3}(\T^2)$ and
    \begin{equation*}
        \ii\lambda u + \Acal(b)u = f\quad\text{in }\Hevd{-1}(\T^2).
    \end{equation*}
\end{proof}

{\noindent\bfseries Step 2: }
We refine the error $R_0$ by setting $R_0 = H_1 + R_1$, i.e. we improve our original ansatz to $H = H_0 + \delta R_0 = H_0 + \delta H_1 + \delta R_1$. Note that the only fast-time scale dependence of $H_1$ is given by the exponentials $\ee^{\pm\ii t}$ in \eqref{eq:ansatz-H1}, so $H_1$ is bounded on a local time interval in $\tau$ and hence also on a time interval of order $1/\delta^2$ in $t$. Thus, it remains to ensure that the new error $R_1$ is bounded on such a time interval. To this end we insert $R_0 = H_1 + R_1$ into \eqref{eq:error_equation_R0} and obtain the following new error equation for $R_1$:
\begin{equation*}
\begin{aligned}
    R_{1,t} &+ \gamma\,\Acal\big((H_0 + \delta H_1 + \delta R_1)^3\big)R_1
    + 3\delta\gamma\,\Acal\big(H_0^2 R_1\big)H_1 - 3\delta\big(H_0^2 R_1\cos(\vartheta + t)\big)_\vartheta\\
    =& -\delta H_{0,\tau} - 3\delta \gamma\Acal\big(H_0^2H_1\big)H_1 + 3\delta\big(H_0^2 H_1\cos(\vartheta + t)\big)_\vartheta\\
    &- \delta^2\gamma\,\Acal\big(3 H_0 (H_1 + R_1)^2 + \delta (H_1 + R_1)^3\big)H_1\\
    &+\delta^2\big(\big(3 H_0 (H_1 + R_1)^2 + \delta (H_1 + R_1)^3\big)\cos(\vartheta + t)\big)_\vartheta\\
    &-\delta^2 G_{1,\tau}\ee^{\ii t} - \delta^2 \overline{G_{1,\tau}}\ee^{-\ii t}.
\end{aligned}
\end{equation*}

{\noindent\bfseries Step 3:}
We can now rewrite the second and third term on the right-hand side to
\begin{equation*}
    - 3\delta \gamma\Acal\big(H_0^2H_1\big)H_1 + 3\delta\big(H_0^2 H_1\cos(\vartheta + t)\big)_\vartheta = \delta U + \delta V \ee^{2\ii t} + \delta\overline{V}\ee^{-2\ii t},
\end{equation*}
where
\begin{equation} \label{eq:U_V}
\begin{split}
    U \coloneq & -3\gamma\Acal\big(H_0^2 G_1\big)\overline{G_1} - 3\gamma\Acal\big(H_0^2\overline{G_1}\big)G_1 + \tfrac{3}{2}\big(H_0^2 G_1 \ee^{-\ii\vartheta}\big)_\vartheta + \tfrac{3}{2}\big(H_0^2 \overline{G_1} \ee^{\ii\vartheta}\big)_\vartheta,\\
    V \coloneq & -3\gamma\Acal\big(H_0^2 G_1)G_1 + \tfrac{3}{2}\delta\big(H_0^2 G_1 \ee^{\ii\vartheta}\big)_\vartheta
\end{split}
\end{equation}
depend on the slow time scale $\tau = \delta^2 t$ only.
Hence, the refined error equation for $R_1$ can be written as follows:
\begin{equation}\label{eq:error_equation_R1}
\begin{aligned}
    R_{1,t} &+ \gamma\,\Acal\big((H_0 + \delta H_1 + \delta R_1)^3\big)R_1
    + 3\delta\gamma\,\Acal\big(H_0^2 R_1\big)H_1 - 3\delta\big(\big(H_0^2 R_1\big)\cos(\vartheta + t)\big)_\vartheta\\
    =& -\delta\Big[ H_{0,\tau} - P_1U\Big] + \delta P_{\geq 2}U + \delta \Big[V\ee^{2\ii t} + \overline{V}\ee^{-2\ii t}\Big] + \Ocal(\delta^2).
\end{aligned}
\end{equation}
The error equation formulated in this specific form allows us to point out a some important observations:
\begin{itemize}
    \item The additional linear terms of order $\delta$ in the first line of \eqref{eq:error_equation_R1} are problematic. We will deal with them later.
    \item The approximate solution $H_0$ is still undetermined. The first summand on the right-hand side of \eqref{eq:error_equation_R1} could thus be removed completely by choosing $H_0$ to satisfy the ODE system $H_{0,\tau} = P_1 U$.
    However, we do not do so just yet because we have to deal with some of the other problematic terms first.
    \item For the second summand on the right-hand side of \eqref{eq:error_equation_R1} $\delta P_{\geq 2} U$ the strategy will be as follows: Firstly, we find a $W(\tau)$ such that $\gamma P_{\geq 2} \Acal\big(H_0^3\big)W = P_{\geq 2} U$. Then, we refine our error similarly to before and set $R_1 = \delta W + R_2$. This will remove the term $\delta P_{\geq 2} U$ at the cost of producing a new term $\delta\gamma P_1\Acal\big(H_0^3\big)W$ which we can subsequently include in the ODE system.
    \item The third summand is problematic. We deal with it by repeating the procedure of steps 1 and 2 from above (cf. ansatz \eqref{eq:ansatz-H1} with only very slight modifications.
    \item All remaining terms are of order $\delta^2$ or higher and thus (at least on a formal level) do not pose problems for the approximation.
\end{itemize}

Hence, we start removing the term $\delta P_{\geq 2} U$. As explained above we wish to find a solution $W(\tau,\vartheta,\zeta)$ to
\begin{equation}\label{eq:W-equation}
    \gamma P_{\geq2}\Acal\big(H_0^3\big)W = P_{\geq 2} U.
\end{equation}
This is possible by the following lemma:
\begin{lemma}
    Let $b\geq b_0>0$ be smooth function that is bounded away from zero. Then, the operator $ P_{\geq 2}\Acal(b) : P_{\geq 2}\Hev{3}(\T^2;\R) \to P_{\geq 2}\Hev{-1}(\T^2;\R)$ is invertible. Moreover, if $f\in\Hev{k}(\T^2;\R)$, then we have $(P_{\geq 2}\Acal(b))^{-1}f \in P_{\geq 2}\Hev{k+4}(\Tbb;\R)$ for all $k\geq 0$.
\end{lemma}
\begin{proof}
    This is an immediate consequence of the Lax--Milgram theorem: We define the bilinear form $B:P_{\geq 2}\Hev{3}(\T^2;\R)\times P_{\geq 2}\Hev{3}(\T^2;\R)$ by
    \begin{equation*}
        B(u,v) \coloneq \int_{\T^2} b\,\nabla(\Delta u + u)\cdot\nabla(\Delta v + v) \dd\vartheta\dd\zeta.
    \end{equation*}
    Clearly, it is continuous and coercive and hence for every $f\in P_{\geq 2}\Hev{-1}(\T^2;\R)$ there exists a unique $u\in P_{\geq 2}\Hev{3}(\T^2;\R)$ such that
    \begin{equation*}
        B(u,v) = -\dprlr{f, \Delta v + v}_{H^{-1}, H^1}
    \end{equation*}
    holds for all $v\in P_{\geq 2}\Hev{3}(\T^2;\R)$. Thence, it follows $P_{\geq 2}\Acal(b) u = P_{\geq 2}f = f$.
\end{proof}
With $W(\tau,\vartheta,\zeta)$ as given by \eqref{eq:W-equation} we refine our error $R_1$ again by setting $R_1 = \delta W + R_2$. Inserting this into \eqref{eq:error_equation_R1}, we obtain the following improved error equation for $R_2$:
\begin{equation}\label{eq:error_equation_R2}
\begin{aligned}
    R_{2,t} &+ \gamma\,\Acal\big((H_0 + \delta H_1 + \delta R_2 + \delta^2 W)^3\big)R_2
    + 3\delta\gamma\,\Acal\big(H_0^2 R_2\big)H_1 - 3\delta\big(\big(H_0^2 R_2\big)\cos(\vartheta + t)\big)_\vartheta\\
    =& -\delta\Big[ H_{0,\tau} - P_1U + \gamma P_1 \Acal\big(H_0^3\big)W\Big] + \delta \Big[V\ee^{2\ii t} + \overline{V}\ee^{-2\ii t}\Big] + \Ocal(\delta^2).
\end{aligned}
\end{equation}
Here, we absorbed further terms of order $\delta^2$ into $\Ocal(\delta^2)$.

We proceed to deal with the term $\delta \big[V\ee^{2\ii t} + \overline{V}\ee^{-2\ii t}\big]$. The strategy is analogous to the steps 1 and 2 from above. We ignore all but the leading-order terms for now but we include the term $\delta \big[V\ee^{2\ii t} + \overline{V}\ee^{-2\ii t}\big]$ and find a particular globally bounded solution $\delta H_2$ (of order $\delta$) to
\begin{equation*}
    \delta H_{2,t} + \gamma \Acal\big(H_0^3\big)\delta H_2 = \delta \big[V\ee^{2\ii t} + \overline{V}\ee^{-2\ii t}\big],
\end{equation*}
or, equivalently, a globally bounded solution to
\begin{equation*}
    H_{2,t} + \gamma \Acal\big(H_0^3\big) H_2 = V\ee^{2\ii t} + \overline{V}\ee^{-2\ii t}.
\end{equation*}
We do so by using an ansatz similar to \eqref{eq:ansatz-H1}:
\begin{equation*}
    H_2(t,\vartheta,\zeta) = G_2(\tau,\vartheta,\zeta) \ee^{2\ii t} + \overline{G_2(\tau,\vartheta,\zeta)} \ee^{-2\ii t}.
\end{equation*}
Thus, we need to solve
\begin{equation}\label{eq:G2-equation}
    2\ii G_2 + \gamma \Acal\big(H_0^3\big)G_2 = V,
\end{equation}
which we already know to be possible from Lemma \ref{lem:Fredholm-solvability}. Hence, we refine the error $R_2$ once more by inserting $R_2 = \delta H_2 + R_3$ into the error equation \eqref{eq:error_equation_R2} for $R_2$ and consequently obtain:
\begin{equation}\label{eq:error_equation_R3}
\begin{aligned}
    R_{3,t} &+ \gamma\,\Acal\big((H_0 + \delta H_1 + \delta R_3 + \delta^2 H_2 + \delta^2 W)^3\big)R_3
    + 3\delta\gamma\,\Acal\big(H_0^2 R_3\big)H_1 - 3\delta\big(\big(H_0^2 R_3\big)\cos(\vartheta + t)\big)_\vartheta\\
    =& -\delta\Big[ H_{0,\tau} - P_1U + \gamma P_1 \Acal\big(H_0^3\big)W\Big] + \Ocal(\delta^2).
\end{aligned}
\end{equation}

{\noindent\bfseries Step 4:}
Finally, we can determine the curve $H_0(\tau)\in\Mcal$: It must obey the ODE system
\begin{equation}\label{eq:ODE-system}
    H_{0,\tau} = P_1\left[U - \gamma\Acal\big(H_0^3\big)W\right].
\end{equation}
Then, $R_3$ satisfies
\begin{equation}\label{eq:error_equation_R3_without_secular_terms}
    R_{3,t} + \gamma\,\Acal\big((H_0 + \delta H_1 + \delta R_3 + \delta^2 H_2 + \delta^2 W)^3\big)R_3 + \delta\Lcal(t) R_3
    = \Ocal(\delta^2),
\end{equation}
where we define the time-dependent linear operator $\Lcal(t) \coloneq \ee^{\ii t}\Lcal_1 + \ee^{-\ii t}\Lcal_{-1}$ with
\begin{equation*}
    \Lcal_1 u \coloneq 3\gamma\,\Acal\big(H_0^2 u\big)G_1 - \tfrac{3}{2}\big(\big(H_0^2 u\big)\ee^{\ii\vartheta}\big)_\vartheta
\end{equation*}
and $\Lcal_{-1} \coloneq\overline{\Lcal_1}$.

Now, in the formal scheme of Remark \ref{rem:delta-squared-argument} the term $\delta\Lcal(t)R_3$ does still pose a problem as it may in general lead to exponential growth of $R_3$ on the time scale $\delta t$. We point out that this would not be an issue if $\Lcal_1$ depended only on the Fourier coefficients $u_{k,l}$ with $k^2 + l^2 \geq 2$ because in that case, testing \eqref{eq:simplified_error_equation} in Remark \ref{rem:delta-squared-argument} with $\Delta R + R$ would just produce an additional $C\delta\norm{\nabla(\Delta R + R)}_{L_2}^2$-term in the estimate \eqref{eq:simplified_functional_estimate} which can be absorbed for $\delta$ is small enough. With this in mind we transform the error equation \eqref{eq:error_equation_R3} by a near-identity transformation $R \coloneq (\id + \delta\Lambda(t)) R_3$, where $\Lambda(t)$ is to be chosen in a way that the transformed error equation for $R$ takes on the desired form. This idea is inspired by normal-form transformations. Note that $\Lcal(t)$ in the original equation depends on $t$ and thus $\Lambda(t)$ must depend on $t$, too.

Hence, we formally apply $\delta\Lambda(t)$ to \eqref{eq:error_equation_R3_without_secular_terms} and obtain
\begin{equation*}
    \big(\delta \Lambda(t)R_3)_t + \gamma\Acal\big(H_0^3\big)\delta\Lambda(t) R_3 + \delta\gamma\Lambda(t)\Acal\big(H_0^3\big)R_3 = \delta\Lambda'(t)R_3 + \gamma\Acal\big(H_0^3\big)\delta\Lambda(t) R_3 + \Ocal(\delta^2).
\end{equation*}
We add this equation to equation \eqref{eq:error_equation_R3_without_secular_terms} and use the definition of $R$ to find
\begin{equation*}
\begin{aligned}
    R_t &+ \gamma\,\Acal\big((H_0 + \delta H_1 + \delta R_3 + \delta^2 H_2 + \delta^2 W)^3\big)R + \delta\gamma\Lambda(t)\Acal\big(H_0^3\big) R\\
    &= \delta\Lambda'(t)R_3 + \gamma\Acal\big(H_0^3\big)\delta\Lambda(t) R_3 - \delta\Lcal(t) R_3 + \Ocal(\delta^2).
\end{aligned}
\end{equation*}
Formally, the linear operator $\delta\Lambda(t)\Acal\big(H_0^3\big)R$ is now fine in the sense that it only depends on the Fourier coefficients with $k^2 + l^2\geq 2$, and the first Fourier modes are cancelled immediately by $\Acal\big(H_0^3\big)$. Thus, we choose the operator $\Lambda(t)$ to satisfy
\begin{equation*}
    \Lambda'(t) + \gamma\Acal\big(H_0^3\big)\Lambda(t) = \Lcal(t) = \ee^{\ii t}\Lcal_1 + \ee^{-\ii t}\Lcal_{-1}.
\end{equation*}
With the ansatz $\Lambda(t) = \ee^{\ii t}\Lambda_1(t) + \ee^{-\ii t}\Lambda_{-1}$ we obtain from there
\begin{equation*}
    \big(\ii + \gamma\Acal\big(H_0^3\big)\big)\Lambda_1 = \Lcal_1
\end{equation*}
and thus we define $\Lambda_1 \coloneq \big(\ii + \gamma\Acal\big(H_0^3\big)\big)^{-1}\Lcal_1$. Note that it follows from Lemma \ref{lem:Fredholm-solvability} that this is actually a compact operator, justifying the previous arguments. This leads finally to the following error equation for $R$:
\begin{equation}\label{eq:error_equation_R}
    R_t + \gamma\,\Acal\big((H_0 + \delta H_1 + \delta R_3 + \delta^2 H_2 + \delta^2 W)^3\big)R + \delta\gamma\Lambda(t)\Acal\big(H_0^3\big) R = \Ocal(\delta^2).
\end{equation}

\begin{remark}
    A local-in-time solution in $\tau$ to the ODE system \eqref{eq:ODE-system} exists by the theorem of Picard--Lindelöf.
    In order to prove rigorously that $H_0$ is close to the full solution on a time interval of order $1$ in $\tau$ (or order $1/\delta^2$ in $t$, respectively), we would have to show that the error $R$ satisfying the error equation \eqref{eq:error_equation_R} remains bounded on this time interval.
\end{remark}

Another important property of the abstract ODE system \eqref{eq:ODE-system} is its rotational invariance with respect to the cylinder axis. This is not immediately obvious: In fact, the full rimming-flow PDE is only rotationally symmetric in the case without any gravity ($\delta = 0$), because the gravitational acceleration chooses a certain direction. However, the system \eqref{eq:ODE-system} is only an approximation of order $\delta$ and thus it is plausible that it retains the symmetry of the case $\delta = 0$. This is the content of the following lemma:

\begin{lemma}\label{lem:ODE-rotational-invariance}
    The system of ODEs \eqref{eq:ODE-system} is rotationally symmetric around the cylinder axis, i.e. for any solution $\tau\mapsto(a_1(\tau), b(\tau))\in\C\times\R$ of \eqref{eq:ODE-system} and every $\phi\in[0,2\pi)$, $(\ee^{\ii\phi}a_1(\tau), b(\tau))$ solves \eqref{eq:ODE-system} as well.
\end{lemma}

\begin{proof}
    We define $\tilde a_1(\tau)\coloneq \ee^{\ii\phi}a_1(\tau)$ and
    \begin{equation*}
        \tilde H_0(\xi,\zeta) \coloneq H_0(\tilde a_1, b) = H_{0,0} + \tilde a_1 \ee^{\ii\xi} + \tilde a_{-1} \ee^{-\ii\xi} + b\cos(\zeta) = H_0(a_1, b)\Big\rvert_{\vartheta=\xi+\phi}.
    \end{equation*}
    Then, $\tilde G_1(\xi,\zeta) \coloneq \ee^{-\ii\phi} G_1(a_1,b)\big\vert_{\theta=\xi+\phi}$ solves the corresponding equation \eqref{eq:G1-equation}
    \begin{equation*}
        \ii\tilde G_1 + \gamma\Acal(\tilde H_0^3)\tilde G_1 = \ee^{-\ii\phi} \left(\ii G_1 + \gamma\Acal(H_0^3)G_1\right)\Big\vert_{\vartheta=\xi+\phi} = \frac{1}{2}\big(\tilde H_0^3 \ee^{\ii\xi}\big)_\xi,
    \end{equation*}
    whence $G_1(\tilde a_1, b) = \ee^{-\ii\phi} G_1(a_1, b)\big\vert_{\vartheta=\xi+\phi}$. We plug this into the definition of $U(a_1, b)$ in \eqref{eq:U_V} and obtain
    \begin{equation*}
    \begin{aligned}
        U(\tilde a_1, b)
        &= -3\gamma\Acal\big(\tilde H_0^2 G_1(\tilde a_1, b)\big)\overline{G_1(\tilde a_1, b)} + \tfrac{3}{2}\big(\tilde H_0^2 G_1(\tilde a_1, b) \ee^{-\ii\xi}\big)_\xi + \mathrm{c.c.}\\
        &= -3\gamma\Acal\big(H_0^2 G_1(a_1, b)\big)\overline{G_1(a_1, b)} + \tfrac{3}{2}\big(H_0^2 G_1(a_1, b) \ee^{-\ii\vartheta}\big)_\vartheta\Big\vert_{\vartheta = \xi+\phi} + \mathrm{c.c.}\\
        &= U(a_1,b)\Big\vert_{\vartheta = \xi+\phi},
    \end{aligned}
    \end{equation*}
    which in turn implies $W(\tilde a_1, b) = W(a_1, b)\big\vert_{\vartheta = \xi + \phi}$ and hence \eqref{eq:ODE-system} becomes
    \begin{equation*}
    \begin{aligned}
        \tilde a_1' \ee^{\ii\xi} + \tilde a_{-1}' \ee^{-\ii\xi} + b'\cos(\zeta)
        &= P_1\left[U(\tilde a_1, b) - \gamma\Acal(\tilde H_0^3)W(\tilde a_1, b)\right]\\
        &= P_1\left[U(a_1, b) - \gamma \Acal\big(H_0^3\big)W(a_1, b)\right]\Big\vert_{\vartheta = \xi + \phi}
    \end{aligned}
    \end{equation*}
    Using the definition of $\tilde a_1$ and substituting $\xi = \vartheta-\phi$, we find that this is equivalent to the original ODE system \eqref{eq:ODE-system} for $a_1$ which is satisfied by assumption.
\end{proof}


\subsection{Linearisation of the ODE system \eqref{eq:ODE-system}}\label{sec:linearise-ode}
We compute the linearisation of \eqref{eq:ODE-system} near the trivial constant solution $H_0(\tau) = H_{0,0}$. To this end, we set for $\eps\ll 1$
\begin{equation*}
    H_0(\tau) = H_{0,0} + \eps a_1(\tau) \ee^{\ii \vartheta} + \eps a_{-1}(\tau) \ee^{-\ii \vartheta} + \eps b(\tau) \cos(\zeta)
\end{equation*}
and $G_1 = g_0 + \eps g_1 + \Ocal(\eps^2)$ where $G_1$ is the unique solution of equation \eqref{eq:G1-equation}:
\begin{equation*}
    \ii \, G_1 + \gamma\,\Acal(H_0^3) G_1 = \dfrac{1}{2}\big(H_0^3\ee^{\ii\vartheta}\big)_\vartheta.
\end{equation*}
Then, to leading order we immediately find $g_0 = \tfrac{H_{0,0}^3}{2} \ee^{\ii \vartheta}$
and hence to order $\eps$ we obtain
\begin{equation*}
    \ii g_1 + \gamma H_{0,0}^3 \Delta (\Delta g_1 + g_1)
    =
    3 \ii H_{0,0}^2 a_1 \ee^{2 \ii \vartheta}
    +
    \frac{3}{2} \ii H_{0,0}^2 b \cos(\zeta) \ee^{\ii \vartheta}.
\end{equation*}
From Lemma \ref{lem:Fredholm-solvability} it thus follows that
\begin{equation*}
    G_1(\tau,\vartheta,\zeta) = \frac{H_{0,0}^3}{2}\ee^{\ii\vartheta} + \frac{3 \eps H_{0,0}^2 a_1 (1 + 12 \gamma H_{0,0}^3 \ii)}{1 + 144 \gamma^2 H_{0,0}^6} \ee^{2 \ii \vartheta}
    +
    \frac{\frac{3}{2}\eps H_{0,0}^2 b (1 + 2 \gamma H_{0,0}^3 \ii)}{1 + 4 \gamma^2 H_{0,0}^6} \cos(\zeta) \ee^{\ii \vartheta} + \Ocal(\eps^2).
\end{equation*}
Thence, we can insert this approximate solution for $G_1$ into the definition of $U$ (cf. \eqref{eq:U_V}), namely
\begin{equation*}
    U \coloneq
    -3\gamma\Acal\big(H_0^2 G_1\big)\overline{G_1} - 3\gamma\Acal\big(H_0^2\overline{G_1}\big)G_1 + \tfrac{3}{2}\big(H_0^2 G_1 \ee^{-\ii\vartheta}\big)_\vartheta + \tfrac{3}{2}\big(H_0^2 \overline{G_1} \ee^{\ii\vartheta}\big)_\vartheta,
\end{equation*}
to obtain
\begin{equation*}
\begin{split}
    U 
    =
    &-
    \frac{27 \eps H_{0,0}^7 \gamma a_{-1} (1 - 12 \gamma H_{0,0}^3 \ii)}{1 + 144 \gamma^2 H_{0,0}^6} \ee^{-\ii \vartheta}
    -
    \frac{27 \eps H_{0,0}^7 \gamma a_{1} (1 + 12 \gamma H_{0,0}^3 \ii)}{1 + 144 \gamma^2 H_{0,0}^6} \ee^{\ii \vartheta}
    -
    \frac{9 \eps H_{0,0}^7 \gamma b}{2 + 8 \gamma^2 H_{0,0}^6} \cos(\zeta)
    \\
    &-
    \frac{9 \eps H_{0,0}^4 a_1 ( 12 \gamma H_{0,0}^3 - \ii)}{2(1 + 144 \gamma^2 H_{0,0}^6)} \ee^{\ii \vartheta}
    -
    \frac{9 \eps H_{0,0}^4 a_{-1} ( 12 \gamma H_{0,0}^3 + \ii)}{2(1 + 144 \gamma^2 H_{0,0}^6)} \ee^{-\ii \vartheta}
    +
    3 \eps \ii H_{0,0}^4 a_1 \ee^{\ii \vartheta}
    -
    3 \eps \ii H_{0,0}^4 a_{-1} \ee^{-\ii \vartheta}\\
    &+\Ocal(\eps^2).
\end{split}
\end{equation*}
Moreover, we have
\begin{equation*}
    \gamma P_{\geq 2} \Acal(H_0^3) W
    =
    P_{\geq 2} U
    =
    \Ocal(\eps^2)
\end{equation*}
and hence $W = \Ocal(\eps^2)$. Thus, for the system of ODEs for the coefficients $a_1\in\C$, $a_{-1} = \overline{a_1}$ and $b\in\R$ we obtain the following linearisation about the trivial steady state $a_1 = a_{-1} = 0,\ b=0$:
\begin{equation}\label{eq:linearised-ode}
\left\{
\begin{aligned}
    a_1'(\tau)
    &=
    -
    \frac{27 H_{0,0}^7 \gamma a_{1} (1 + 12 \gamma H_{0,0}^3 \ii)}{1 + 144 \gamma^2 H_{0,0}^6} 
    -
    \frac{9 H_{0,0}^4 a_1 ( 12 \gamma H_{0,0}^3 - \ii)}{2(1 + 144 \gamma^2 H_{0,0}^6)} 
    +
    3 \ii H_{0,0}^4 a_1 
    +
    \Ocal\big(\eps\big)
    \\
    b'(\tau) 
    &=
    -
    \frac{9 H_{0,0}^7 \gamma b}{2 + 8 \gamma^2 H_{0,0}^6} + \Ocal\big(\eps\big).
\end{aligned}
\right.
\end{equation}
We may rewrite $a_1(\tau) = \rho(\tau) \ee^{\ii\phi(\tau)}$ in polar coordinates, where $\rho \geq 0$ and $\phi$ are real-valued functions. Inserting this into the ODE \eqref{eq:linearised-ode}, we obtain
\begin{equation*}
\left\{
\begin{aligned}
    \rho'(\tau)
    &=
    -
    \frac{81\gamma H_{0,0}^7}{1 + 144 \gamma^2 H_{0,0}^6}
    \rho(\tau)
    +
    \Ocal\big(\eps\big)
    \\
    \phi'(\tau)
    &=
    \frac{108 \gamma^2 H_{0,0}^{10}}{1 + 144 \gamma^2 H_{0,0}^6}
    +
    \frac{15 H_{0,0}^4}{2(1 + 144 \gamma^2 H_{0,0}^6)}
    +
    \Ocal\big(\eps\big)
    \\
    b'(\tau) 
    &=
    -
    \frac{9 H_{0,0}^7 \gamma b}{2 + 8 \gamma^2 H_{0,0}^6} + \Ocal\big(\eps\big).
\end{aligned}
\right.
\end{equation*}
This linearisation is actually completely decoupled. In particular, we recover the eigenvalue
\begin{equation*}
    \lambda = -\ii + \delta^2 \left(-\frac{81\gamma m^7}{1 + 144\gamma^2 m^6} + \ii\frac{108\gamma^2 m^{10} + \tfrac{15}{2}m^4}{1 + 144\gamma^2 m^6}\right),
\end{equation*}
of which we have already computed the complex conjugate in the proof of Theorem \ref{thm:exp-stability-perturbed-steady-state}, even though that computation was only valid for the short cylinder case $\ell < \pi$ due to the non-uniqueness of solutions in the case critical case $\ell = \pi$. Moreover, we observe that transversal perturbations decay exponentially to zero. In view of Remark \ref{rem:exp-conv-to-travelling-wave} such solutions will converge on the slow time scale $\tau = \delta^2 t$ to a cylindrical shape. In the rotating coordinates $(\theta-t,\zeta)$ its axis spirals in clockwise direction towards the centre, although from the outside the cylinder still rotates in a counter-clockwise direction very quickly.

\subsection{Numerical simulations for the the ODE system \eqref{eq:ODE-system}}\label{sec:numerics}
A particular strength of the ODE system \eqref{eq:ODE-system} compared, for instance, to the stability result in Theorem \ref{thm:exp-stability-perturbed-steady-state}, is the fact that it remains valid even for solutions which are not necessarily close to the constant solution. It is thus natural to ask if the linear behaviour described above is a good description of the global behaviour. We approach this question by some numerical simulations. Note that, while the original rimming-flow equation is highly \emph{nonlinear}, for the simulation of the ODE system we only have to solve two \emph{linear} PDEs although the complete ODE system is still nonlinear. The solution to the linear PDEs \eqref{eq:G1-equation} and \eqref{eq:W-equation} is computed using the finite element library deal.II \cite{2023:arndt.bangerth.ea:deal} and the direct UMFPACK solver from \cite{davis_umfpack}.

Since $b(0) = 0$ implies that $b(\tau) = 0$ for all $\tau>0$, the problem reduces to a two-dimensional system in this case. In fact, the ODE derived in \cite{jlv2023} for a cross-section rimming-flow and solutions in the short cylinder case $\ell < \pi$ (where transversal perturbations decay exponentially on the fast time scale $t$) are completely described by the phase portrait in Figure \ref{fig:phase-portrait-b-zero}. Due to the rotational invariance of the ODE system from Lemma \ref{lem:ODE-rotational-invariance} the phase portrait is symmetric, too.
\begin{figure}[H]
    \centering
    \includegraphics[width=0.5\linewidth]{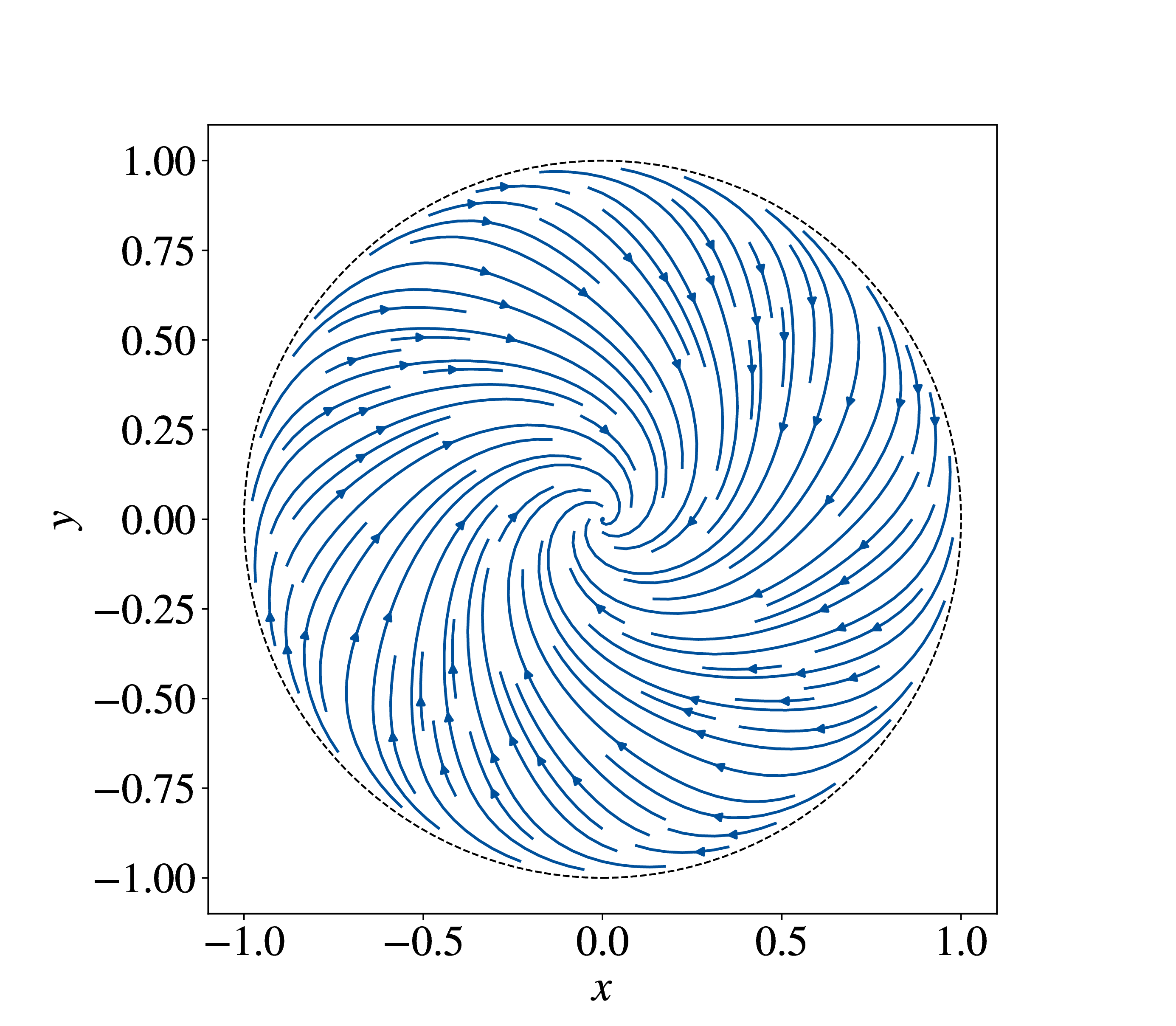}
    \caption{Phase portrait of the ODE system \eqref{eq:ODE-system} for $m = \gamma = 1$ when $H_0(\tau,\vartheta) = m - x(\tau)\cos(\vartheta) - y(\tau)\sin(\vartheta)$ is independent of $\zeta$: The fluid forms a circle in the cross-section of the cylinder. Its centre $(x(\tau),y(\tau))$ spirals inwards to the cylinder axis.}
    \label{fig:phase-portrait-b-zero}
\end{figure}

In Figure \ref{fig:snapshots_2d} we show the state of a cylindrical fluid profile at three distinct increasing times while its centre axis spirals inwards according to the phase portrait in Figure \ref{fig:phase-portrait-b-zero}.
\begin{figure}[H]
    \centering
    \begin{subfigure}[b]{0.32\textwidth}
        \centering
        \includegraphics[width=\textwidth]{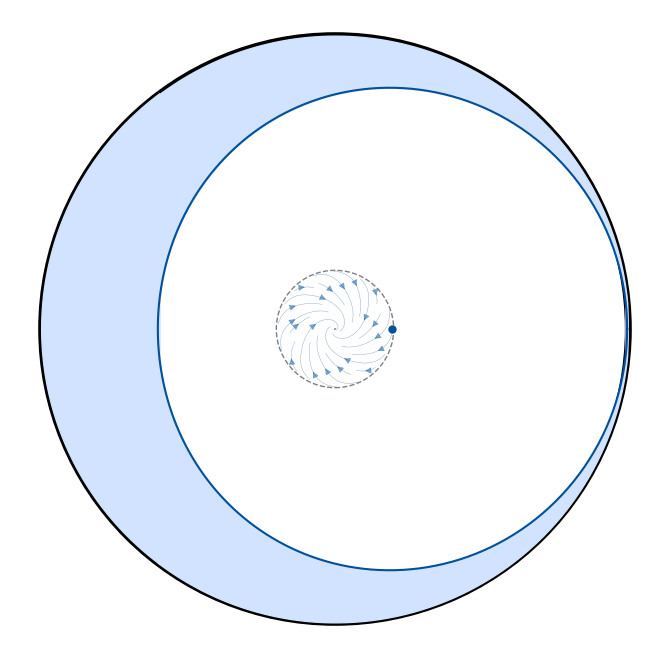}
    \end{subfigure}
    \begin{subfigure}[b]{0.32\textwidth}
        \centering
        \includegraphics[width=\textwidth]{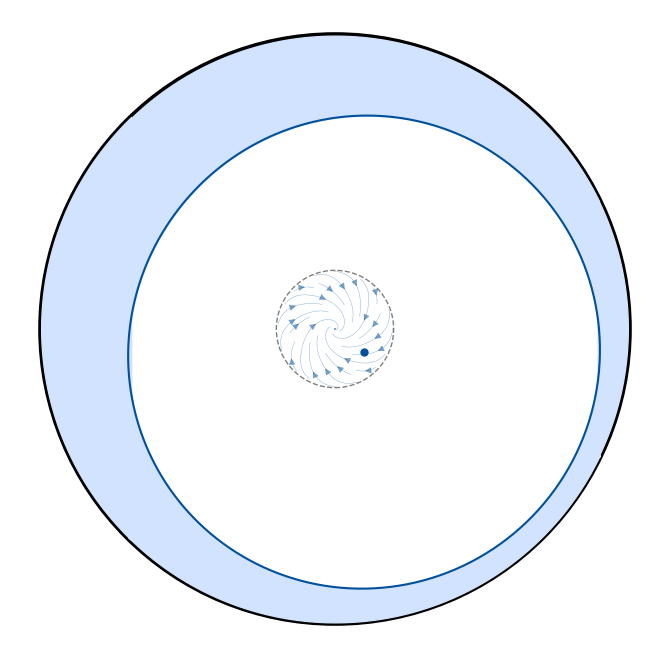}
    \end{subfigure}
    \begin{subfigure}[b]{0.32\textwidth}
        \centering
        \includegraphics[width=\textwidth]{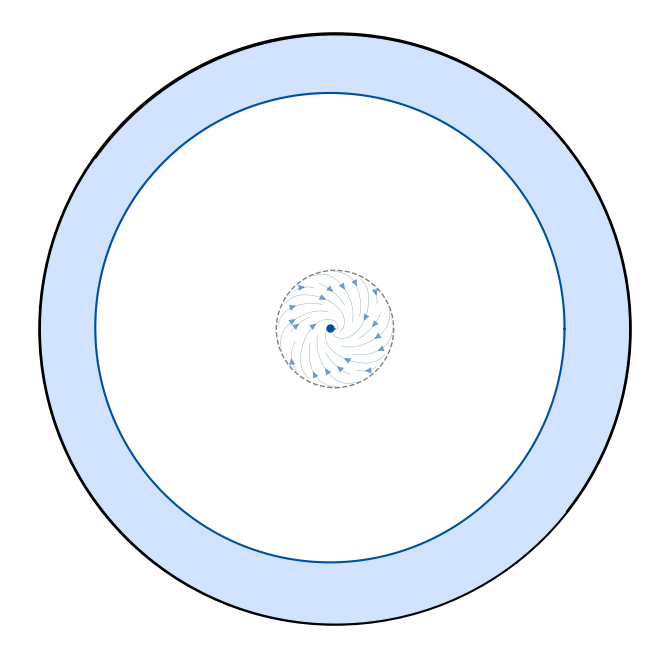}
    \end{subfigure}
    \caption{The circular shape of the fluid surface spirals towards the centre of the cylinder on the time scale $\tau = \delta^2 t$ according to the phase portrait of the system \eqref{eq:ODE-system}. Additionally, the whole cylinder rotates with angular velocity one (not shown in this figure).}
    \label{fig:snapshots_2d}
\end{figure}

If, on the other hand, $b\neq 0$, the dynamics are truly three-dimensional. However, since the system is still rotationally invariant in $a_1$ by Lemma \ref{lem:ODE-rotational-invariance}, we can plot a two-dimensional phase portrait of $b$ against amount by which the fluid profile is shifted from the cylinder axis $\abs{a_1}$, see Figure \ref{fig:phase-portrait-abs-a1}. In particular, we observe that neither $\abs{a_1(\tau)}$ nor $\abs{b(\tau)}$ must decrease monotonically. Instead, for certain solutions an uneven distribution of mass along the $\zeta$-axis can temporarily cause an increase of the displacement from the cylinder axis and vice versa.
\begin{figure}[H]
    \centering
    \includegraphics[width=0.65\linewidth]{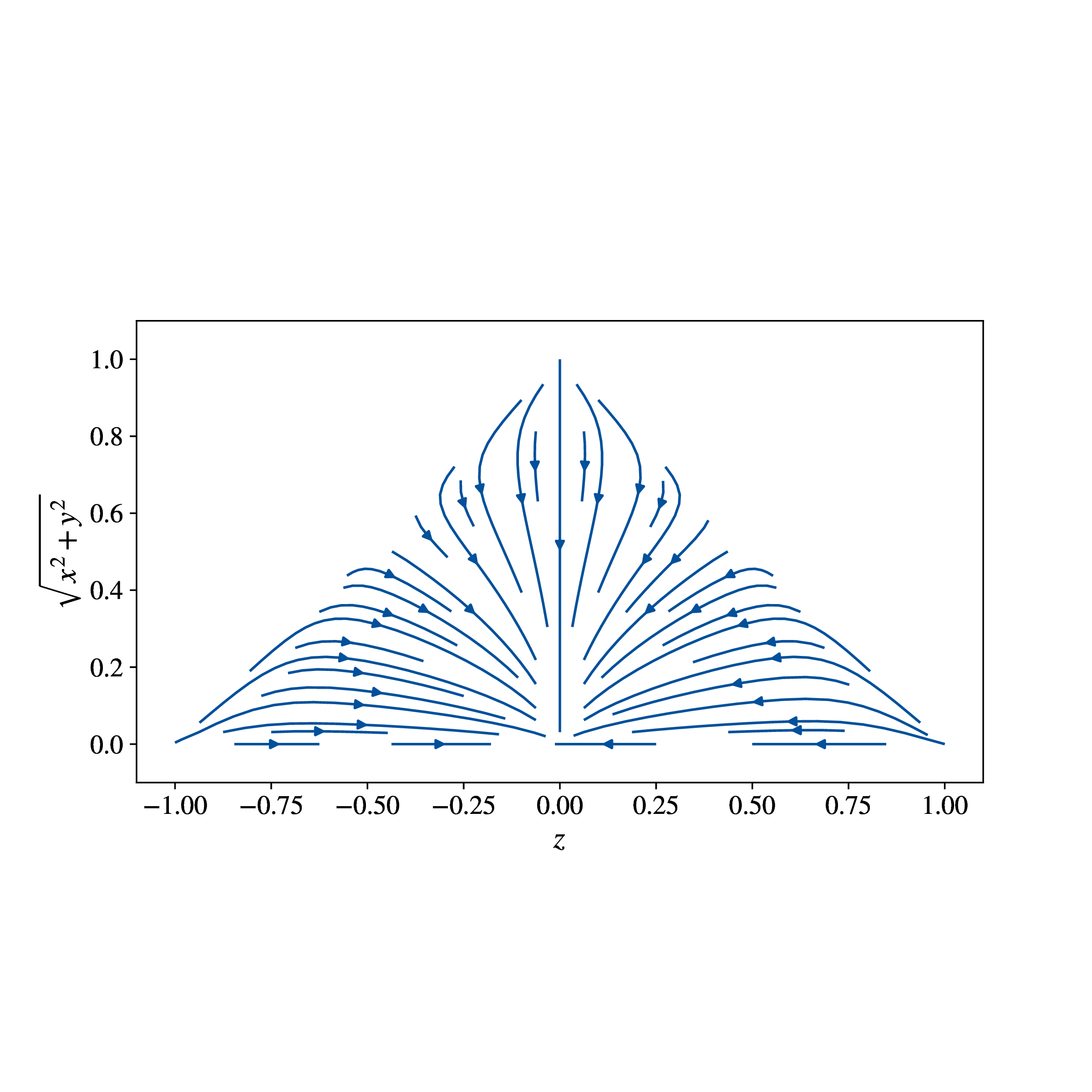}
    \caption{Dynamics of $b(\tau)$ and $\abs{a_1(\tau)}$ in the case $m=\gamma=1$. The ODE is only valid when $H_0(\tau)$ is uniformly positive, that is if $\abs{b(\tau)} + \abs{a_1(\tau)} < 1$.}
    \label{fig:phase-portrait-abs-a1}
\end{figure}


\section*{Statements and Declarations}

\noindent\textsc{Conflict of Interest. }
The authors have no competing interests to declare that are
relevant to the content of this article.
\medskip

\noindent\textsc{Data Availability. }
The source code for the numerical simulations can be provided by the authors upon request.
\medskip

\noindent\textsc{Funding. }
The project was supported by Deutsche Forschungsgemeinschaft (DFG, German Research Foundation) -- project number 545145736 -- through the scientific network \emph{(In)stability Phenomena in Asymptotic Models in Fluid Dynamics}.
\noindent 
J. J. L. Velázquez gratefully acknowledges the support by the Deutsche Forschungsgemeinschaft (DFG)
through the collaborative research centre “Analysis of Criticality: from Complex Phenomena to Models and Estimates” of the University of Bonn (CRC 1720, Project-ID
539309657) and the DFG under Germany’s Excellence Strategy - EXC2047/2-390685813.
J. Joussen is funded by Deutsche Forschungsgemeinschaft (DFG,
German Research Foundation) under Germany´s
Excellence Strategy – EXC 2075 – 390740016.

\printbibliography

\end{document}